\documentclass[a4paper,11pt]{siamart190516}
\usepackage{LatexDefinitions}

\newcommand{\HK}{\tn{HK}}
\newcommand{\W}{\tn{W}}
\newcommand{\CE}{\mc{CE}}
\newcommand{\CES}{\mc{CES}}

\newcommand{\Cos}{\ol{\cos}}
\newcommand{\cHK}{c}

\newcommand{\JHKSM}{J_{\tn{SM}}}
\newcommand{\proj}{\tn{P}}
\newcommand{\cont}{\mathrm{C}}
\newcommand{\Lebesgue}{\mc{L}}

\newcommand{\Log}{\tn{Log}}
\newcommand{\Exp}{\tn{Exp}}
\newcommand{\WLin}{{\W_{\tn{lin}}}}

\newcommand{\PiOpt}{\Pi_{\mathrm{opt}}}

\newcommand{\HKLin}{{\HK_{\mathrm{lin}}}}

\newcommand{\measL}{\meas_{\Lebesgue}}
\newcommand{\measpL}{\meas_{+,\Lebesgue}}
\newcommand{\probL}{\meas_{1,\Lebesgue}}

\usepackage[normalem]{ulem}
\usepackage{multirow}
\usepackage{rotating}
\usepackage{geometry}
\definecolor{darkblue}{rgb}{0.1,0.1,0.6}
\def\l{\left(}
\def\r{\right)}

\def\lda{\left\|}
\def\rda{\right\|}

\def\Lp#1{\textrm{\textnormal{L}}^{#1}}
\DeclareMathOperator{\re}{re}
\DeclareMathOperator{\im}{im}

\title{The Linearized Hellinger--Kantorovich Distance}
\author{Tianji Cai\footnote{Department of Physics, University of California, Santa Barbara, US}, Junyi Cheng\footnotemark[1], Bernhard Schmitzer\footnote{Campus Institute Data Science, Universit\"at G\"ottingen, G\"ottingen, Germany}, Matthew Thorpe\footnote{Department of Mathematics, University of Manchester, Manchester, UK}}
\date{\today}

\begin{document}
\maketitle
\begin{abstract}
In this paper we study the local linearization of the Hellinger--Kantorovich distance via its Riemannian structure. We give explicit expressions for the logarithmic and exponential map and identify a suitable notion of a Riemannian inner product. Samples can thus be represented as vectors in the tangent space of a suitable reference measure where the norm locally approximates the original metric. Working with the local linearization and the corresponding embeddings allows for the advantages of the Euclidean setting, such as faster computations and a plethora of data analysis tools, whilst still enjoying approximately the descriptive power of the Hellinger--Kantorovich metric.
\end{abstract}




\section{Introduction}
\subsection{Overview}
Optimal transport or Wasserstein distances, due to their Lagrangian properties, often better capture variations in signals and images than `pointwise' similarity measures such as Euclidean $\Lp{p}$-norms or the Kullback--Leibler divergence.
Wasserstein distances, in particular the Wasserstein-2 ($\W_2$) distance, come with a rich geometric structure, such as geodesics, barycenters \cite{WassersteinBarycenter} and a weak Riemannian structure (see Section \ref{sec:W2Riemann} for more details).
A well developed theory, see for example~\cite{Villani-OTOAN2008,Villani-TOT2003,Santambrogio-OTAM,AmbrosioGradientFlows2005}, and success in a range of applications, for example~\cite{Kolouri:2017} for an overview, have made the distance a valuable tool in the analysis of images and signals.
However, there are two major drawbacks of the Wasserstein distance.
The first is, despite recent numerical advances (see \cite{PeyreCuturiCompOT} for an overview), it is still relatively expensive to compute the Wasserstein distance, particularly in large and high dimensional data sets.
The second is that Wasserstein distances are only meaningful for measures of equal mass, and more importantly, the mass distributions of the compared measures must be matched exactly. This makes the optimal matching susceptible to noise in the form of small, but non-local mass fluctuations. Consequently, small perturbations can suppress more relevant features.

\subsection{Linearized optimal transport}
In~\cite{OptimalTransportTangent2012} it was proposed to linearize the $\W_2$ distance by exploiting its (weak) Riemannian structure.
Formally, the manifold is approximated locally by a tangent space at a reference point.
By the logarithmic map samples are embedded into the tangent space which is a Hilbert space with an inner product (see Section \ref{sec:W2Lin} for more details).
The norm of each embedded sample equals its $\W_2$ distance to the reference measure. And the (Hilbertian) distance between two embedded samples approximately equals the $\W_2$ distance between the two original samples.
One advantage of this strategy is a reduction of computational cost: to embed $n$ samples into the tangent space, one needs to solve $n$ optimal transport problems, whereas computing the full pairwise metric matrix requires to solve $O(n^2)$ problems.
Another advantage is the linear structure on the space of embeddings, compared to the non-linear metric structure of the original Wasserstein space, that allows for the application of methods from data analysis, for example principle component analysis, that have been primarily developed for linear settings.
But since the obtained embedding is (at least locally) approximating the $\W_2$ distance we expect the tangent linear structure to still enjoy many of the properties of the $\W_2$ distance such as the cost of translations, as opposed to the naive linear structure on the space of measures or densities.

For example,~\cite{park18} used the linear Wasserstein setting to generate new images by first clustering in the linear Wasserstein space, then learning the principal directions in the tangent planes for each cluster.
Hence, new images were generated using Euclidean data analysis techniques, such as $k$-means and principal component analysis (PCA), whilst keeping the Wasserstein flavor.
Other applications of the linear Wasserstein space have included a PCA based approach for super resolution on faces~\cite{kolouri15}, and classification, using a Fisher linear discriminant analysis technique, on images of nuclei~\cite{ozolek14}.

\subsection{Hellinger--Kantorovich distance}
As regards the second disadvantage of the Wasserstein distances raised above, the recently proposed Hellinger--Kantorovich (HK) distance~\cite{LieroMielkeSavare-HellingerKantorovich-2015a,ChizatDynamicStatic2018,ChizatOTFR2015,KMV-OTFisherRao-2015} defines a metric between non-negative measures by allowing for the creation/destruction of mass within the optimal transport framework.
In particular, the HK distance allows mass to be transported, created and destroyed, whereas the Wasserstein distances only allowed for mass to be transported.

One way to define the HK distance is via a modification of the famous Benamou--Brenier functional for the $\W_2$ distance \cite{BenamouBrenier2000} where one adds an additional source term to the continuity equation and a corresponding penalty in the cost functional by the Riemannian tensor of the Hellinger--Kakutani or Fisher--Rao metrics (see Section \ref{sec:HKBB} for details). This implies that HK also enjoys a weak Riemannian structure which was one of the motivations for its introduction, see \cite{LieroMielkeSavare-HK2-2016}.
As of now this structure is still much less understood than the $\W_2$ case.
Equivalently, HK can be computed via a particular `soft-marginal' Kantorovich-type transport problem \cite{LieroMielkeSavare-HellingerKantorovich-2015a} (see Section \ref{sec:HKKantorovich} for details) which lends itself to numerical approximation via entropic regularization \cite{ChizatEntropicNumeric2018}.

HK is a particular variant of an `unbalanced' transport problem and many other models to combine transport and creation/destruction are conceivable, see for instance \cite{DNSTransportDistances09,Caffarelli-McCann-FreeBoundariesOT-2010,PiccoliRossi-GeneralizedWasserstein2014,SchmitzerWirthUnbalancedW1-2017}. Some discussion is also provided in \cite{ChizatOTFR2015}.

\subsection{Contribution and outline}

The main contribution of this paper is the adaptation of the linearization strategy for the $\W_2$ distance \cite{OptimalTransportTangent2012} to the HK distance. We introduce logarithmic and exponential maps for HK, discuss how to approximate them computationally and demonstrate their usefulness for transport-based data analysis in numerical examples.
\vspace{\baselineskip}

The paper is organized as follows.
In the remainder of the introduction, for convenience and clarity, we define our notation.

In Section~\ref{sec:W2} we recap the $\W_2$ distance, with an emphasis on the Benamou--Brenier formulation~\cite{BenamouBrenier2000} (Section \ref{sec:W2BB}). Then, via the Kantorovich formulation (Section~\ref{sec:W2Kant}), we recall how to construct geodesics from optimal Kantorovich transport plans (Section~\ref{sec:W2Geo}).
A discussion of the Riemannian structure (Section~\ref{sec:W2Riemann}) leads us to recall the definition of the linearized Wasserstein distance from~\cite{OptimalTransportTangent2012} (Section~\ref{sec:W2Lin}).

Section~\ref{sec:HK} follows Sections~\ref{sec:W2BB}-\ref{sec:W2Geo} but now in the setting of the HK distance.
In particular, we start with defining the HK distance in the Benamou--Brenier setting before recalling a Kantorovich formulation, and finally construct geodesics from optimal Kantorovich transport plans.
While this construction (Proposition \ref{prop:HKGeodesicsSoftMarginal}) is implicitly already contained in \cite{LieroMielkeSavare-HellingerKantorovich-2015a} (and to a lesser extent in \cite{ChizatDynamicStatic2018}) and presumably known to the experts in the field, we are not aware of an explicit statement in the literature and believe that it is of practical relevance to those that want to familiarize themselves with the HK metric.

The linearization of HK is described in Section \ref{sec:HKLin}.
In Section~\ref{sec:HKRiemann} we gain some intuition for its Riemannian structure and identify a candidate for the logarithmic map, given as an explicit function from an optimal Kantorovich-type transport plan.
While for $\W_2$ a tangent vector is represented by a velocity field, as expected for HK one obtains an additional scalar mass creation/destruction field. This field can become singular in the case where mass is created from nothing, leading to a third, measure-valued, tangent component.
This third component may be considered undesirable in some applications and we discuss sufficient conditions to ensure that it remains zero (Remark \ref{rem:HKEuc}) so that one obtains an embedding into a Hilbert space.
The definition of logarithmic map and local linearization are formalized in Section \ref{sec:HKLog} where we also show that the logarithmic map behaves as expected along geodesics. Finally, in Section \ref{sec:HKExp} we give an explicit form of the corresponding exponential map.

While the formulas may seem considerably more complex compared to the Wasserstein case, from a numerical perspective the local linearization via HK is not much harder to perform than for $\W_2$. The involved transport problems can be solved in a Kantorovich-type formulation, for instance with an adapted Sinkhorn algorithm, and an approximate logarithmic map can be extracted from the optimal coupling with explicit formulas, leading to an embedding into a Hilbert space (see again Remark \ref{rem:HKEuc}, as above), which becomes finite-dimensional after discretization (similar numerical approximations apply, see Section \ref{sec:W2Lin} and Remark \ref{rem:HKBarycentric}).
The remaining challenge is to fix the intrinsic length-scale of the HK metric appropriately (see Remark \ref{rem:HKLengthScale}). This single real-valued parameter can be determined by standard validation procedures.

To underscore this, we conclude the article with numerical examples in Section \ref{sec:Numerics} that also demonstrate the advantages of linearized HK over the more classical $\W_2$ linearization. The computational preliminaries are briefly summarized in Section \ref{sec:NumericsPrel}, before we start with a toy example for illustration (Section \ref{sec:NumericsEllipses}) and then two examples on more realistic datasets, cell microscopy images (Section \ref{sec:NumericsCells}) and particle collider events (Section \ref{sec:NumericsJets}).

\subsection{Preliminaries and notation}
\begin{itemize}
	\item $\Omega$ is a convex, closed, bounded subset of $\R^d$ with non-empty interior.
	\item For a compact metric space $X$, $\cont(X)$ denotes the space of continuous real-valued functions over $X$, $\cont(X)^n$ the space of continuous $\R^n$-valued functions.
	\item $\meas(X)$ denotes the space of signed Radon measures over $X$, $\measp(X)$ denotes the space of non-negative Radon measures, $\prob(X)$ the set of probability measures, and the space of $\R^n$-valued measures is denoted by $\meas(X)^n$.
	\item We identify $\meas(X)$ and $\meas(X)^n$ with the dual spaces of $\cont(X)$ and $\cont(X)^n$.
	\item The Lebesgue measure on various domains is denoted by $\Lebesgue$. We add a subscript when the domain is not clear from the context. The set of Lebesgue-absolutely continuous measures is denoted by $\measL(X)$, and $\probL(X)\assign \measL(X) \cap \prob(X)$.
	\item Let $A$ and $B$ be two compact metric spaces and let $\mu : B \to \meas(A)$, $\nu \in \meas(B)$. If there exists a measure $\rho \in \meas(A)$ that satisfies
	\begin{align*}
		\int_A \varphi(a)\,\diff \rho(a) = \int_B \left[ \int_A \varphi(a)\,\diff \big(\mu(b)\big)(a) \right] \diff \nu(b)
	\end{align*}
	for every $\varphi \in \cont(A)$ then we call $\rho$ the Pettis integral of $\mu$ with respect to $\nu$ and write
	\begin{align*}
		\rho = \int_B \mu(b)\,\diff \nu(b).
	\end{align*}
	The same notation applies when $\mu$ is a function into vector-valued measures.
	\item The Euclidean norm and inner product on $\R^n$ are denoted by $\|\cdot\|$ and $\langle\cdot,\cdot\rangle$ respectively, we will also use $\|\cdot\|$ as the norm on $\C$ and the total variation norm on measures (which for positive measures $\mu\in\measp(\Omega)$ is just $\|\mu\|=\mu(\Omega)$).
\end{itemize}
\vspace{\baselineskip}

\section{Wasserstein-2 distance} \label{sec:W2}
In this section we review some basic properties of the Wasserstein-2 distance $\W_2$ on $\prob(\Omega)$.
While they may be familiar to many readers this review serves as a preparation for the Hellinger--Kantorovich distance $\HK$.
Our presentation of $\W_2$ is done in a different order than usual and some expressions could be given more compactly. These deviations are made to keep the presentations of $\W_2$ and $\HK$ as symmetric as possible.
\vspace{\baselineskip}

\subsection{Benamou--Brenier formulation} \label{sec:W2BB}
We start with the dynamic formulation by Benamou and Brenier that computes $\W_2$ by minimizing an action functional over solutions to the continuity equation~\cite{BenamouBrenier2000}.

\begin{definition}[Continuity equation]
	\label{def:WContEq}
	For $\mu_0,\mu_1 \in \prob(\Omega)$ we denote by $\CE(\mu_0,\mu_1)$ the set of solutions for the continuity equation on $[0,1] \times \Omega$, i.e.~the set of pairs of measures $(\rho,\omega) \in \meas([0,1] \times \Omega)^{1+d}$ where $\rho$ interpolates between $\mu_0$ and $\mu_1$ and that solve
	\begin{align*}
		\partial_t \rho + \ddiv \omega = 0
	\end{align*}
	in a distributional sense. More precisely, we require for all $\phi \in C^1([0,1] \times \Omega)$ that
	\begin{align}
		\int_{[0,1] \times \Omega} \partial_t \phi\,\diff \rho
		+ \int_{[0,1] \times \Omega} \nabla \phi \cdot \diff \omega
		=
		\int_\Omega \phi(1,\cdot)\,\diff \mu_1 - \int_\Omega \phi(0,\cdot)\,\diff \mu_0.
	\end{align}
\end{definition}
\begin{definition}[Wasserstein-2 distance, dynamic formulation \cite{BenamouBrenier2000}]
\label{def:WBB}
Let $J_{\W} : \meas([0,1] \times \Omega)^{1+d} \to \RCupInf$ be given by
\begin{subequations}
\label{eq:WBB}
\begin{align}
	\label{eq:WBBJ}
	J_{\W}(\rho,\omega) & \assign \begin{cases}
			\int_{[0,1] \times \Omega} \|\RadNik{\omega}{\rho}\|^2\diff \rho
				& \tn{if } \rho \geq 0, \omega \ll \rho \\
				+ \infty & \tn{else.}
		\end{cases}
\end{align}
Then for $\mu_0,\mu_1 \in \prob(\Omega)$ we set
\begin{align}	
	\label{eq:WBBMin}
	\W_2(\mu_0,\mu_1)^2 & \assign \inf \left\{
		J_{\W}(\rho,\omega)
		\middle| (\rho,\omega) \in \CE(\mu_0,\mu_1)
		\right\}.
\end{align}
\end{subequations}
\end{definition}

It is well known that $\W_2$ is a metric on $\prob(\Omega)$.
Minimizers of \eqref{eq:WBBMin} exist and will be referred to as \emph{constant speed geodesics} between $\mu_0$ and $\mu_1$ with respect to $\W_2$.

\begin{remark}
	\label{rem:TimeNotation}
	For $(\rho,\omega) \in \CE(\mu_0,\mu_1)$ with $J_{\W}(\rho,\omega)<\infty$ it is easy to see that the time-marginals of $\rho$ and $\omega$ are Lebesgue-absolutely continuous.
	Therefore, it is often convenient to describe measures $\rho$ and $\omega$ via their disintegration with respect to time.
	For instance, by specifying $\rho_t \in \meas(\Omega)$ for Lebesgue-almost every $t \in [0,1]$ we characterize a measure $\rho \in \meas([0,1] \times \Omega)$ by setting
	\begin{align*}
		\int_{[0,1] \times \Omega} \phi\,\diff \rho \assign \int_{[0,1]} \int_\Omega \phi(t,\cdot)\,\diff \rho_t\,\diff t
	\end{align*}
	for all $\phi \in \cont([0,1] \times \Omega)$. We write $\rho=\rho_t \otimes \diff t$. Similarly we will proceed for $\omega$ and other measures used later on.
\end{remark}

\begin{proposition}[$\W_2$ distance and geodesics between Dirac measures]
\label{prop:WDirac}
For $x_0, x_1 \in \Omega$ one finds
\begin{align}
	\W_2(\delta_{x_0},\delta_{x_1})^2 & = \|x_0-x_1\|^2\,.
\end{align}
Let
\begin{align}
	X(x_0,x_1;t) & = (1-t)\,x_0+t\,x_1\,.
		\label{eq:WDiracGeodesicX}
\end{align}
The unique constant speed geodesic between $\delta_{x_0}$ and $\delta_{x_1}$ for the $\W_2$ metric is given by (cf.~Remark \ref{rem:TimeNotation})
\label{eq:WDiracGeodesic}
\begin{align}
	\rho_t & = \delta_{X(x_0,x_1;t)}\, &
	\omega_t & = \delta_{X(x_0,x_1;t)} \cdot \partial_t X(x_0,x_1;t)\,.
\end{align}
\end{proposition}
For fixed $x_0,x_1 \in \Omega$, $X(x_0,x_1;\cdot)$ parametrizes the straight line between them at constant speed, which is the constant speed geodesic in $\Omega$. A Dirac-to-Dirac geodesic in $W_2$ consists of a single Dirac traveling along this line.
\vspace{\baselineskip}

\subsection{Kantorovich formulation} \label{sec:W2Kant}
Of course there is also the `original' static Kantorovich formulation of $\W_2$.
\begin{proposition}[Kantorovich-type formulation of $\W_2$]
\label{prop:WStatic}
Let $\mu_0, \mu_1 \in \prob(\Omega)$. Then one has
\begin{align}
	\W_2(\mu_0,\mu_1)^2 & = \inf\left\{ \int_{\Omega^2} \|x_0-x_1\|^2\,\diff \pi(x_0,x_1)
		+ \sum_{i\in \{0,1\}} \iota_{\{\mu_i\}}(\proj_{i\sharp} \pi) \,\middle|\,
		\pi \in \prob(\Omega^2)
		\right\}
	\label{eq:WStatic}
\end{align}
where $\iota_{\{\mu_i\}}$ denotes the indicator function of $\{\mu_i\}$, i.e.
\begin{align*}
	\iota_{\{\mu_i\}}(\nu) & \assign \begin{cases}
		0 & \tn{if } \nu=\mu_i, \\
		+\infty & \tn{else.}
		\end{cases}
\end{align*}
and $\proj_i : \Omega \times \Omega \to \Omega$, $(x_0,x_1) \mapsto x_i$ are the projections from the product space onto the marginals.
The set $\Pi(\mu_0,\mu_1)=\{\pi \in \prob(\Omega^2)\,|\, \proj_{i\sharp} \pi=\mu_i \tn{ for } i=0,1\}$ is called the set of transport plans or couplings between $\mu_0$ and $\mu_1$.
Moreover, minimal $\pi$ in \eqref{eq:WStatic} exist and we will denote the set of $\pi\in\prob(\Omega^2)$ that minimize \eqref{eq:WStatic} for $\W_2(\mu_0,\mu_1)$ by $\PiOpt(\mu_0,\mu_1)$.
If $\mu_0 \ll \Lebesgue$ then the minimizer $\pi$ in~\eqref{eq:WStatic} is unique and lives on the graph of a Monge map, i.e.~$\pi=(\id,T)_\sharp \mu_0$ for some measurable $T : \Omega \to \Omega$.
\end{proposition}
Equivalence between dynamic and static formulation for $\W_2$ is a classical result. Proofs can be found, for instance, in \cite[Theorem 8.1]{Villani-TOT2003} and \cite[Theorem 5.28]{Santambrogio-OTAM}. Existence, uniqueness and Monge-map structure of minimizers are also treated in \cite{Villani-TOT2003,Santambrogio-OTAM}, see also \cite{MonotoneRerrangement-91}.
\begin{remark}
On non-compact domains one typically adds the assumption that $\mu_0$ and $\mu_1$ have finite second moments in Proposition~\ref{prop:WStatic}. This is trivially satisfied in our setting since we assume $\Omega$ to be compact.
\end{remark}
\vspace{\baselineskip}

\subsection{Geodesics from optimal couplings} \label{sec:W2Geo}
Constant speed geodesics for $\W_2(\mu_0,\mu_1)$ (i.e.~minimizers of \eqref{eq:WBB}) can be constructed from optimal couplings $\pi$ in \eqref{eq:WStatic} by superposition of Dirac-to-Dirac geodesics.
The intuition behind this is as follows: when $\pi(x_0,x_1)>0$ for a pair $(x_0,x_1) \in \Omega^2$, the corresponding amount of mass travels along a Dirac-to-Dirac geodesic from $x_0$ to $x_1$. The geodesic between $\mu_0$ and $\mu_1$ is the superposition of all these Dirac-to-Dirac geodesics.
A precise formula is given in the following proposition.

\begin{proposition}[General geodesics for $\W_2$]
	\label{prop:WGeodesics}
	Let $\mu_0, \mu_1 \in \prob(\Omega)$ and let $\pi \in \measp(\Omega \times \Omega)$ be a corresponding minimizer of \eqref{eq:WStatic}.
	Define $X$ by~\eqref{eq:WDiracGeodesicX}.
	Then a constant speed geodesic between $\mu_0$ and $\mu_1$ is given by
    \begin{subequations}	
	\label{eq:WGeodesics}
	\begin{align}
		\rho_t & \assign
			\int_{\Omega^2} \delta_{X(x_0,x_1;t)}\,\diff \pi(x_0,x_1)
		 = X(\cdot,\cdot;t)_\sharp \pi \\
	\intertext{where $X(\cdot,\cdot;t)_\sharp \pi$ denotes the push-forward of $\pi$ under $X$ with the $t$-argument fixed, and}
		\omega_t & \assign
			\int_{\Omega^2} \left[\delta_{X(x_0,x_1;t)}
				\cdot \partial_t X(x_0,x_1;t)\right]\,\diff \pi(x_0,x_1)
		 = X(\cdot,\cdot;t)_\sharp \left(\partial_t X(\cdot,\cdot;t) \cdot \pi\right).
	\end{align}
	\end{subequations}
\end{proposition}
For details see \cite[Theorem 5.27]{Santambrogio-OTAM}.
We emphasize that formulas \eqref{eq:WGeodesics} are a superposition of formulas \eqref{eq:WDiracGeodesic}.
\vspace{\baselineskip}

\subsection[Riemannian structure of W2]{Riemannian structure of $\W_2$}
\label{sec:W2Riemann}
We briefly review some aspects of the Riemannian structure of $\W_2$.
A more complete picture can be found, for instance, in \cite[Sections 2.3.2 and 7.2]{AmbrosioGigli2013}.
At the formal level, \eqref{eq:WBB} looks like a functional to find constant speed geodesics on a Riemannian manifold, where the manifold is $\prob(\Omega)$, the curve is given by $t \mapsto \rho_t$, the tangent vectors are encoded by the velocity field $v_t \assign \RadNik{\omega_t}{\rho_t}$, which must lie in $\Lp{2}(\Omega,\rho_t)$ $t$-almost everywhere when $J_\W$ is finite, and the Riemannian inner product between tangent vectors $v$ and $w$ at $\rho_t$ is given by
\begin{align}
	g_{\W_2}(\rho_t;v,w) \assign \int_\Omega \la v,w\ra \,\diff \rho_t.
\end{align}
The relation between tangent vectors $v_t$ and the curve $\rho_t$ is encoded in the continuity equation, see Definition~\ref{def:WContEq}.

Let $\mu_0 \in \probL(\Omega)$, i.e.~$\mu_0 \ll \Lebesgue$,  $\mu_1 \in \prob(\Omega)$, $\pi \in \Pi(\mu_0,\mu_1)$ be a minimizing coupling for $\W_2(\mu_0,\mu_1)^2$ in \eqref{eq:WStatic}, and $(\rho, \omega)$ be corresponding minimizers of \eqref{eq:WBB} that were constructed via Proposition \ref{prop:WGeodesics}. Then one finds
\begin{align}
	\label{eq:W2TangentFormPre}
	\W_2(\mu_0,\mu_1)^2 & = \int_0^1 \int_\Omega \|v_t\|^2\,\diff \rho_t\,\diff t
	\qquad \tn{with} \qquad v_t \assign \RadNik{\omega_t}{\rho_t}.
\end{align}
	By assumption on $\mu_0$ the optimal $\pi$ is unique and induced by a Monge map $T : \Omega \to \Omega$, i.e.~$\pi=(\id,T)_\sharp \mu_0$ (see Proposition \ref{prop:WStatic}) and in this particular case one obtains:
\begin{align*}
	\rho_t & = X(\cdot,\cdot;t)_\sharp (\id,T)_\sharp \mu_0 = \big((1-t) \cdot \id + t \cdot T\big)_\sharp \mu_0, \\
	\omega_t & = X(\cdot,\cdot;t)_\sharp\left( \partial_t X(\cdot,\cdot;t) \cdot (\id,T)_\sharp \mu_0 \right) = \big((1-t) \cdot \id + t \cdot T\big)_\sharp \big( (T-\id) \cdot \mu_0 \big)
\end{align*}
and thus $v_t\big((1-t)\,x_0 + t\,T(x_0)\big) = T(x_0)-x_0$
and in particular
\begin{align}
	\label{eq:W2Logarithmic}
	v_0(x_0) & = T(x_0)-x_0.
\end{align}
Consequently, for all $t \in [0,1]$
\begin{align*}
	\int_\Omega \left\|\RadNik{\omega_t}{\rho_t}\right\|^2\,\diff \rho_t
	& = \int_\Omega \left\|v_t\right\|^2\,\diff \big((1-t) \cdot \id + t \cdot T\big)_\sharp \mu_0 \\
	& = \int_\Omega \left\|v_t \circ \big((1-t) \cdot \id + t \cdot T\big) \right\|^2\,\diff \mu_0 \\
	& = \int_\Omega \left\|v_0 \right\|^2\,\diff \mu_0
\end{align*}
and finally
\begin{align}
	\label{eq:W2TangentForm}
	\W_2(\mu_0,\mu_1)^2 & = \int_\Omega \left\|v_0\right\|^2\,\diff \mu_0 = g_{\W_2}(\mu_0;v_0,v_0).
\end{align}
We interpret the map $T \mapsto v_0$ implied by \eqref{eq:W2Logarithmic} such that it takes $\mu_1$ to the tangent vector at $t=0$ of the constant-speed geodesic from $\mu_0$ to $\mu_1$. Thus, it is formally the logarithmic map at $\mu_0$ and we will denote it in the following as $\Log_{\W_2}(\mu_0;\cdot)$. With this, \eqref{eq:W2TangentForm} becomes
\begin{equation} \label{eq:W2Euc}
	\W_2(\mu_0,\mu_1)^2 = g_{\W_2}\big(\mu_0;\Log_{\W_2}(\mu_0;\mu_1),\Log_{\W_2}(\mu_0;\mu_1)\big).
\end{equation}
This is analogous to the classical result in (finite-dimensional) Riemannian geometry.
As is well known, the corresponding exponential map is given by $\Exp_{\W_2}(\mu_0,v_0) \assign (\id + v_0)_\sharp \mu_0$ and one finds $\Exp_{\W_2}(\mu_0,v_0)=T_\sharp \mu_0=\mu_1$, as expected.


\subsection{Linearized Wasserstein-2 Distance}
\label{sec:W2Lin}

We can use the formulation~\eqref{eq:W2Euc} to linearize $\W_2$ around the support point $\mu_0 \in \probL(\Omega)$, as was proposed in \cite{OptimalTransportTangent2012} for applications in the geometric analysis of ensembles of images. For $\mu_1, \mu_2 \in \prob(\Omega)$ we set
\begin{equation} \label{eq:WLinMonge}
\begin{split}
	\WLin(\mu_0;\mu_1,\mu_2)^2 & \assign g_{\W_2}\big(\mu;\Log_{\W_2}(\mu_0;\mu_1)-\Log_{\W_2}(\mu_0;\mu_2),\Log_{\W_2}(\mu_0;\mu_1)-\Log_{\W_2}(\mu_0;\mu_2) \big) \\
	&	= \int_\Omega \left\| \Log_{\W_2}(\mu_0;\mu_1)-\Log_{\W_2}(\mu_0;\mu_2) \right\|^2 \,\diff \mu_0. 
\end{split}
\end{equation}
We call $\WLin$ the linearized Wasserstein-2 distance (we note that in~\cite{OptimalTransportTangent2012} this was called the linearized optimal transport distance, however optimal transport is a more general term that includes both Wasserstein distances and Hellinger--Kantorovich distances so to avoid confusion we change the terminology from~\cite{OptimalTransportTangent2012}).
We notice that $\WLin(\mu_0;\mu_1,\mu_2)$ can be written as the $\Lp{2}(\mu_0)$ distance on $\Omega$ between $\Log_{\W_2}(\mu_0;\mu_1)$ and $\Log_{\W_2}(\mu_0;\mu_2)$.
This identification is useful as it now allows one to use off-the-shelf tools from Euclidean geometry such as principal component analysis.

When the Monge map $T$, used in the definition of the logarithmic map, does not exist one instead uses the shortest generalized geodesic (see also~\cite[Definition 9.2.2]{AmbrosioGradientFlows2005}).
That is we let
\begin{equation} \label{eq:GenGeo}
\begin{split}
\gamma\in \prob(\Omega^3) \quad \text{with} \quad & P_{i\sharp} \gamma = \mu_i, P_{01\sharp} \gamma \in \PiOpt(\mu_0,\mu_1) \text{ and } P_{02\sharp} \gamma \in \PiOpt(\mu_0,\mu_2)
\end{split}
\end{equation}
(such a $\gamma$ exists by~\cite[Lemma 5.3.2]{AmbrosioGradientFlows2005}) and define the linearized optimal transport distance by
\begin{equation} \label{eq:WLinGen}
\WLin(\mu_0;\mu_1,\mu_2)^2 \assign \min_{\gamma \text{ satisfying~\eqref{eq:GenGeo}}} \int_{\Omega^3} \|x_1-x_2\|^2 \, \diff \gamma(x_0,x_1,x_2).
\end{equation} 
In the case the Monge map exists then $\PiOpt(\mu_0,\mu_1)$ and $\PiOpt(\mu_0,\mu_2)$ contain a single transport plan each which can be written as $\pi_{01} = (\id,T_{01})_{\sharp}\mu_0\in \PiOpt(\mu_0,\mu_1)$, $\pi_{02} = (\id,T_{02})_{\sharp}\mu_0\in \PiOpt(\mu_0,\mu_2)$, and furthermore the $\gamma$ that satisfies~\eqref{eq:GenGeo} is unique and given by $\gamma=(\id,T_{01},T_{02})_{\sharp}\mu_0$.
In this case~\eqref{eq:WLinMonge} and~\eqref{eq:WLinGen} coincide.
If the Monge map does not exist then the minimisation in~\eqref{eq:WLinGen} is no longer necessarily over a singleton, even though optimal plans $\pi_{01}$ and $\pi_{02}$ might be unique.

To remove the minimisation in~\eqref{eq:WLinGen} in the general case (without Monge maps), following~\cite{OptimalTransportTangent2012}, one approximates optimal plans $\pi_{01}\in \PiOpt(\mu_0,\mu_1)$, $\pi_{02}\in\PiOpt(\mu_0,\mu_2)$ (when the optimal plans are not unique then one must choose amongst the set of optimal plans -- in practice this is determined by the algorithm used to solve the Kantorovich optimisation problem~\eqref{eq:WStatic}) by a plan induced by a map through barycentric projection, namely
\begin{align*}
\pi_{01} & \approx (\id,T_{01})_{\sharp}\mu_0, & T_{01}(x_0) \assign \int_\Omega x_1 \, \diff \pi_{01,x_0}(x_1), \\
\pi_{02} & \approx (\id,T_{02})_{\sharp}\mu_0, & T_{02}(x_0) \assign \int_\Omega x_2 \, \diff \pi_{02,x_0}(x_2)
\end{align*}
where $\{\pi_{01,x_0}\}_{x_0\in\Omega}\subset \prob(\Omega)$, $\{\pi_{02,x_0}\}_{x_0\in\Omega}\subset \prob(\Omega)$ are the disintegrations of $\pi_{01}$, $\pi_{02}$ with respect to $\mu_0$, i.e.
\begin{align*}
\int_{\Omega^2} \phi(x_0,x_1) \, \diff \pi_{01}(x_0,x_1) & = \int_\Omega \l \int_\Omega \phi(x_0,x_1) \, \diff \pi_{01,x_0}(x_1) \r \, \diff \mu_0(x_0), \\
\int_{\Omega^2} \phi(x_0,x_2) \, \diff \pi_{02}(x_0,x_2) & = \int_\Omega \l \int_\Omega \phi(x_0,x_2) \, \diff \pi_{02,x_0}(x_2) \r \, \diff \mu_0(x_0)
\end{align*}
for any measurable function $\phi:\Omega^2\to [0,\infty]$ (see~\cite[Theorem 5.3.1]{AmbrosioGradientFlows2005}).
The approximate Monge maps $T_{01}$ and $T_{02}$ are then used in \eqref{eq:WLinMonge}.
\vspace{\baselineskip}

\section{Hellinger--Kantorovich distance} \label{sec:HK}
In this section we review some properties of the Hellinger--Kantorovich distance $\HK$ on $\measp(\Omega)$. The presentation is structured analogously to that of $\W_2$.

\subsection{Benamou--Brenier-type formulation}
\label{sec:HKBB}
We first introduce $\HK$ as a generalization of the Benamou--Brenier formula (Definition \ref{def:WBB}) where now an action is minimized over solutions to the continuity equation with source.
\begin{definition}[Continuity equation with source]
	For $\mu_0,\mu_1 \in \meas(\Omega)$ we denote by $\CES(\mu_0,\mu_1)$ the set of solutions for the continuity equation with source on $[0,1] \times \Omega$, i.e.~the set of triplets of measures $(\rho,\omega,\zeta) \in \meas([0,1] \times \Omega)^{1+d+1}$ where $\rho$ interpolates between $\mu_0$ and $\mu_1$ and that solve
	\begin{align*}
		\partial_t \rho + \ddiv \omega = \zeta
	\end{align*}
	in a distributional sense. More precisely, we require for all $\phi \in \cont^1([0,1] \times \Omega)$ that
	\begin{align}
		\int_{[0,1] \times \Omega} \partial_t \phi\,\diff \rho
		+ \int_{[0,1] \times \Omega} \nabla \phi \cdot \diff \omega
		+ \int_{[0,1] \times \Omega} \phi\,\diff \zeta
		=
		\int_\Omega \phi(1,\cdot)\,\diff \mu_1 - \int_\Omega \phi(0,\cdot)\,\diff \mu_0.
	\end{align}
\end{definition}

\begin{definition}[Hellinger--Kantorovich distance, dynamic formulation \cite{KMV-OTFisherRao-2015,ChizatOTFR2015,LieroMielkeSavare-HellingerKantorovich-2015a}]
\label{def:HKBB}
Let $J_{\HK} : \allowbreak \meas([0,1] \allowbreak \times \Omega)^{1+d+1} \to \RCupInf$ be given by
\begin{subequations}
\label{eq:HKBB}
\begin{align}
	\label{eq:HKBBJ}
	J_{\HK}(\rho,\omega,\zeta) & \assign \begin{cases}
			\int_{[0,1] \times \Omega} \left(
				\lda\RadNik{\omega}{\rho}\rda^2 + \tfrac{1}{4} \l\RadNik{\zeta}{\rho}\r^2 \right) \diff \rho
				& \tn{if } \rho \geq 0, \omega, \zeta \ll \rho, \\
				+ \infty & \tn{else.}
		\end{cases}
	\intertext{Then for $\mu_0,\mu_1 \in \measp(\Omega)$ we set}
	\label{eq:HKBBMin}
	\HK(\mu_0,\mu_1)^2 & \assign \inf \left\{
		J_{\HK}(\rho,\omega,\zeta)
		\middle| (\rho,\omega,\zeta) \in \CES(\mu_0,\mu_1)
		\right\}.
\end{align}
\end{subequations}
\end{definition}

\begin{remark}[Intrinsic length scale of $\HK$]
\label{rem:HKLengthScale}
The factor of $\frac14$ in $J_{\HK}$ is chosen for algebraic convenience.
In more generality one could define 
\[
	J_{\HK,\kappa}(\rho,\omega,\zeta) \assign \begin{cases}
		\int_{[0,1] \times \Omega} \l \lda\RadNik{\omega}{\rho}\rda^2
		+ \tfrac{\kappa^2}{4} \l\RadNik{\zeta}{\rho}\r^2 \r \diff \rho &
		 \text{if } \rho \geq 0, \omega, \zeta \ll \rho, \\
		+ \infty & \text{else.}
	\end{cases}
\]
The parameter $\kappa>0$ controls the relative importance of the transport part of the cost, $\int_{[0,1] \times \Omega} \|\RadNik{\omega}{\rho}\|^2\diff \rho$, and the destruction/creation part, $\int_{[0,1] \times \Omega} (\RadNik{\zeta}{\rho})^2\diff \rho$.
In particular, if we let $\HK_{\kappa}(\mu_0,\mu_1)$ be as in~\eqref{eq:HKBBMin} but with the cost function $J_{\HK,\kappa}$ then $\HK(\mu_0,\mu_1) = \HK_{1}(\mu_0,\mu_1)$, $\HK_\kappa(\mu_0,\mu_1)\to \W_2(\mu_0,\mu_1)$ as $\kappa\to\infty$, and $\HK_\kappa(\mu_0,\mu_1)/\kappa$ converges to the Hellinger--Kakutani distance between $\mu_0$ and $\mu_1$ as $\kappa\to0$~\cite[Theorems 7.22 and 7.24]{LieroMielkeSavare-HellingerKantorovich-2015a}.
Setting the parameter $\kappa$ is equivalent to re-scaling the set $\Omega$ to $\Omega/\kappa$ (and all measures accordingly), then computing $\HK_1$ on $\Omega/\kappa$ and finally multiplying the result by $\kappa$ again.
Transport in $\HK_1$ is bounded by $\frac{\pi}{2}$, in the sense that mass at location $x_0$ will never be transported outside of the ball $B(x_0,\pi/2)$ (see the discussion after Remark~\ref{rem:DiracConeIsometry}) and therefore the transport in $\HK_\kappa$ is bounded by $\frac{\kappa\pi}{2}$.
This observation is useful in applications as it allows one to prescribe how far mass should be transported and provides a good intuition for the choice of $\kappa$.
Our experiments in Section~\ref{sec:Numerics} show that classification performance with respect to $\kappa$ is relatively robust around the optimal value. With intuition and robustness, a coarse cross validation search for an acceptable value is therefore sufficient.
\end{remark}

\begin{proposition}
\label{prop:HKBasic}
$\HK$ is a metric on $\measp(\Omega)$.
Minimizers of \eqref{eq:HKBBMin} exist and will be referred to as \emph{constant speed geodesics} between $\mu_0$ and $\mu_1$ with respect to $\HK$. More concretely, minimal $(\rho,\omega,\zeta)$ in \eqref{eq:HKBBMin} satisfy $\HK(\rho_s,\rho_t)=|s-t|\cdot \HK(\mu_0,\mu_1)$ for Lebesgue-almost all $s,t \in [0,1]$ and
\begin{align*}
	\int_{\Omega} \left(
				\|\RadNik{\omega_t}{\rho_t}\|^2 + \tfrac{1}{4} (\RadNik{\zeta_t}{\rho_t})^2 \right) \diff \rho_t
	= \HK(\mu_0,\mu_1)^2
\end{align*}
for Lebesgue-almost all $t \in (0,1)$.
\end{proposition}
\begin{proof}
	These results are established, for instance, in \cite[Theorems 2.1 and 2.2]{ChizatOTFR2015}
	which rely on \cite[Theorem 5.4]{DNSTransportDistances09}.
\end{proof}

We now find constant speed geodesics between Diracs.

\begin{proposition}[$\HK$ distance between Dirac measures \cite{ChizatOTFR2015,LieroMielkeSavare-HellingerKantorovich-2015a}]
For $x_0, x_1 \in \Omega$, $m_0, m_1 \in \R_+$ one finds
\begin{align}
	\label{eq:HKDirac}
	\HK(\delta_{x_0} \cdot m_0,\delta_{x_1} \cdot m_1)^2 & =
	m_0 + m_1 - 2\sqrt{m_0\,m_1} \Cos(\|x_0-x_1\|)
\end{align}
where $\Cos(s)=\cos\big(\min\{s, \tfrac{\pi}{2}\}\big)$.

\end{proposition}
\begin{proof}
	This result follows quickly from \cite[Theorem 4.1]{ChizatOTFR2015} by plugging the Dirac-to-Dirac geodesic into $J_{\HK}$.
	It is also given in \cite{LieroMielkeSavare-HellingerKantorovich-2015a} in Equation (6.31) or is quickly derived from Theorem 6.7 and (6.34).
\end{proof}

\begin{remark}[Local isometry and geodesics between Dirac measures]
\label{rem:DiracConeIsometry}
For $\|x_0-x_1\|\leq \tfrac{\pi}{2}$ the $\HK$ distance between two Dirac measures equals the distance between two points in $\C$ in polar coordinates where the radii are given by $\sqrt{m_i}$ and the angle between the two vectors is given by $\|x_0-x_1\|$. For instance, for $\|x_0-x_1\|\leq \tfrac{\pi}{2}$ one can write:
\begin{align*}
	\HK(\delta_{x_0} \cdot m_0,\delta_{x_1} \cdot m_1) & =
	\|\sqrt{m_0} - \sqrt{m_1} \exp(i\|x_0-x_1\|)\|
\end{align*}
This is visualized in Figure \ref{fig:DiracConeIsometry}.
\end{remark}

From this local isometry one can deduce the structure of geodesics between Dirac measures. If $\|x_0-x_1\|<\tfrac{\pi}{2}$ the geodesic is given by a single `travelling' Dirac of the form
\begin{align*}
	\rho_t = \delta_{X(t)}\cdot M(t)
\end{align*}
where $X : [0,1] \to \Omega$ describes the movement from $x_0$ to $x_1$ and $M : [0,1] \to \R_+$ the evolution of the mass.
For $\W_2$ one would of course have that $X(t)$ parametrizes the constant speed straight line from $x_0$ to $x_1$ and $M(t)$ would be fixed to 1.
From Remark \ref{rem:DiracConeIsometry} and Figure \ref{fig:DiracConeIsometry} we see that in the $\HK$ metric $X$ and $M$ essentially describe the straight line between the two embedded points $\sqrt{m_0}$ and $\sqrt{m_1} \exp(i\|x_0-x_1\|)$ in $\C$ in polar coordinates.
The angle is given by $X$ and the squared radius is given by $M$. Explicit formulas for $X$ and $M$ (as well as for $\omega$ and $\zeta$) are given in Proposition \ref{prop:HKDiracGeodesics}.

When $\|x_0-x_1\|> \tfrac{\pi}{2}$ the geodesic between the two measures is equal to the geodesic in the Hellinger--Kakutani distance: the mass at $x_0$ is decreased from $m_0$ to $0$, the mass at $x_1$ is increased from $0$ to $m_1$ and no transport occurs. The geodesic will be of the form
\begin{align*}
	\rho_t = \delta_{x_0} \cdot M_0(t) + \delta_{x_1} \cdot M_1(t)\,.
\end{align*}
The precise form of $M_0$ and $M_1$ is given in Proposition \ref{prop:HKDiracGeodesics}.

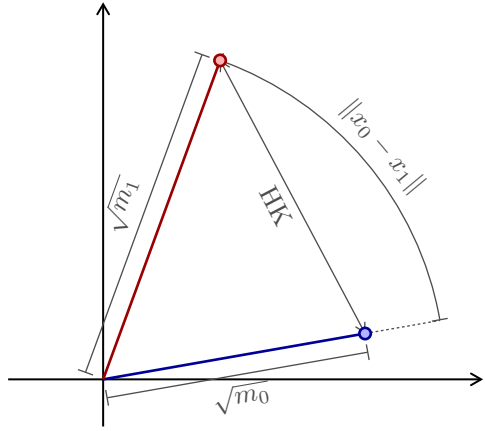
\begin{figure}[hbt]
\centering
	{
\def\radPt{0.03}
\def\lblOffset{0.1}
\def\angMu{10}
\def\angNu{70}
\def\radMu{1.4}
\def\radNu{1.8}
\def\angMuB{0}
\def\angNuB{100}

\begin{tikzpicture}[
				x=2.5cm,y=2.5cm,
				dashedFine/.style={dash pattern=on 1pt off 1pt},
				axis/.style={line width=.75pt, black},
				lbl/.style={line width=.5pt, black!70!white},
				mu/.style={line width=1.pt, blue!60!black, fill=blue!30!white},
				nu/.style={line width=1.pt, red!60!black, fill=red!30!white}
	]
	
	\draw[axis,->] (-0.5,0) -- (2,0);
	\draw[axis,->] (0,-0.25) -- (0,2);
	\coordinate (cMu) at (\angMu:\radMu);
	\coordinate (cNu) at (\angNu:\radNu);

	\draw[lbl,|-|] ($(0,0)+(\angMu-90:\lblOffset)$) -- ($(cMu)+(\angMu-90:\lblOffset)$) node [midway,below,sloped] {$\sqrt{m_0}$};
	\draw[lbl,|-|] ($(0,0)+(\angNu+90:\lblOffset)$) -- ($(cNu)+(\angNu+90:\lblOffset)$) node [midway,above,sloped] {$\sqrt{m_1}$};

	\draw[lbl,<->] ($(cMu)$) -- ($(cNu)$) node [midway,below,sloped] {$\HK$};
	\draw[lbl,|-|] (\angMu:\radNu) arc[start angle=\angMu, end angle=\angNu, radius=\radNu]
		node[midway,above,sloped] {$\|x_0-x_1\|$};

	\draw[lbl,dashedFine] (\angMu:\radMu) -- (\angMu:\radNu);

	\draw[mu] (0,0) -- (cMu);
	\draw[mu] (cMu) circle[radius=\radPt];
	\draw[nu] (0,0) -- (cNu);
	\draw[nu] (cNu) circle[radius=\radPt];

\end{tikzpicture}
}
\caption{Local isometry between Dirac measures $\delta_{x_i} \cdot m_i$, $i=0,1$, with respect to the $\HK$ metric and points in $\C$ with respect to the Euclidean distance. When $\|x_0-x_1\|\leq \tfrac{\pi}{2}$ the geodesic between the two measures is described by the corresponding straight line in $\C$. Angle and (squared) radius of this curve (in constant speed parametrization) are given by the functions $X$ and $M$, \eqref{eq:HKDiracGeodesic}.}
\label{fig:DiracConeIsometry}
\end{figure}

\begin{proposition}[$\HK$ geodesics between Dirac measures]
\label{prop:HKDiracGeodesics}
For $x_0, x_1 \in \Omega$, $\|x_0-x_1\|<\tfrac{\pi}{2}$ $m_0, m_1 \in \R_+$ let
\begin{subequations}
\label{eq:HKDiracGeodesic}
\begin{align}
	\varphi(x_0,m_0,x_1,m_1;t) & \assign \begin{cases}
		\cos^{-1}\l \frac{(1-t)\sqrt{m_0} + t\sqrt{m_1} \cos\|x_0-x_1\|}{\sqrt{M(x_0,m_0,x_1,m_1;t)}} \r
		& \tn{for } t \in (0,1), \\
		0 & \tn{for } t=0 \vee m_1=0, \\
		\|x_0-x_1\| & \tn{for } t=1 \vee m_0=0, \\
		\end{cases}
		\label{eq:HKDiracGeodesicPhi} \\
	X(x_0,m_0,x_1,m_1;t) & \assign x_0 + \frac{x_1-x_0}{\|x_0-x_1\|}
		\cdot \varphi(x_0,m_0,x_1,m_1;t)
		\label{eq:HKDiracGeodesicX} \\
	M(x_0,m_0,x_1,m_1;t) & \assign (1-t)^2 m_0 + t^2m_1 + 2t(1-t) \sqrt{m_0m_1}\cos\|x_0-x_1\|
		\label{eq:HKDiracGeodesicM} \\
		& = (1-t)m_0+tm_1 - t(1-t)\,\HK(m_0 \cdot \delta_{x_0},m_1 \cdot \delta_{x_1})^2\,.
		\label{eq:HKDiracGeodesicMAlt}
\end{align}
\end{subequations}
If $\|x_0-x_1\|<\tfrac{\pi}{2}$ a constant speed geodesic between $\delta_{x_0} \cdot m_0$ and $\delta_{x_1} \cdot m_1$ for the $\HK$ metric is given by (cf.~Remark \ref{rem:TimeNotation}, Proposition \ref{prop:WDirac})
\begin{subequations}
\begin{align}
	\rho_t & \assign \delta_{X(x_0,m_0,x_1,m_1;t)} \cdot M(x_0,m_0,x_1,m_1;t)\,, \\
	\omega_t & \assign \delta_{X(x_0,m_0,x_1,m_1;t)} \cdot M(x_0,m_0,x_1,m_1;t) \cdot \partial_t X(x_0,m_0,x_1,m_1;t)\,, \\
	\zeta_t & \assign \delta_{X(x_0,m_0,x_1,m_1;t)} \cdot \partial_t M(x_0,m_0,x_1,m_1;t)\,.
\end{align}
\end{subequations}
For $\|x_0-x_1\|> \tfrac{\pi}{2}$ the geodesic is given by `teleportation' of the form
\begin{subequations}
\begin{align}
	\rho_t & \assign \delta_{x_0} \cdot M(x_0,m_0,x_0,0;t)
		+ \delta_{x_1} \cdot M(x_1,0,x_1,m_1;t)\,, \\
	\omega_t & \assign 0\,, \\
	\zeta_t & \assign \delta_{x_0} \cdot \partial_t M(x_0,m_0,x_0,0;t)
		+ \delta_{x_1} \cdot \partial_t M(x_1,0,x_1,m_1;t)\,.
\end{align}
\end{subequations}
For $\|x_0-x_1\|=\tfrac{\pi}{2}$ any combination of the two with the proper boundary conditions is a geodesic.

In addition, when $\|x_0-x_1\|<\tfrac{\pi}{2}$, one has for
\begin{align*}
x(t) & \assign X(x_0,m_0,x_1,m_1;t), & m(t) & \assign M(x_0,m_0,x_1,m_1;t),
\end{align*}
that
\begin{align}
	\HK(\delta_{x_0} \cdot m_0, \delta_{x_1} \cdot m_1)^2
	& = \int_0^1 \left[ \|\dot{x}(t)\|^2 + \tfrac14 \l \tfrac{\dot{m}(t)}{m(t)}\r^2 \right] \cdot m(t)\,\diff t
		\label{eq:HKDiracGeodesicXMCost} \\
	& = \left[ \|\dot{x}(t)\|^2 + \tfrac14 \l \tfrac{\dot{m}(t)}{m(t)}\r^2 \right] \cdot m(t)
	\quad \tn{for } t \in (0,1).
		\label{eq:HKDiracGeodesicXMCostConst}
\end{align}
If $m_0>0$ this interval can be extended to $t=0$, if $m_1>0$ it can be extended to $t=1$.
Finally, for $m_0>0$ one has
\begin{align}
	\dot{x}(0) & = \frac{x_1-x_0}{\|x_1-x_0\|} \cdot \sqrt{\frac{m_1}{m_0}} \cdot \sin(\|x_0-x_1\|), &
	\frac{\dot{m}(0)}{m(0)} & = 2 \l \sqrt{\frac{m_1}{m_0}} \cdot \cos(\|x_0-x_1\|) -1\r
	\label{eq:HKDiracGeodesicZeroDerivatives}
\end{align}
with the convention that $\dot{x}(0)=0$ for $x_1=x_0$.
\end{proposition}
\begin{proof}
	The forms of $X, M$ and $(\rho, \omega, \zeta)$ are given, for instance, in \cite[Theorem 4.1]{ChizatOTFR2015}.
	The precise form of the functions $X$ and $M$ is given in the respective proof. These formulas can also be found in \cite{LieroMielkeSavare-HellingerKantorovich-2015a}, equations (8.7).
	 Equation~\eqref{eq:HKDiracGeodesicMAlt} can be found immediately with \eqref{eq:HKDirac}.
	 
	 Equation~\eqref{eq:HKDiracGeodesicXMCost} follows directly by plugging the constructed optimal $(\rho,\omega,\zeta)$ into $J_{\HK}$.
	 Equation~\eqref{eq:HKDiracGeodesicXMCostConst} holds for Lebesgue-almost all $t \in (0,1)$ by Proposition~\ref{prop:HKBasic}.
	 Extension to all $t \in (0,1)$ follows from smoothness of the functions $x$ and $m$ on $t \in (0,1)$. This smoothness extends to $t=0$ and $t=1$ when $m_0>0$ and $m_1>$ respectively.
	 
	 For $m_0>0$, $\frac{\dot{m}(0)}{m(0)}$ can easily be evaluated. In principle $\dot{x}(0)$ could also be obtained by explicit (but tedious) computation. Alternatively, one can argue via conservation of angular momentum, as in Remark~\ref{rem:HKDiracGeodesics}.
\end{proof}

\begin{remark}
\label{rem:HKDiracGeodesics}
It is instructive to study in more detail the precise form of the functions $\varphi$, $X$ and $M$ in~\eqref{eq:HKDiracGeodesic} for $\|x_0-x_1\|<\tfrac{\pi}{2}$. Let $x(t)\assign X(x_0,m_0,x_1,m_1;t)$, $m(t) \assign M(x_0,m_0,x_1,m_1;t)$ and $\varphi(t) \assign \varphi(x_0,m_0,x_1,m_1;t)$.

Set $\varphi_1 \assign \|x_0-x_1\|$, $a_0 \assign \sqrt{m_0}$, $a_1 \assign \sqrt{m_1} \cdot \exp(i\cdot \varphi_1)$. In Remark \ref{rem:DiracConeIsometry} it was noted that $\HK(\delta_{x_0} \cdot m_0, \delta_{x_1} \cdot m_1)^2 = \|a_0-a_1\|^2$ and we deduced that the constant speed geodesic between the two measures is described by a constant speed straight line in $\C$, given by $a(t) \assign (1-t)\,a_0 + t\,a_1$.
$m(t)$ and $\varphi(t)$ can then be recovered from the relation $a(t)=\sqrt{m(t)} \cdot \exp(i \cdot \varphi(t))$.
We obtain that
\begin{align*}
m(t) & =\|a(t)\|^2 = (1-t)^2 \cdot \|a_0\|^2 + t^2 \cdot \|a_1\|^2 + 2\,t\,(1-t) \cdot \re(a_0 \cdot a_1) \\
	& = (1-t)^2 \cdot m_0 + t^2 \cdot m_1 + 2\,t\,(1-t) \cdot \sqrt{m_0\,m_1} \cdot \cos(\varphi_1).
\end{align*}
This corresponds to \eqref{eq:HKDiracGeodesicM}.
Via the law of cosines one finds
\begin{align*}
 \|a(t)-a_0\|^2 = \|a_0\|^2 + \|a(t)\|^2 - 2 \sqrt{\|a_0\|\,\|a(t)\|} \cos(\varphi(t))
\end{align*}
from which, using $\|a(t)\|^2=m(t)$, one quickly deduces \eqref{eq:HKDiracGeodesicPhi}. The angle $\varphi$ must then be `projected' back to the straight line between $x_0$ and $x_1$ via \eqref{eq:HKDiracGeodesicX}.

Since $a(t)$ describes constant speed movement on a straight line in $\C$ one has
\begin{align}
	\label{eq:HKDiracGeodesicELA}
	\const = a(t) \times \dot{a}(t) = \dot{\varphi}(t) \cdot \|a(t)\|^2
	= \|\dot{x}(t)\| \cdot m(t) \assignRe \omega_0.
\end{align}
Here $a \times b$ denotes $\re(a) \cdot \im(b)-\im(a) \cdot \re(b)$ which is related to the cross product.
This corresponds to the conservation of angular momentum. With $\dot{a}(t)=a_1-a_0$ one finds
\begin{align}
	\label{eq:HKDiracGeodesicOmegaZero}
	\omega_0 & = a_0 \times (a_1-a_0) = a_0 \times a_1 = \|a_0\| \cdot \|a_1\| \cdot \sin(\varphi_1)=
	\sqrt{m_0\,m_1} \cdot \sin(\varphi_1).
\end{align}
When $m_0>0$, combining \eqref{eq:HKDiracGeodesicELA} and \eqref{eq:HKDiracGeodesicOmegaZero} yields
\begin{align*}
\|\dot{x}(0)\|=\sqrt{\tfrac{m_1}{m_0}} \sin(\|x_1-x_0\|)
\end{align*}
from which one obtains the left side of \eqref{eq:HKDiracGeodesicZeroDerivatives}.
\end{remark}
\vspace{\baselineskip}

\subsection{Kantorovich-type formulation}
\label{sec:HKKantorovich}
Similar to the classical $\W_2$ distance there are also (multiple) static (\emph{Kantorovich-type}) formulations of $\HK$ in terms of measures on the product space $\Omega \times \Omega$.
Unlike in the classical `balanced' case, transport can no longer be described by a coupling $\pi \in \Pi(\mu_0,\mu_1)$, since particles may change their mass during transport (and $\mu_0$ and $\mu_1$ may even have different mass, i.e.~$\Pi(\mu_0,\mu_1)=\emptyset$).

In the static formulation for $\HK$ given in \cite{LieroMielkeSavare-HellingerKantorovich-2015a} the effect of mass changes is captured by choosing a particular cost function and by relaxing the marginal constraints $\proj_{i\sharp} \pi = \mu_i$ from Proposition \ref{prop:WStatic} and penalizing the difference with the Kullback--Leibler divergence instead.

\begin{definition}[Kullback--Leibler divergence]
For $\mu, \nu \in \meas(\Omega)$ we set
\begin{align}
	\KL(\mu|\nu) & = \begin{cases}
		\int \varphi(\RadNik{\mu}{\nu})\,\diff \nu & \tn{if } \mu,\nu \geq 0, \mu \ll \nu, \\
		+ \infty & \tn{else}
		\end{cases}
\end{align}
with $\varphi(s) = s\,\log(s)-s+1$ for $s>0$ and $\varphi(0)=1$. 
\end{definition}

Note that $\varphi$ is strictly convex and continuous on $\R_+$, and $\KL$ is convex and lower-semicontinuous on $\meas(\Omega)^2$.
We now state a Kantorovich formulation of $\HK$, cf Proposition~\ref{prop:WStatic}.

\begin{proposition}[Kantorovich-type formulation of $\HK$: soft-marginals {\cite[Theorem 8.18]{LieroMielkeSavare-HellingerKantorovich-2015a}}]
\label{prop:HKStatic}
Let
\begin{subequations}
\label{eq:HKSoftMarginal}
\begin{align}
	\label{eq:HKSoftMarginalCost}
	\cHK(x_0,x_1) & \assign \begin{cases}
		-2\log(\cos(\|x_0-x_1\|)) & \tn{if } \|x_0-x_1\| < \tfrac{\pi}{2} \\
		+ \infty & \tn{else.}
	\end{cases}	 \\
	\label{eq:HKSoftMarginalObjective}
	\JHKSM(\pi) & \assign
	\int_{\Omega^2} \cHK\,\diff \pi
		+ \sum_{i\in \{0,1\}} \KL(\proj_{i\sharp} \pi|\mu_i). \\
	\intertext{Then}
	\HK(\mu_0,\mu_1)^2 & = \inf\left\{  \JHKSM(\pi) \,\middle|\,
		\pi \in \measp(\Omega^2)
		\right\}
	\label{eq:HKSoftMarginalProblem}
\end{align}
\end{subequations}
and minimal $\pi$ in \eqref{eq:HKSoftMarginalProblem} exist.
\end{proposition}

As in the Wasserstein case there are conditions under which optimal transport maps exist for $\HK$.

\begin{proposition}
	\label{prop:HKMonge}
	Let $\mu_0 \in \measpL(\Omega)$, $\mu_1 \in \measp(\Omega)$. Then the minimizer $\pi$ for $\HK(\mu_0,\mu_1)^2$ in \eqref{eq:HKSoftMarginal} is unique and induced by a Monge map, i.e.~there is some $T : \Omega \to \Omega$ such that $\pi=(\id,T)_{\sharp} \sigma$ for $\sigma = \proj_{0\sharp} \pi$.
\end{proposition}
\begin{proof}
	This is a direct application of \cite[Theorem 6.6]{LieroMielkeSavare-HellingerKantorovich-2015a}.
\end{proof}

\begin{remark}[Global mass rescaling behaviour of $\HK$]
	\label{rem:HKMassScaling}
	Let $\mu_0, \mu_1 \in \prob(\Omega)$ and $m_0,m_1 \in \R_+$. It was shown in \cite[Theorem 3.3]{LaMi2017} that
	\begin{align}
		\label{eq:HKMassScaling}
		\HK(m_0 \cdot \mu_0, m_1 \cdot \mu_1)^2 = \sqrt{m_0 \cdot m_1} \cdot \HK(\mu_0,\mu_1)^2 + (\sqrt{m_0}-\sqrt{m_1})^2
	\end{align}
	and if $\pi$ is optimal in \eqref{eq:HKSoftMarginal} for $\HK(\mu_0,\mu_1)^2$ then $\sqrt{m_0\cdot m_1} \cdot \pi$ is optimal for $\HK(m_0 \cdot \mu_0, m_1 \cdot \mu_1)^2$.

	This implies that the `unbalanced' effects of the $\HK$ distance are already fully encoded in its behaviour on probability measures. Extension to measures of arbitrary mass is done via the simple formula \eqref{eq:HKMassScaling}.
	
	Consequently, the benefit for data analysis applications that we expect from using $\HK$ instead of $\W_2$ is not so much the ability to deal with differences in the \emph{total} mass of measures, but its ability to deal with \emph{local} mass discrepancies, i.e.~creating mass in one part of the image while reducing it in another part, if this seems more likely than a long range transport.
	
	Therefore, for numerical purposes we typically normalize our samples before comparison. If the total mass of samples is deemed relevant for the subsequent analysis, its effect can be recovered via the transformation \eqref{eq:HKMassScaling} or the total masses can be kept as separate features.
\end{remark}

The following lemma, giving some properties of optimal couplings, will be useful in the sequel.

\begin{lemma} 
	\label{lem:SoftCouplingProperties}
	Let $\mu_0, \mu_1 \in \measp(\Omega)$ and $\pi \in \measp(\Omega^2)$ be a minimizer for $\HK(\mu_0,\mu_1)$ in \eqref{eq:HKSoftMarginal}.
	Set $\nu_i=\proj_{i\sharp} \pi$ and let
	\begin{align}
		\label{eq:SoftCouplingDecomp}
		\mu_i = u_i \cdot \nu_i + \mu_i^\perp
	\end{align}
	be the Lebesgue-decomposition of $\mu_i$ with respect to $\nu_i$ for $i \in \{0,1\}$.
	Then,
	\begin{enumerate}[(i)]
		\item $\nu_i \ll \mu_i$, \label{item:SemiCouplingPropertiesMargDom}
		\item $\|x_0-x_1\|\geq \pi/2$ $\mu_0(x_0)$-$\mu_1^\perp(x_1)$-almost everywhere, in particular $\mu_0\perp\mu_1^\perp$ and the same statement holds with the roles of $\mu_0$ and $\mu_1$ reversed.
		\label{item:SemiCouplingPropertiesMargPerp} %
	\end{enumerate}
\end{lemma}
Intuitively, the mass particles in $\mu_0^\perp$ and $\mu_1^\perp$ do not undergo any transport, but are entirely handled by the Hellinger term in \eqref{eq:HKBB}, or by the $\KL$ terms in \eqref{eq:HKSoftMarginal}. Lemma \ref{lem:SoftCouplingProperties} \eqref{item:SemiCouplingPropertiesMargPerp} states that this only happens if there is no other alternative, i.e.~any potential transport target is at least at distance $\pi/2$ and therefore has $+\infty$ cost in \eqref{eq:HKSoftMarginal}.

\begin{proof}
\eqref{item:SemiCouplingPropertiesMargDom}: Plugging $\pi=0$ into \eqref{eq:HKSoftMarginal} we find that $\HK(\mu_0,\mu_1)^2 \leq \|\mu_0\|+\|\mu_1\|<\infty$.
Thus the optimal $\pi$ must satisfy $\KL(\proj_{i\sharp}\pi|\mu_i)<\infty$ which implies $\nu_i = \proj_{i\sharp}\pi \ll \mu_i$.

\eqref{item:SemiCouplingPropertiesMargPerp}:
	Assume the statement were not true. By inner regularity of Radon measures there will then be measurable $B_0, B_1 \subset \Omega$ and $\veps>0$ such that
	\begin{align*}
		\|x_0-x_1\| \leq \tfrac{\pi}{2}-\veps \quad
		\forall x_0 \in B_0,\, x_1 \in B_1, \qquad
		\mu_0(B_0)>0, \qquad \mu_1(B_1)=\mu_1^\perp(B_1)>0,
		\qquad
		\nu_1(B_1)=0.
	\end{align*}
	Set $\lambda_i \assign \tfrac{\mu_i \restr B_i}{\mu_i(B_i)}$ for $i \in \{0,1\}$ and $\lambda \assign \lambda_0 \otimes \lambda_1$. For $\delta \geq 0$ we now analyze
	\begin{align*}
		\JHKSM(\pi + \delta \cdot \lambda)-\JHKSM(\pi) & = \delta \cdot \int_{\Omega^2} \cHK\,\diff \lambda
			+ [\KL(\nu_ 0 \restr B_0 + \delta \cdot \lambda_0|\mu_0 \restr B_0)-\KL(\nu_0 \restr B_0|\mu_0\restr B_0)] \\
			& \qquad + [\KL(\delta \cdot \lambda_1|\mu_1 \restr B_1)-\KL(0|\mu_1\restr B_1)].
	\end{align*}
	Since $\|x_1-x_0\|\leq \pi/2-\veps$ $\lambda(x_0,x_1)$-almost everywhere, the first term is linear in $\delta$ with finite slope. For the second term we obtain by convexity of $\KL$ that it is bounded from above by
	\begin{align*}
		\delta \cdot [\KL(\nu_0 \restr B_0+\lambda_0|\mu_0 \restr B_0) - \KL(\nu_0 \restr B_0|\mu_0 \restr B_0)].
	\end{align*}
	This expression is linear in $\delta$ with finite slope since $\RadNik{\lambda_0}{\mu_0}=\mu_0(B_0)^{-1}$ on $B_0$.
	The third term equals
	\begin{align*}
		[\varphi(\delta/\mu_1(B_1))-\varphi(0)] \cdot \mu_1(B_1).
	\end{align*}
	Together this implies that
	\begin{align*}
		\JHKSM(\pi + \delta \cdot \lambda)-\JHKSM(\pi) \leq \delta \cdot C + [\varphi(\delta/\mu_1(B_1))-\varphi(0)] \cdot \mu_1(B_1)
	\end{align*}
	for some $C \in \R$. Since the slope of $\varphi(z)$ diverges to $-\infty$ as $z \searrow 0$, for sufficiently small $\delta>0$ this will be negative. Hence, $\pi$ would not be optimal, concluding the proof by contradiction.
\end{proof}
\vspace{\baselineskip}

\subsection{Geodesics from static optimal couplings}
\label{sec:HKGeodesics}
Similar to the standard Wasserstein-2 distance, constant speed geodesics for $\HK(\mu_0,\mu_1)$ can be constructed via the superposition of Dirac-to-Dirac geodesics although the formula is more complicated, cf Proposition~\ref{prop:WGeodesics}.

\begin{proposition}[General geodesics for $\HK$ from optimal soft-marginal coupling]
	\label{prop:HKGeodesicsSoftMarginal}
	Let $\mu_0, \mu_1 \in \measp(\Omega)$ and let $\pi \in \measp(\Omega \times \Omega)$ be a corresponding minimizer of \eqref{eq:HKSoftMarginal}. Set $\nu_i = \proj_{i\sharp} \pi$ and consider the Lebesgue decompositions of $\mu_i$ with respect to $\nu_i$ as follows:
	\begin{align*}
		\mu_i = u_i \cdot \nu_i + \mu_i^\perp
	\end{align*}
	where $u_i=\RadNik{\mu_i}{\nu_i}$ is the density of the $\nu_i$-absolutely continuous part of $\mu_i$ and $\mu_i^\perp$ is the singular part.
	For convenience we introduce the auxiliary functions
	\begin{align}
		\hat{X}(x_0,x_1;t) & \assign X\big(x_0,u_0(x_0),x_1,u_1(x_1);t\big), &
		\hat{M}(x_0,x_1;t) & \assign M\big(x_0,u_0(x_0),x_1,u_1(x_1);t\big).
	\end{align}
	These are well-defined $\pi(x_0,x_1)$-almost everywhere.

	Then a constant speed geodesic between $\mu_0$ and $\mu_1$ is given by:
	\begin{subequations}
	\label{eq:HKGeodesicsSoftMarginal}
	\begin{align}
		\tilde{\rho}_t & \assign \hat{X}(\cdot,\cdot;t)_\sharp \left( \hat{M}(\cdot,\cdot;t) \cdot \pi \right)
		\label{eq:HKGeodesicsSoftMarginalRhoPre} \\
		\rho_t & \assign \tilde{\rho}_t + (1-t)^2 \cdot \mu_0^\perp + t^2 \cdot \mu_1^\perp
		\label{eq:HKGeodesicsSoftMarginalRho} \\
		\omega_t & \assign
			\hat{X}(\cdot,\cdot;t)_\sharp \left( \partial_t \hat{X}(\cdot,\cdot;t) \cdot \hat{M}(\cdot,\cdot;t) \cdot \pi \right)
		\label{eq:HKGeodesicsSoftMarginalOmega} \\
		\tilde{\zeta}_t & \assign \hat{X}(\cdot,\cdot;t)_\sharp \left( \partial_t \hat{M}(\cdot,\cdot;t) \cdot \pi \right)
		\label{eq:HKGeodesicsSoftMarginalZetaPre} \\
		\zeta_t & \assign \tilde{\zeta}_t - 2(1-t) \cdot \mu_0^\perp + 2t \cdot \mu_1^\perp
		\label{eq:HKGeodesicsSoftMarginalZeta}		
		\intertext{When $\pi=(\id,T)_\sharp \sigma$ is induced by a Monge map, see Proposition~\ref{prop:HKMonge}, then the above formulas simplify and one obtains:}
		\label{eq:HKGeodesicsSoftMarginalSimpleRho}
		\tilde{\rho}_t & = \hat{X}(\cdot,T(\cdot);t)_\sharp
			\l \hat{M}(\cdot,T(\cdot);t) \cdot \sigma \r \\
		\label{eq:HKGeodesicsSoftMarginalSimpleOmega}
		\omega_t & = \hat{X}(\cdot,T(\cdot);t)_\sharp
				\l \partial_t \hat{X}(\cdot,T(\cdot);t)
				\cdot \hat{M}(\cdot,T(\cdot);t) \cdot \sigma \r \\
		\label{eq:HKGeodesicsSoftMarginalSimpleZeta}
		\tilde{\zeta}_t & = \hat{X}(\cdot,T(\cdot);t)_\sharp
			\l \partial_t \hat{M}(\cdot,T(\cdot);t) \cdot \sigma \r
	\end{align}
	\end{subequations}
\end{proposition}

The following Lemmas are needed for the proof (and in the sequel).

\begin{lemma}
	\label{lem:JensenPushfwd}
	Let $A, B$ be measurable spaces, $T : A \to B$ measurable, $\mu \in \meas(A)^n$, $\nu\in\measp(A)$, $\mu \ll \nu$ (which implies $T_\sharp \mu \ll T_\sharp \nu$), and $f : \R^n \to \RCupInf$ convex.
	Then
	\begin{align*}
		\int_B f\left(\RadNik{T_\sharp \mu}{T_\sharp \nu}\right)\,\diff T_\sharp \nu
		\leq \int_A f\left(\RadNik{\mu}{\nu}\right)\,\diff \nu.
	\end{align*}
\end{lemma}
\begin{proof}
By disintegration of measures (e.g.~\cite[Thm.~5.3.1]{AmbrosioGradientFlows2005}) there exists a measurable family $\{\nu_b\}_{b\in B}\subset \prob(A)$ such that $T(a)=b$ for $\nu_b$-a.e.~$a$ for $T_\sharp \nu$-a.e.~$b$ and
\[ \int_A \phi(a) \, \diff \nu(a) = \int_B \left[ \int_{A} \phi(a) \,\diff\nu_b(a) \right] \diff T_\sharp \nu(b) \]
for all measurable maps $\phi : A \to[0,\infty]$.
Consequently, for any measurable $\phi : B \to [0,\infty]^n$ one has
\begin{multline*}
	\int_B \phi(b) \cdot \RadNik{T_\sharp \mu}{T_\sharp \nu}(b)\,\diff T_\sharp \nu(b)
	= \int_B \phi(b) \cdot \diff T_\sharp \mu(b)
	= \int_A \phi(T(a)) \cdot \diff \mu(a)
	= \int_A \phi(T(a)) \cdot \RadNik{\mu}{\nu}(a)\,\diff \nu(a) \\
	= \int_B \left[ \int_A \phi(T(a)) \cdot \RadNik{\mu}{\nu}(a)\,\diff \nu_b(a) \right]
		\diff T_\sharp \nu(b)
	= \int_B \phi(b) \cdot \left[ \int_A \RadNik{\mu}{\nu}(a)\,\diff \nu_b(a) \right]
		\diff T_\sharp \nu(b).
\end{multline*}
From this we deduce that
\begin{align}
	\label{eq:HKChofVar}
	\RadNik{T_\sharp \mu}{T_\sharp \nu}(b) = \int_A \RadNik{\mu}{\nu}(a)\,\diff \nu_b(a)
\end{align}
for $T_\sharp \nu$-a.e.~$b \in B$.
Now one finds
\begin{align*}
	\int_A f\left(\RadNik{\mu}{\nu}\right)\,\diff \nu
	& = \int_B \left[ \int_A f\left(\RadNik{\mu}{\nu}(a)\right)\,\diff \nu_b(a) \right] \diff T_\sharp \nu(b) \\
	& \geq \int_B \left[ f \left(
		\int_A \RadNik{\mu}{\nu}(a)\,\diff \nu_b(a) \right) \right] \diff T_\sharp \nu(b) \\
	& = \int_B f\left(\RadNik{T_\sharp \mu}{T_\sharp \nu}(b)\right)\,\diff T_\sharp \nu(b)
\end{align*}
where the inequality is due to Jensen and the last equality is due to \eqref{eq:HKChofVar}.
\end{proof}

\begin{lemma}
	\label{lem:HKKLInequality}
	Let $x_0,x_1 \in \Omega$, $u_0,u_1,m \in \R_+$. Then
	\begin{align}
	\label{eq:HKKLInequality}
		\cHK(x_0,x_1) \cdot m + \sum_{i \in \{0,1\}} \varphi\left(\tfrac{m}{u_i}\right) \cdot u_i
		\geq \HK(u_0 \cdot \delta_{x_0},u_1 \cdot \delta_{x_1})^2
	\end{align}
	where for $u_i=0$ we use the convention $\varphi(m/u_i) \cdot u_i=0$ if $m=0$ and $+\infty$ if $m>0$.
	Equality holds if and only if $m=\cos(\min\{\|x_0-x_1\|,\pi/2\}) \cdot \sqrt{u_0 \cdot u_1}$.
\end{lemma}
\begin{proof}
	The left side is a strictly convex function of $m \in \R_+$.
	Separating the different cases, e.g.~whether $\|x_0-x_1\|$ is smaller or greater-equal than $\pi/2$, and whether some $u_i$ are zero, one quickly finds that the minimum equals the right side and that the unique minimizer is as specified.
\end{proof}

We can now prove Proposition~\ref{prop:HKGeodesicsSoftMarginal}.

\begin{proof}[Proof of Proposition \ref{prop:HKGeodesicsSoftMarginal}]
	We start by showing that $(\rho,\omega,\zeta)$ as constructed above are contained in $\CES(\mu_0,\mu_1)$.
	Let $\phi \in C^1([0,1] \times \Omega)$. Then one finds
	\begin{align*}
		&
		\int_{[0,1] \times \Omega} \partial_t \phi\,\diff \rho
		+ \int_{[0,1] \times \Omega} \nabla \phi \cdot \diff \omega
		+ \int_{[0,1] \times \Omega} \phi\,\diff \zeta
		\displaybreak[0] \\
		= {} &
			\int_0^1 \int_{\Omega^2} \Big[
				(\partial_t \phi)(t,\hat{X}(x_0,x_1;t)) \cdot \hat{M}(x_0,x_1;t) \\
		& \qquad \qquad +(\nabla \phi)(t,\hat{X}(x_0,x_1;t)) \cdot \partial_t \hat{X}(x_0,x_1;t) \cdot \hat{M}(x_0,x_1;t) \\
		& \qquad \qquad + \phi(t,\hat{X}(x_0,x_1;t)) \cdot \partial_t \hat{M}(x_0,x_1;t)
				\Big]
				\diff \pi(x_0,x_1)\,
				\diff t \\
		& \qquad + \int_0^1 \int_\Omega \Big[
			(\partial_t \phi)(t,\cdot) \cdot (1-t)^2 - \phi(t,\cdot) \cdot 2(1-t) \Big] \diff \mu_0^\perp \,\diff t \\
		& \qquad + \int_0^1 \int_\Omega \Big[
			(\partial_t \phi)(t,\cdot) \cdot t^2 + \phi(t,\cdot) \cdot 2t \Big] \diff \mu_1^\perp \,\diff t
		\displaybreak[0] \\
		= {} &
			\int_0^1 \int_{\Omega^2} \frac{\diff}{\diff t}\Big[
				\phi(t,\hat{X}(x_0,x_1;t)) \cdot \hat{M}(x_0,x_1;t)
				\Big]
				\diff \pi(x_0,x_1)\,
				\diff t \\
		& \qquad + \int_0^1 \int_{\Omega} \frac{\diff}{\diff t}\Big[
				\phi(t,\cdot) \cdot (1-t)^2
				\Big]
				\diff \mu_0^\perp\,
				\diff t
		+ \int_0^1 \int_{\Omega} \frac{\diff}{\diff t}\Big[
				\phi(t,\cdot) \cdot t^2
				\Big]
				\diff \mu_1^\perp\,
				\diff t \\
		= {} &
			\int_{\Omega^2}
				\Big[\phi(t,\hat{X}(x_0,x_1;t)) \cdot \hat{M}(x_0,x_1;t)\Big]_{t=0}^1
				\diff \pi(x_0,x_1)
			+ \int_\Omega \phi(1,\cdot)\,\diff \mu_1^\perp
			- \int_\Omega \phi(0,\cdot)\,\diff \mu_0^\perp
				\assignRe \tn{($\ast$)}
	\end{align*}
	Now we use that for $\pi$-almost all $(x_0,x_1)$ one has
	\begin{align*}	
	\hat{X}(x_0,x_1;0) & =x_0, &
	\hat{X}(x_0,x_1;1) & =x_1, &
	\hat{M}(x_0,x_1;0) &=u_0(x_0), &
	\hat{M}(x_0,x_1;1) &=u_1(x_1)
	\end{align*}
	to find:
	\begin{align*}
		\tn{($\ast$)} & =
			\int_{\Omega^2}
				\Big[\phi(1,x_1) \cdot u_1(x_1)
				-\phi(0,x_0) \cdot u_0(x_0)
				\Big]
				\diff \pi(x_0,x_1)
			+ \int_\Omega \phi(1,\cdot)\,\diff \mu_1^\perp
			- \int_\Omega \phi(0,\cdot)\,\diff \mu_0^\perp \\
		& = \int_\Omega \phi(1,\cdot)\,\diff \mu_1 - \int_\Omega \phi(0,\cdot)\,\diff \mu_0.
	\end{align*}

	Next, we show optimality in \eqref{eq:HKBBMin}.
	Plugging $(\rho,\omega,\zeta)$ constructed above into $J_{\HK}$, \eqref{eq:HKBBJ}, one obtains
	\begin{multline}
		J_{\HK}(\rho,\omega,\zeta) 
		\leq J_{\HK}\big(\tilde{\rho}_t \otimes \diff t, \omega, \tilde{\zeta}_t \otimes \diff t\big) 
		\\
		+ J_{\HK}\big((1-t)^2 \cdot \mu_0^\perp \otimes \diff t, 0, - 2 (1-t) \cdot \mu_0^\perp \otimes \diff t\big)
		+ J_{\HK}\big(t^2 \cdot \mu_1^\perp \otimes \diff t, 0, 2 t \cdot \mu_1^\perp \otimes \diff t\big)
		\label{eq:HKGeodesicConstructionProofJDecomp}
	\end{multline}
	where we decompose $(\rho,\omega,\zeta)$ into the moving and teleporting parts and use that $J_{\HK}$ is sub-additive (which follows from convexity and positive 1-homogeneity). See Remark \ref{rem:TimeNotation} for the notation.
	For the second term we get
	\begin{multline}
		J_{\HK}\big((1-t)^2 \cdot \mu_0^\perp \otimes \diff t, 0, - 2 (1-t) \cdot \mu_0^\perp \otimes \diff t\big)		
		= \int_0^1 \int_\Omega \tfrac{1}{4} \left(\frac{-2(1-t)}{(1-t)^2}\right)^2 \,(1-t)^2\diff \mu_0^\perp \,\diff t
		= \|\mu_0^\perp\|
		\label{eq:HKGeodesicConstructionProofJDecompB}
	\end{multline}
	and likewise for the third term
	\begin{align}
		J_{\HK}\big(t^2 \cdot \mu_1^\perp \otimes \diff t, 0, 2 t \cdot \mu_1^\perp \otimes \diff t\big) = \|\mu_1^\perp\|.
		\label{eq:HKGeodesicConstructionProofJDecompC}
	\end{align}
	For the first term we find
	\begin{align}
		& J_{\HK}\big(\tilde{\rho}_t \otimes \diff t, \omega, \tilde{\zeta}_t \otimes \diff t\big)
		= \int_0^1 \int_\Omega \left( \left\|\RadNik{\omega_t}{\tilde{\rho}_t}\right\|^2 + \tfrac{1}{4} \left(\RadNik{\tilde{\zeta}_t}{\tilde{\rho}_t}\right)^2 \right)\,\diff \tilde{\rho}_t \,\diff t
		\nonumber \\		
		\intertext{and then by using Lemma \ref{lem:JensenPushfwd} (and flipping the order of integration)}
		\leq {} & \int_{\Omega^2} \int_0^1 \left( \left\|\partial_t \hat{X}(x_0,x_1;t)\right\|^2 + \tfrac{1}{4} \left(\frac{\partial_t \hat{M}(x_0,x_1;t)}{\hat{M}(x_0,x_1;t)}\right)^2 \right)\,\hat{M}(x_0,x_1)\,\diff t\,\diff \pi(x_0,x_1)
		\nonumber \\
		\intertext{with Proposition~\ref{prop:HKDiracGeodesics} one obtains that the inner integral equals $\HK\left(
			u_0(x_0) \cdot \delta_{x_0},u_1(x_1) \cdot \delta_{x_1}\right)^2$ and thus}
		= {} & \int_{\Omega^2} \HK\left(
			u_0(x_0) \cdot \delta_{x_0},u_1(x_1) \cdot \delta_{x_1}\right)^2\,\diff \pi(x_0,x_1)
		\label{eq:HKGeodesicConstructionProofJDecompIntermediate} \\
		\intertext{and further with \eqref{eq:HKKLInequality} (with $m=1$)}
		\leq {} & \int_{\Omega^2} \left[
			\cHK(x_0,x_1) + \sum_{i \in \{0,1\}} \varphi\left(\frac{1}{u_i(x_i)} \right) u_i(x_i)
			\right] \diff\pi(x_0,x_1)
		\leq \int_{\Omega^2} \cHK \,\diff \pi + \sum_{i \in \{0,1\}} \KL(\nu_i|u_i \cdot \nu_i)
		\label{eq:HKGeodesicConstructionProofJDecompA}
	\end{align}
	Combining now \eqref{eq:HKGeodesicConstructionProofJDecomp}-\eqref{eq:HKGeodesicConstructionProofJDecompA} and using the mutual singularity of $\nu_i$ and $\mu_i^\perp$ we find
	\begin{align*}
		J_{\HK}(\rho,\omega,\zeta) \leq \int_{\Omega^2} \cHK \,\diff \pi + \sum_{i \in \{0,1\}} \KL(\proj_{i\sharp} \pi|\mu_i)
		= \JHKSM(\pi)
	\end{align*}
	and by optimality of $\pi$ for \eqref{eq:HKSoftMarginalProblem}, $(\rho,\omega,\zeta)$ must be optimal for \eqref{eq:HKBB}.
\end{proof}

\begin{corollary}
	\label{cor:HKDecompositionForm}
	When $\pi$ is a minimizer of \eqref{eq:HKSoftMarginalObjective} for computing $\HK(\mu_0,\mu_1)^2$ then
	\begin{align*}
		\HK(\mu_0,\mu_1)^2 = \int_{\Omega^2} \HK(u_0(x_0) \cdot \delta_{x_0}, u_1(x_1) \cdot \delta_{x_1})^2\,\diff \pi(x_0,x_1)
		+ \|\mu_0^\perp\| + \|\mu_1^\perp\|
	\end{align*}
	with $u_i$, $\mu_i^\perp$, $i \in \{0,1\}$ as in Proposition \ref{prop:HKGeodesicsSoftMarginal}.
\end{corollary}
\begin{proof}
	By showing that $(\rho,\omega,\zeta)$ is optimal for $J_{\HK}$ in the previous proof we showed that all inequalities in the previous proof are actually equalities. The result then follows by combining \eqref{eq:HKGeodesicConstructionProofJDecomp} -- \eqref{eq:HKGeodesicConstructionProofJDecompIntermediate}.
\end{proof}
\vspace{\baselineskip}

\section{Linearized Hellinger--Kantorovich distance}
\label{sec:HKLin}

In Sections \ref{sec:W2Riemann} and \ref{sec:W2Lin} we recapped the Riemannian structure of the Wasserstein distance and its local linearization, which we now adapt to the Hellinger--Kantorovich distance.

\subsection{Riemannian structure}
\label{sec:HKRiemann}
First we seek the equivalent of \eqref{eq:W2TangentForm}, i.e.~we want to express $\HK(\mu_0,\mu_1)^2$ in terms of the particles' initial tangent direction at $t=0$.
Since $\HK$ allows transport as well as mass changes, the tangent space will now consist of a velocity field and a mass growth field. Some particular care must be applied to the regions where `teleport' occurs, in particular where mass is created from nothing.

Let $\mu_0 \in \measpL(\Omega)$, $\mu_1 \in \measp(\Omega)$. Let $\pi$ be a corresponding minimizer for $\HK(\mu_0,\mu_1)$ in \eqref{eq:HKSoftMarginal} and let $(\rho,\omega,\zeta)$ be the corresponding minimizers of~\eqref{eq:HKBB} that were constructed via Proposition~\ref{prop:HKGeodesicsSoftMarginal}. Then one has
\begin{align}
	\label{eq:HKTangentFormPre}
	\HK(\mu_0,\mu_1)^2 & = \int_0^1 \int_\Omega \left[
		\left\|\RadNik{\omega_t}{\rho_t}\right\|^2 + \tfrac{1}{4} \left(\RadNik{\zeta_t}{\rho_t}\right)^2
		\right]\diff \rho_t \, \diff t.
\end{align}
A subtle difference between \eqref{eq:HKTangentFormPre} and \eqref{eq:W2TangentFormPre} is that in \eqref{eq:HKTangentFormPre} the integrand
\begin{align*}
\int_\Omega \left[ \left\|\RadNik{\omega_t}{\rho_t}\right\|^2 + \tfrac{1}{4} \left(\RadNik{\zeta_t}{\rho_t}\right)^2 \right]\diff \rho_t
\end{align*}
must be handled with particular care for $t \in \{0,1\}$, as one may have that $\RadNik{\zeta_t}{\rho_t}$ diverges in some locations as $t \to 0$ and $1$ where $\rho_t$ vanishes in the limit $t=\{0,1\}$. Thus, we cannot simply rewrite \eqref{eq:HKTangentFormPre} in terms of this integrand at $t=0$, as we have done for $\W_2$ in deriving~\eqref{eq:W2TangentForm} from \eqref{eq:W2TangentFormPre}. The following proposition handles these subtleties.

\begin{proposition}
\label{prop:HKTangentForm}
Let $\mu_0 \in \measpL(\Omega)$, $\mu_1 \in \measp(\Omega)$. Let $\pi$ be the unique minimizer for $\HK(\mu_0,\mu_1)$ in \eqref{eq:HKSoftMarginal} which can be written as $\pi=(\id,T)_\sharp \sigma$ for some measurable $T : \Omega \to \Omega$ and $\sigma \in \measp(\Omega)$. Let $(\rho,\omega,\zeta)$ be corresponding minimizers of \eqref{eq:HKBB} that were constructed via Proposition~\ref{prop:HKGeodesicsSoftMarginal} (and let $(\tilde{\rho},\tilde{\zeta})$ be the corresponding parts of $(\rho,\zeta)$ as given by \eqref{eq:HKGeodesicsSoftMarginalRhoPre} and \eqref{eq:HKGeodesicsSoftMarginalZetaPre}.
Consider the following Lebesgue decompositions of $\mu_0$ and $\mu_1$ with respect to the marginals of $\pi$:
\begin{align*}
	\mu_0 & = u_0 \cdot \sigma + \mu_0^\perp, &
	\mu_1 & = u_1 \cdot T_\sharp \sigma + \mu_1^\perp.
\end{align*}
Set further for $t \in [0,1)$:
\begin{align}
	\label{eq:HKTangentVectors}
	v_t & \assign \RadNik{\omega_t}{\rho_t}, &
	\alpha_t & \assign \RadNik{\tilde{\zeta}_t}{\rho_t} -2(1-t) \RadNik{\mu_0^\perp}{\rho_t}.
\end{align}
Then
\begin{subequations}
\label{eq:HKTangentSimpleForm}
\begin{align}
	\label{eq:HKTangentSimpleFormV}
	v_0(x) & = \begin{cases}
		\frac{T(x)-x}{\|T(x)-x\|} \cdot \sqrt{\frac{u_1(T(x))}{u_0(x)}} \cdot \sin(\|T(x)-x\|) & \tn{$\sigma$-a.e.}, \\
		0 & \tn{$\mu_0^\perp$-a.e.,}
	\end{cases} \\
	\label{eq:HKTangentSimpleFormAlpha}
	\alpha_0(x) & = \begin{cases}
		2 \l \sqrt{\frac{u_1(T(x))}{u_0(x)}} \cdot \cos(\|T(x)-x\|)-1 \r & \tn{$\sigma$-a.e.}, \\
		-2 & \tn{$\mu_0^\perp$-a.e..}
	\end{cases}
\end{align}
\end{subequations}
with the convention $v_0(x)=0$ if $T(x)=x$, and
\begin{align}
	\label{eq:HKTangentForm}
	\HK(\mu_0,\mu_1)^2 & = \int_\Omega \left[\|v_0\|^2 + \tfrac{1}{4} (\alpha_0)^2 \right] \,\diff \mu_0 + \|\mu_1^\perp\|\,.
\end{align}
\end{proposition}

Here $v_t$ describes the spatial movement of mass particles, $\alpha_t$ describes the change of mass of moving particles and of those that disappear entirely at $t=1$. And $\mu_1^\perp$ describes the mass particles that are created from nothing.

\begin{proof}
	Uniqueness and the form $\pi = (\id,T)_\sharp \sigma$ of the optimal coupling follows from Proposition~\ref{prop:HKMonge} which implies
	\begin{align*}
		\proj_{0\sharp} \pi & = \sigma, &
		\proj_{1\sharp} \pi & = T_\sharp \sigma.
	\end{align*}
	First, we evaluate 	\eqref{eq:HKTangentVectors} explicitly at $t=0$. For this we plug $\pi=(\id,T)_\sharp \sigma$ into Proposition~\ref{prop:HKGeodesicsSoftMarginal}. One obtains:
	\begin{align}
		\tilde{\rho}_0 & = u_0 \cdot \sigma, \nonumber \\
		\rho_0 & = \mu_0, \nonumber \\
		\omega_0 & = \partial_t X(\cdot,u_0(\cdot),T(\cdot),u_1(T(\cdot));t)\big|_{t=0} \cdot u_0 \cdot \sigma, \nonumber \\
		\tilde{\zeta}_0 & = \partial_t M(\cdot,u_0(\cdot),T(\cdot),u_1(T(\cdot));t)\big|_{t=0} \cdot \sigma \nonumber
	\end{align}
	and thus
	\begin{subequations}
	\label{eq:HKTangentVectorsExplicit}
	\begin{align}
		v_0 & = \begin{cases}
			\partial_t X(\cdot,u_0(\cdot),T(\cdot),u_1(T(\cdot));t)\big|_{t=0} & \tn{$\sigma$-a.e.,} \\
			0 & \tn{$\mu_0^\perp$-a.e.,}
			\end{cases}
			\label{eq:HKTangentVectorsExplicitV} \\
		\alpha_0 & = \begin{cases}
			\frac{\partial_t M(\cdot,u_0(\cdot),T(\cdot),u_1(T(\cdot));t)\big|_{t=0}}{u_0} & \tn{$\sigma$-a.e.}, \\
			-2 & \tn{$\mu_0^\perp$-a.e.}
			\end{cases}
			\label{eq:HKTangentVectorsExplicitAlpha}
	\end{align}
	\end{subequations}
	which become \eqref{eq:HKTangentSimpleForm} with \eqref{eq:HKDiracGeodesicZeroDerivatives}.
	
	By Corollary \ref{cor:HKDecompositionForm} one has
	\begin{align*}
		\HK(\mu_0,\mu_1)^2 & = \int_\Omega \HK(u_0(x) \cdot \delta_x, u_1(T(x)) \cdot \delta_{T(x)})^2\,\diff \sigma(x)
			+ \|\mu_0^\perp\| + \|\mu_1^\perp\|. \\
	\intertext{According to Lemma \ref{lem:SoftCouplingProperties} \eqref{item:SemiCouplingPropertiesMargDom} we find $\sigma \ll \mu_0$ and thus $u_0>0$ $\sigma$-almost everywhere. Then, with \eqref{eq:HKDiracGeodesicXMCostConst}, extended to $t=0$ (due to $u_0>0$ $\sigma$-a.e.), one obtains}
		\HK(\mu_0,\mu_1)^2 & = \int_\Omega \left[
			\left\|\partial_t X(\cdot,u_0(\cdot),T(\cdot),u_1(T(\cdot));0)\right\|^2
			+ \tfrac14 \l \frac{\partial_t M(\cdot,u_0(\cdot),T(\cdot),u_1(T(\cdot));0)}{u_0} \r^2 \right] \cdot u_0\,\diff \sigma \\
			& \qquad+ \int_\Omega \tfrac14 ({-2})^2\, \diff \mu_0^\perp + \|\mu_1^\perp\|
	\intertext{and with \eqref{eq:HKTangentVectorsExplicit}}
	\HK(\mu_0,\mu_1)^2 & = \int_\Omega \left[ \|v_0\|^2 + \tfrac14 \alpha_0^2 \right] \diff \mu_0 + \|\mu_1^\perp\|\,.
	\end{align*}
\end{proof}

\begin{remark}
\label{rem:AlphaAsymmetry}
Note that we could have decided to set $\alpha_0(x) \assign 0$ $\mu_0^\perp$-a.e.~in \eqref{eq:HKTangentSimpleFormAlpha} and instead include $\|\mu_0^\perp\|$ in \eqref{eq:HKTangentForm} for a more symmetric treatment of $\mu_0^\perp$ and $\mu_1^\perp$. This would however require us to add $\mu_0^\perp$ as an additional component in the definition of the logarithmic map \eqref{eq:HKLog} below. Also, when some parts of $\mu_1$ tend to zero, with the current definition the corresponding values of $\alpha_0$ continuously tend to $-2$, whereas with the alternative definition, they would approach $-2$ but in the limit jump to $0$, and the $\mu_0^\perp$-component would also be discontinuous.
While the logarithmic map \eqref{eq:HKLog} is not continuous, with the above convention, it is at least `less discontinuous'. Discontinuity cannot be avoided entirely, due to the cut-locus of $\HK$ geodesics at the relative distance $\pi/2$ (cf.~Proposition \ref{prop:HKDiracGeodesics}).
\end{remark}

\begin{remark}
\label{rem:HKEuc}
Loosely speaking, by Lemma~\ref{lem:SoftCouplingProperties} \eqref{item:SemiCouplingPropertiesMargPerp} the mass particles of measures $\mu_0$ and $\mu_1^\perp$ are at least at distance $\pi/2$ to each other.
Therefore, if we assume that $\mu_1$ is concentrated on $\spt \mu_0 + B_{\pi/2}$ (i.e.~the Minkowski sum of $\spt \mu_0$ and the centered open ball of radius $\pi/2$) then it follows that $\mu_1^\perp=0$, and in particular, we may write
\[ \HK(\mu_0,\mu_1)^2 = \int_\Omega \left[ \| v_0\|^2 + \frac14 (\alpha_0)^2 \right] \, \diff \mu_0. \]
This holds of course if $\Omega \subset \spt \mu_0 + B_{\pi/2}$, or, when for a dataset of samples $\{\mu_1,\ldots,\mu_n\}$, $\mu_0$ is chosen as linear mean or $\HK$-barycenter (see \cite{FriMatSch19} for details).
\end{remark}

\begin{remark}
\label{rem:HKBarycentric}
In order to treat the case where transport maps $T$ do not exist then one can, as in Section~\ref{sec:W2Lin}, approximate the transport plan $\pi$ by a plan induced by a map, i.e. $\pi \approx (\id,T)_{\sharp} \sigma$ (where we keep $\sigma = P_{0\sharp}\pi$), using barycentric projections described in Section~\ref{sec:W2Lin}.
Notice that $u_1$ and $\mu_1^\perp$ are defined by the Lebesgue decomposition $\mu_1 = u_1\cdot T_{\sharp} \sigma + \mu_1^\perp$.
However, since in equations~\eqref{eq:HKTangentSimpleForm} one only needs $u_1\circ T$ (and in particular does not need to know $u_1$), then rather than define $u_1 = \RadNik{\mu_1}{T_{\sharp} \sigma}$ in numerical implementations we propose to approximate $u_1\circ T$ directly, i.e.
\[ u_1(T(x_0)) \approx \int_\Omega \RadNik{\mu_1}{P_{1\sharp}\pi}(x_1)\,\diff \pi_{x_0}(x_1) \]
where $\{\pi_{x_0}\}_{x_0\in\Omega}\subset\prob(\Omega)$ is the disintegration of $\pi$ with respect to $\sigma$, i.e. $\int_{\Omega^2} \phi(x_0,x_1) \, \diff \pi(x_0,x_1) = \int_\Omega \int_\Omega \phi(x_0,x_1) \, \diff \pi_{x_0}(x_1) \, \diff \sigma (x_0)$ for all measurable functions $\phi:\Omega^2\to [0,\infty]$.
\end{remark}

\subsection{Logarithmic map and Riemannian inner product}
\label{sec:HKLog}
Motivated by Proposition \ref{prop:HKTangentForm} we can now introduce a logarithmic map and an inner product for the $\HK$ distance, in analogy to Section \ref{sec:W2Lin}.
\begin{definition}
	\label{def:HKLog}
	Let $\mu_0 \in \measpL(\Omega)$, $\mu_1 \in \measp(\Omega)$ and let $(v_t,\alpha_t,\mu_1^\perp)$ be given as in Proposition~\ref{prop:HKTangentForm}.
	Then we define the Logarithmic map for $\HK$ at support point $\mu_0$ for the measure $\mu_1$ as
\begin{equation}
	\label{eq:HKLog}
	\Log_{\HK}(\mu_0;\mu_1) \assign (v_0,\alpha_0,\sqrt{\mu_1^\perp}),
\end{equation}
see Remark \ref{rem:SquareRoot} for the square root.
Given another measure $\tilde{\mu}_1 \in \measp(\Omega)$ with $\Log_{\HK}(\mu_0;\tilde{\mu}_1) \assign (\tilde{v}_0,\tilde{\alpha}_0,\sqrt{\tilde{\mu}_1^\perp})$ we define the corresponding inner product as
	\begin{align}
		\label{eq:HKInnerProduct}
		g_{\HK}\big(\mu_0; (v_0,\alpha_0,\sqrt{\mu_1^\perp}), (\tilde{v}_0,\tilde{\alpha}_0,\sqrt{\tilde{\mu}_1^\perp}) \big)
		\assign \int_\Omega \left[ \la v_0,\tilde{v}_0 \ra + \tfrac{1}{4} \alpha_0\,\tilde{\alpha}_0 \right] \, \diff \mu_0
		+ \int_\Omega \sqrt{ \RadNik{\mu_1^\perp}{\lambda}\,\RadNik{\tilde{\mu}_1^\perp}{\lambda}} \,\diff \lambda
	\end{align}
	where $\lambda$ is some measure in $\measp(\Omega)$ with $\mu_1^\perp, \tilde{\mu}_1^\perp \ll \lambda$.\end{definition}

\begin{remark}[Square root of measure]
	\label{rem:SquareRoot}
	For general non-negative measures there is no reasonable definition of a square root.
	The root in the third component of \eqref{eq:HKLog} is to be understood in a purely formal sense.
	It is never evaluated directly, but only in expressions where it becomes meaningful.
	For instance, the integral in the last term of \eqref{eq:HKInnerProduct} is well-defined and does not depend on the choice of $\lambda$ since the function $\R_+^2 \ni (s,t) \mapsto \sqrt{s \cdot t}$ is jointly positively 1-homogeneous. We can therefore interpret this last term as a bilinear form between the square roots of two non-negative measures.
	Likewise, let $\rho$ be the mid-point of a constant speed geodesic between $\mu_0$ and $\mu_1$. Then we should have that $\Log_{\HK}(\mu_0;\rho)=\tfrac12 \cdot \Log_{\HK}(\mu_0;\mu_1)$. This becomes true with the square root (see Propositions \ref{prop:LogarithmicMapGeodesics} and \ref{prop:HKExp}).
	The third term in \eqref{eq:HKInnerProduct} is related to the Hellinger distance, which can be interpreted as the `$L^2$-distance between the measure square roots' and where the same notion of measure square root is used.
\end{remark}

Uniqueness of $(v_0,\alpha_0,\sqrt{\mu_1^\perp})$ is implied by the uniqueness of the optimal coupling $\pi$ from which they are constructed, which follows from Proposition~\ref{prop:HKMonge}. Hence, $\Log_{\HK}(\mu;\cdot)$ is well-defined. 
Referring to \eqref{eq:HKLog} and \eqref{eq:HKInnerProduct} as logarithmic map and inner product is a slight abuse of notation, since the third component of the inner product is merely defined on the cone of non-negative measures and thus lacks the full vector space structure.

Now, analogous to \eqref{eq:W2Euc}, by Proposition~\ref{prop:HKTangentForm} we have
\begin{equation} \label{eq:HKEuc}
\HK(\mu_0,\mu_1)^2 = g_{\HK}(\mu_0;\Log_{\HK}(\mu_0;\mu_1),\Log_{\HK}(\mu_0;\mu_1)).
\end{equation}
And so, in analogy to \eqref{eq:WLinMonge} we use this to linearize $\HK$ around the support point $\mu_0$.
For $\mu_0\in \measpL(\Omega)$, $\mu_1, \mu_2 \in \measp(\Omega)$ we set
\begin{align*}
\HKLin(\mu_0;\mu_1,\mu_2)^2 & \assign g_{\HK}\l \mu_0;\Log_{\HK}(\mu_0;\mu_1)-\Log_{\HK}(\mu_0;\mu_2),\Log_{\HK}(\mu_0;\mu_1)-\Log_{\HK}(\mu_0;\mu_2) \r. 
\end{align*}
We call $\HKLin$ the linear Hellinger--Kantorovich distance.
As in the Wasserstein-2 case we can see the linear Hellinger--Kantorovich distance as a distance between the formal logarithmic maps in an (almost) Euclidean space.
Indeed,
\begin{multline}
\HKLin(\mu_0;\mu_1,\mu_2)^2 = \lda \Log_{\HK}(\mu_0;\mu_1)_1 - \Log_{\HK}(\mu_0;\mu_2)_1 \rda^2_{\Lp{2}(\mu_0)} \\
 + \frac14 \lda \Log_{\HK}(\mu_0;\mu_1)_2 - \Log_{\HK}(\mu_0;\mu_2)_2 \rda^2_{\Lp{2}(\mu_0)}
 + \left\| \Log_{\HK}(\mu_0;\mu_1)_3 - \Log_{\HK}(\mu_0;\mu_2)_3 \right\|^2_{\tn{Hell}}
\end{multline}
where by $\|\cdot\|_{\tn{Hell}}$ we denote the Hellinger distance over measure square roots (see Remark \ref{rem:SquareRoot}).
And thus the linear Hellinger--Kantorovich distance can be embedded in the space $\Lp{2}(\mu_0;\R^d)\times \Lp{2}(\mu_0;\R)\times \sqrt{\measp(\Omega)}$ where the third component is the cone of square roots of non-negative measures, equipped with the Hellinger metric.
In light of Remark~\ref{rem:HKEuc}, when $\mu_0$ has sufficiently wide support on $\Omega$ then the third component is always zero and the embedding can be made into the Euclidean space $\Lp{2}(\mu_0;\R^d) \times\Lp{2}(\mu_0;\R)$ where $g_{\HK}(\mu_0;\cdot,\cdot)$ is an inner product.

The next proposition verifies that the logarithmic map and constant speed geodesics are connected in the expected way.

\begin{proposition}[Logarithmic map and geodesics]
	\label{prop:LogarithmicMapGeodesics}
	Let $\mu_0\in\measpL(\Omega)$, $\mu_1\in\measp(\Omega)$ and let $\Log_{\HK}(\mu_0;\mu_1)=(v_0,\alpha_0,\sqrt{\mu_1^\perp})$, where $(v_0,\alpha_0,\mu_1^\perp)$ are given by Proposition~\ref{prop:HKTangentForm}.
	Let $\rho$ be constructed via Proposition \ref{prop:HKGeodesicsSoftMarginal}, equation \eqref{eq:HKGeodesicsSoftMarginalRho}, and $\tau\in [0,1]$.
	Then,
	\begin{align*}
		\Log_{\HK}(\mu;\rho_\tau) = (\tau \cdot v_0,\tau \cdot \alpha_0,\tau \cdot \sqrt{\mu_1^\perp}).
	\end{align*}
\end{proposition}

\begin{proof}
	The intuition behind this result is simple.
	Let $(\rho_t,\omega_t,\zeta_t)$ be the constant speed geodesic between $\mu_0$ and $\mu_1$ constructed via Proposition~\ref{prop:HKGeodesicsSoftMarginal}. Then it is easy to show that the rescaled measures $(\rho^\tau,\omega^\tau,\zeta^\tau)$ given by
	\begin{align}
		\label{eq:LogMapScalingRhoScaled}
		\rho^\tau_t & \assign \rho_{\tau \cdot t}, &
		\omega^\tau_t & \assign \tau \cdot \omega_{\tau \cdot t}, &
		\zeta^\tau_t & \assign \tau \cdot \zeta_{\tau \cdot t}
	\end{align}
	are a constant speed geodesic between $\mu_0$ and $\rho_\tau$.
	Indeed, by explicit computation one finds that $(\rho^\tau,\omega^\tau,\zeta^\tau) \in \CES(\mu_0,\rho_\tau)$. 	Using the constant speed property of $(\rho,\omega,\zeta)$, cf.~Proposition \ref{prop:HKBasic}, one obtains that
	\begin{align*}
	\int_{\Omega} \left(
				\|\RadNik{\omega^\tau_t}{\rho^\tau_t}\|^2 + \tfrac{1}{4} (\RadNik{\zeta^\tau_t}{\rho^\tau_t})^2 \right) \diff \rho^\tau_t
	= \tau^2 \cdot
	\int_{\Omega} \left(
				\|\RadNik{\omega_{\tau \cdot t}}{\rho_{\tau \cdot t}}\|^2 + \tfrac{1}{4} (\RadNik{\zeta_{\tau \cdot t}}{\rho_{\tau \cdot t}})^2 \right) \diff \rho_{\tau \cdot t}
	= \tau^2 \cdot \HK(\mu_0,\mu_1)^2
	\end{align*}
	for Lebesgue almost all $t \in (0,1)$.
	Since $\HK(\mu_0,\rho_\tau) = \tau \cdot \HK(\mu_0,\mu_1)$ by construction ($\rho_\tau$ was picked at time $\tau$ along a constant-speed geodesic), $(\rho^\tau,\omega^\tau,\zeta^\tau)$ must be a constant speed geodesic between $\mu_0$ and $\rho_\tau$.
	Plugging $(\rho^\tau,\omega^\tau,\zeta^\tau)$ into \eqref{eq:HKTangentVectors} one quickly obtains the desired result.

	However, by the above arguments we merely know that $(\rho^\tau,\omega^\tau,\zeta^\tau)$ describe \emph{some} constant speed geodesic between $\mu_0$ and $\rho_\tau$.
	While we conjecture that it is unique in this case (due to $\mu_0 \ll \Lebesgue$), because of the cut-locus behaviour of $\HK$ at the distance $\pi/2$ (cf.~Proposition \ref{prop:HKDiracGeodesics}) this uniqueness has not yet been established.
	Therefore, to make this proof fully rigorous we must make sure that the above geodesic coincides with the one constructed via Proposition \ref{prop:HKGeodesicsSoftMarginal}.
	The computations are lengthy and somewhat tedious, but essentially standard. 	
	We give a brief sketch.
	Let
	\begin{align*}
		T^\tau & \assign X(\cdot,u_0(\cdot), T(\cdot), u_1(T(\cdot)); \tau), &
		U^\tau & \assign M(\cdot,u_0(\cdot), T(\cdot), u_1(T(\cdot)); \tau).
	\end{align*}
	This gives the intermediate position and relative mass of particles in $\pi$ on the constant speed geodesic between $\mu_0$ and $\mu_1$ at time $\tau$.
	Then by Proposition \ref{prop:HKGeodesicsSoftMarginal}, $\rho_\tau = T^\tau_\sharp (U^\tau \cdot \sigma) + (1-\tau)^2 \cdot \mu_0^\perp + \tau^2 \cdot \mu_1^\perp$.
	Further, set
	\begin{align*}
		W^\tau & \assign \sqrt{u_0 \cdot U^\tau} \cdot \cos(\|\cdot-T^\tau(\cdot)\|), &
		\sigma^\tau & \assign W^\tau \cdot \sigma, &
		\pi^\tau \assign (\id,T^\tau)_\sharp \sigma^\tau
		 + (1-\tau)\cdot (\id,\id)_\sharp \mu_0^\perp.
	\end{align*}
	By plugging $\pi^\tau$ into \eqref{eq:HKSoftMarginalObjective} one finds that it is optimal for $\HK(\mu_0,\rho_\tau)$ in \eqref{eq:HKSoftMarginalProblem} (to show this one uses Lemmas \ref{lem:JensenPushfwd} and \ref{lem:HKKLInequality}, Corollary \ref{cor:HKDecompositionForm}, and the joint sub-additivity of $\KL(\cdot|\cdot)$ in both arguments).
	Since $\mu_0 \ll \Lebesgue$, $\pi^\tau$ is the unique minimizer (Proposition \ref{prop:HKMonge}).
	As in Proposition \ref{prop:HKGeodesicsSoftMarginal} we then consider the Lebesgue-decompositions of $\mu_0$ and $\rho_\tau$ with respect to the marginals of $\pi^\tau$:
	\begin{align*}
		\mu_0 & = u_0^\tau \cdot \proj_{0\sharp} \pi^\tau + \mu_0^{\tau,\perp}, &
		\rho_\tau & = u_1^\tau \cdot \proj_{1\sharp} \pi^\tau + \rho_\tau^{\perp}.		
	\end{align*}
	One finds (by using Lemma \ref{lem:SoftCouplingProperties} \eqref{item:SemiCouplingPropertiesMargPerp} to show singularity of $\sigma^\tau$ and $T^\tau_\sharp \sigma^\tau$ with respect to $\mu_0^\perp$ and $\mu_1^\perp$)
	\begin{align*}
		u_0^\tau & = \begin{cases}
			\sqrt{\tfrac{u_0}{U^\tau}} \cdot \tfrac{1}{\cos(\|\cdot - T^\tau(\cdot)\|)} & \tn{$\sigma^\tau$-a.e.,} \\
			\tfrac{1}{1-\tau} & \tn{$\mu_0^\perp$-a.e.}, \\
			\end{cases} &
		u_1^\tau \circ T^\tau & = \begin{cases}
			\sqrt{\tfrac{U^\tau}{u_0}} \cdot \tfrac{1}{\cos(\|\cdot - T^\tau(\cdot)\|)} & \tn{$\sigma^\tau$-a.e.,} \\
			1-\tau & \tn{$\mu_0^\perp$-a.e.}, \\
			\end{cases}
			\\
		\mu_0^{\tau,\perp} & = 0, &
		\rho_\tau^{\perp} & = \tau^2 \cdot \mu_1^\perp.
	\end{align*}
	Plugging this into Proposition \ref{prop:HKGeodesicsSoftMarginal}, one obtains
	\begin{align}
		\label{eq:LogMapScalingRhoExplicit}
		\rho^\tau_t & = \hat{X}^\tau(\cdot,T^\tau(\cdot);t)_\sharp [\hat{M}^\tau(\cdot,T^\tau(\cdot);t) \cdot \sigma^\tau]
			+ t^2 \cdot \rho_\tau^\perp
	\end{align}
	where we set $\sigma^\tau$-almost everywhere
	\begin{align*}
		\hat{X}^\tau(\cdot,T^\tau(\cdot);t) & = X(\cdot,u_0^\tau(\cdot),T^\tau(\cdot),u_1^\tau(T^\tau(\cdot));t), &
		\hat{M}^\tau(\cdot,T^\tau(\cdot);t) & = M(\cdot,u_0^\tau(\cdot),T^\tau(\cdot),u_1^\tau(T^\tau(\cdot));t).
	\end{align*}
	One finds similar formulas for $\omega^\tau_t$ and $\zeta^\tau_t$.
	Using that $X$ and $M$ parametrize constant-speed geodesics, and the equalities
	\begin{align*}
		X(x_0,\lambda \, u_0,x_1,\lambda\, u_1;t) & = X(x_0,u_0,x_1,u_1;t), &
		M(x_0,\lambda \, u_0,x_1,\lambda\, u_1;t) & = M(x_0,u_0,x_1,u_1;t) \cdot \lambda,
	\end{align*}
	for $\lambda>0$, one finds that
	\begin{align*}
		\hat{X}^\tau(\cdot,T^\tau(\cdot);t) & = \hat{X}(\cdot,T(\cdot);\tau \cdot t), &
		\hat{M}^\tau(\cdot,T^\tau(\cdot);t) \cdot \sigma^\tau & = \hat{M}(\cdot,T(\cdot);\tau \cdot t) \cdot \sigma
	\end{align*}
	and therefore, that the to definitions of $\rho^\tau_t$ in \eqref{eq:LogMapScalingRhoScaled} and \eqref{eq:LogMapScalingRhoExplicit} agree (and so do the analogous formulas for $\omega^\tau_t$ and $\zeta^\tau_t$).
	Therefore the re-parametrized geodesic is indeed the one induced by the unique minimizer of \eqref{eq:HKSoftMarginalProblem}.
\end{proof}
\vspace{\baselineskip}

\subsection{Exponential map}
\label{sec:HKExp}

In the previous section we showed how to map from the space of non-negative measures to the tangent space at $\mu_0\in \measpL(\Omega)$.
In this section we consider the inverse transformation; the exponential map defined below (with the same abuse of notation as discussed in the context of Remark \ref{rem:SquareRoot} above).

\begin{proposition}[Exponential map]
	\label{prop:HKExp}
	Let $\mu_0 \in \measpL(\Omega)$, $\mu_1 \in \measp(\Omega)$ and let $(v_0,\alpha_0,\sqrt{\mu_1^\perp})=\Log_{\HK}(\mu_0;\mu_1)$.
	Set
	\begin{align*}
		a(x) & \assign \|v_0(x)\|, & b(x) & \assign \frac{\alpha_0(x)}{2}+1,
	\end{align*}
	\begin{align*}
		S \assign \left\{ x \in \Omega \,\middle|\, (v_0(x),\alpha_0(x))=(0,-2)
		\,\Leftrightarrow\, (a(x),b(x))=(0,0) \right\},
	\end{align*}
	and
	\begin{align*}
		q(x) & \assign \sqrt{a(x)^2 + b(x)^2}.
	\end{align*}
	For $x \in \Omega \setminus S$ let $\varphi(x)$ be given as unique solution to the equation
	\begin{align*}
		\begin{pmatrix}
			a(x) \\ b(x)
		\end{pmatrix}
		=
		q(x)
		\cdot
		\begin{pmatrix}
			\sin(\varphi(x)) \\ \cos(\varphi(x))
		\end{pmatrix}
		\qquad \tn{subject to} \qquad \varphi(x) \in [0,\pi/2]
	\end{align*}
	and set
	\begin{align*}
		\hat{T}(x) \assign x+\begin{cases}
			\frac{v_0(x)}{\|v_0(x)\|} \cdot \varphi(x) & \tn{if } v_0(x) \neq 0, \\
			0 & \tn{else.}
			\end{cases}
	\end{align*}
	Then,
	\begin{align}
		\label{eq:HKExp}
		\mu_1 & = \hat{T}_{\sharp} (q^2 \cdot \mu_0 \restr (\Omega \setminus S)) + \mu_1^\perp.
	\end{align}
	We refer to this map, from $(\alpha_0,v_0,\sqrt{\mu_1^\perp})$ to $\mu_1$ as \emph{exponential map}, denoted by
	\begin{align*}
		\mu_1 = \Exp_{\HK}(\mu_0;v_0,\alpha_0,\sqrt{\mu_1^\perp}).
	\end{align*}
	In addition,
	\begin{align*}
		[0,1] \ni t \mapsto \rho_t \assign \Exp_{\HK}(\mu_0;t \cdot v_0, t \cdot \alpha_0, t \cdot \sqrt{\mu_1^\perp})
	\end{align*}
	describes the constant speed geodesic between $\mu_0$ and $\mu_1$ given by Proposition~\ref{prop:HKGeodesicsSoftMarginal}.
\end{proposition}

\begin{proof}
	Throughout this proof let $T$, $\sigma$, $\pi$, $u_i$ and $\mu_i^\perp$ be given as in Proposition~\ref{prop:HKTangentForm}.
	
	We start by showing for $\alpha_0$, as given by \eqref{eq:HKTangentSimpleFormAlpha}, that $\alpha_0>-2$ $\sigma$-almost everywhere.
	Indeed, one finds that $\|T-\id\| < \tfrac{\pi}{2}$ $\sigma$-a.e., as the transport cost $\cHK$ is infinite at distance greater or equal $\tfrac{\pi}{2}$ and $(\id,T)_\sharp \sigma$ is optimal for \eqref{eq:HKSoftMarginal}.
	Further, due to Lemma \ref{lem:SoftCouplingProperties} \eqref{item:SemiCouplingPropertiesMargDom} one has $T_\sharp \sigma \ll \mu_1$ and thus $u_1>0$ $T_\sharp \sigma$-a.e., which implies $u_1 \circ T>0$ $\sigma$-almost everywhere.
	This implies that $\sigma(S)=0$.
	Conversely, by construction one has $\alpha_0=-2$ $\mu_0^\perp$-almost everywhere. Together with [$\alpha_0>-2$ on $\Omega \setminus S$] and [$\sigma(S)=0$] this implies that $\mu^\perp_0 = \mu_0 \restr S$.
	
	Plugging now $v_0$ and $\alpha_0$ from \eqref{eq:HKTangentSimpleFormAlpha} into the definitions of $a$, $b$, $q$ and $\varphi$ above, we obtain that
	\begin{align}
		\label{eq:HKExpProofQPhi}
		q^2 & = \frac{u_1 \circ T}{u_0}, & \varphi & = \|T-\id\|
	\end{align}
	$\sigma$-almost everywhere.
	Further, for $\|v_0(x)\|>0$ one has $\frac{v_0}{\|v_0\|} = \frac{T-\id}{\|T-\id\|}$ in \eqref{eq:HKTangentSimpleFormV} and thus also $\hat{T}=T$ $\sigma$-almost everywhere.

	With Lemma \ref{lem:SoftCouplingProperties} \eqref{item:SemiCouplingPropertiesMargDom} one has $\sigma \ll \mu_0$ and thus $u_0>0$ $\sigma$-a.e., which allows us to write $\sigma=\frac{1}{u_0} \cdot (\mu_0-\mu_0^\perp) = \frac{1}{u_0} \cdot \mu_0 \restr (\Omega \setminus S)$.
	With this one obtains
	\begin{align*}
		\mu_1 & = u_1 \cdot T_\sharp \sigma + \mu_1^\perp
		= u_1 \cdot \hat{T}_{\sharp} \sigma + \mu_1^\perp
		= \hat{T}_{\sharp} (u_1 \circ T \cdot \sigma) + \mu_1^\perp
		= \hat{T}_{\sharp} \l \frac{u_1 \circ T}{u_0} \cdot \mu_0 \restr (\Omega \setminus S) \r + \mu_1^\perp
	\end{align*}
	which equals \eqref{eq:HKExp} with \eqref{eq:HKExpProofQPhi}.	

	The second part of the Proposition follows now from Proposition~\ref{prop:LogarithmicMapGeodesics}.
	Indeed, let $\rho_t$ be the point on the constant speed geodesic from $\mu_0$ to $\mu_1$, as given by Proposition~\ref{prop:HKGeodesicsSoftMarginal}. Then by Proposition~\ref{prop:LogarithmicMapGeodesics} one has
	\begin{align*}
		(\hat{v}_0,\hat{\alpha}_0,\hat{\mu}) \assign \Log_{\HK}(\mu_0;\rho_t) = (t \cdot v_0, t \cdot \alpha_0, t \cdot \sqrt{\mu_1^\perp})
	\end{align*}
	and by the first part of this Proposition we know that $\Exp_{\HK}(\mu_0;\hat{v}_0,\hat{\alpha}_0,\hat{\mu}) = \rho_t$.
\end{proof}

\section{Numerical examples}
\label{sec:Numerics}
We now present several numerical examples as illustration for the analytic results of Section \ref{sec:HKLin} and to show the usefulness of the linearized HK metric for data analysis applications, in particular as a refinement of linearized $\W_2$.
Section \ref{sec:NumericsPrel} summarizes the numerical setup. Section \ref{sec:NumericsEllipses} gives an example on simple synthetic images to gain some intuition. Sections \ref{sec:NumericsCells} and \ref{sec:NumericsJets} give examples for the analysis of microscopy images of liver cells and particle collider events, adapted from \cite{OptimalTransportTangent2012} and \cite{Cai_2020}.
Example code is available at \url{https://github.com/bernhard-schmitzer/UnbalancedLOT/}.

\subsection{Preliminaries}
\label{sec:NumericsPrel}
\paragraph{Discretization, numerical approximation of $\HK^2$, logarithmic and exponential map}
In our examples, $\Omega$ is a rectangle in $\R^2$, typically representing the image domain, and the samples $(\mu_i)_{i=1}^n$ are given as sums of Dirac measures at pixel locations.
For numerical approximation of the soft-marginal formulation of $\HK^2$, \eqref{eq:HKSoftMarginal}, we use the unbalanced Sinkhorn algorithm introduced in \cite{ChizatEntropicNumeric2018}, for practical implementation details see  \cite{schmitzer2019stabilized}.
The entropic regularization parameter can be chosen sufficiently small so that the entropic blur is negligible compared to the discretization errors.
On the approximate optimal transport plan $\pi$ (which is typically non-deterministic, due to discretization) we then use the barycentric projection as described in Remark \ref{rem:HKBarycentric} and then extract an approximate logarithmic map.

When applying the exponential map at a discrete reference measure $\mu_0$, the result will be a discrete measure where the points of $\mu_0$ are moved to new locations (and their masses are re-scaled). Often, it is desirable to visualize the new measure on a fixed reference grid (e.g.~the same grid that the input samples live on). For rasterization we then use bilinear interpolation coefficients to distribute mass to nearby grid points.

\paragraph{Choice of reference measure}
The choice of the reference measure $\mu_0$ is important for a successful analysis. Intuitively, the metric space $(\measp(\Omega),\HK)$ is a (weak) curved Riemannian manifold (similar to $(\prob(\Omega),\W_2)$). Thus, if the reference point $\mu_0$ is very far from the samples $(\mu_i)_{i=1}^n$ (or if the samples are very far from each other) then we expect that the local linearization is a poor approximation.
A natural candidate for $\mu_0$ is the barycenter of the samples. The HK barycenter was recently studied in \cite{HKBarycenters2019,FriMatSch19}. In principle, this can be approximated numerically with the methods from \cite{ChizatEntropicNumeric2018,schmitzer2019stabilized}.
For comparison we also consider the linearized $\W_2$ distance. 
The corresponding barycenter was analyzed in \cite{WassersteinBarycenter}, we use the algorithm described in \cite{benamou2015iterative}.
However, for large numbers of samples $n$ or when samples live on large grids, computing the barycenter could be numerically prohibitive. Thus it is worthwhile to consider alternative choices.
In \cite{OptimalTransportTangent2012} it was proposed to use the linear average of the samples, similarly one could consider the Hellinger mean. Finally, in \cite{Cai_2020} a uniform measure on a Cartesian background-grid was used. We will numerically compare using the linear mean of samples with the uniform measure for choice of reference measure in Section~\ref{sec:NumericsJets}.	
In all these cases it is ensured that $\mu_i^\perp=0$, $i=1,\ldots,n$ (see Remark \ref{rem:HKEuc}).

\subsection{Deforming and resizing ellipses}
\label{sec:NumericsEllipses}

\begin{figure}
	\centering
	\includegraphics[width=\textwidth]{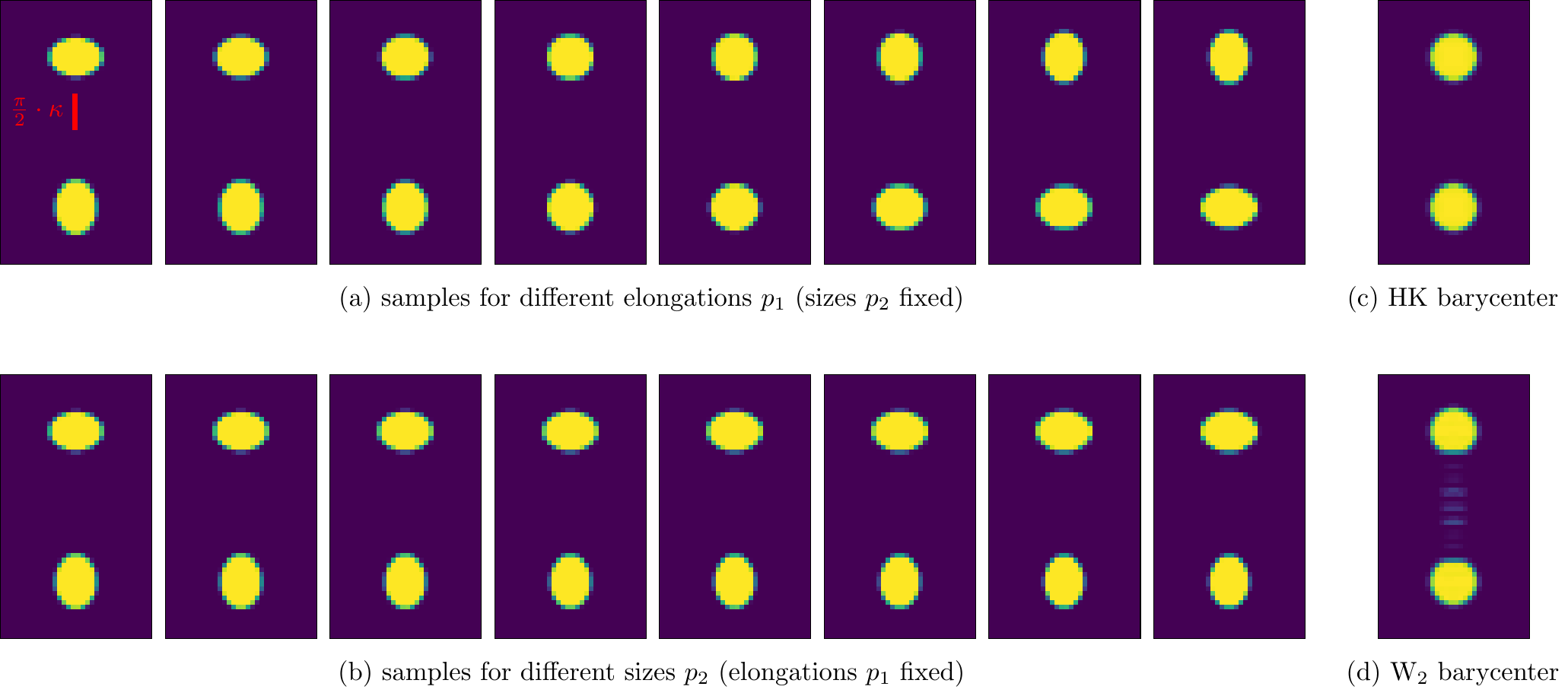}
	\caption{Some representative samples from the ellipses dataset and the corresponding barycenters in the $\HK$ and $\W_2$ distances. The red line in the top left sample has length $\tfrac{\pi}{2}\cdot \kappa$, representing the maximal transport distance under $\HK$.}
	\label{fig:NumericsEllipsesSamples}
\end{figure}

\paragraph{Motivation and setup}
In this example the set of samples is given by images of two ellipses on a $64 \times 64$ pixel grid. The mass density is 1 within the ellipses and 0 outside (and non-binary at the boundaries due to rasterization effects). Each image is characterized by two parameters $p_1, p_2 \in [-1,1]^2$. $p_1$ specifies the elongation of the ellipses. For $p_1=0$, both are circles; for $p_1>0$ one becomes elongated horizontally, the other one vertically; for $p_1<0$ the roles are reversed. 
$p_2$ controls their sizes (and hence their relative masses); sizes being equal for $p_2=0$; one expanding and the other one shrinking for $p_2>0$; and again, roles reversed for $p_2<0$.
The maximal change in the ellipse diameter between $p_2=-1$ and $p_2=+1$ is approximately $0.5$ pixels (ellipses are rendered at a higher resolution first and then reduced to $64 \times 64$ pixels). The corresponding relative change in mass is approximately $10\%$.
Examples for different $p_1,p_2$ are shown in Figure \ref{fig:NumericsEllipsesSamples} (a,b).
The sample images are generated by sampling both parameters on $8$ equidistant points from $[-1,1]$, yielding a total of $n=64$ input images.
Prior to analysis, all images are normalized to unit mass, cf.~Remark \ref{rem:HKMassScaling}.

We now analyze this dataset with the linearized $\W_2$ and the linearized $\HK$ distances.
For $\HK$ we set the length-scale parameter $\kappa$ to $5$ (see Remark \ref{rem:HKLengthScale}) so that the maximal transport distance is sufficient to track the deformation of the ellipses induced by $p_1$ but it separates well the two ellipses from each other (cf.~Figure \ref{fig:NumericsEllipsesSamples}). In more realistic applications several values for this parameter should be tested and validated, as in Section \ref{sec:NumericsJets}.
As reference point $\mu_0$ we use the (approximate) $\W_2$/$\HK$-barycenter of the samples. These are visualized in Figure \ref{fig:NumericsEllipsesSamples} (c,d).
Note that due to the difference of the masses of the two ellipses due to variations of the parameter $p_2$, the $\W_2$ barycenter has some mass located in between the two disks (the discrete pattern stems from the fact that we used only a finite number of values for $p_2$), whereas the $\HK$ barycenter exhibits no such artifacts.
Then we obtain the linear embeddings via the logarithmic maps.
We set $w_{i,0} \assign \Log_{\W_2}(\mu_0;\mu_i)$ and $(v_{i,0},\alpha_{i,0}) \assign \Log_{\HK}(\mu_0;\mu_i)$ for $i=1,\ldots,n$ (for our choice of $\mu_0$, the singular component $\mu_i^\perp$ is always zero, see above, and so we ignore it).

\begin{figure}
	\centering
	\includegraphics[]{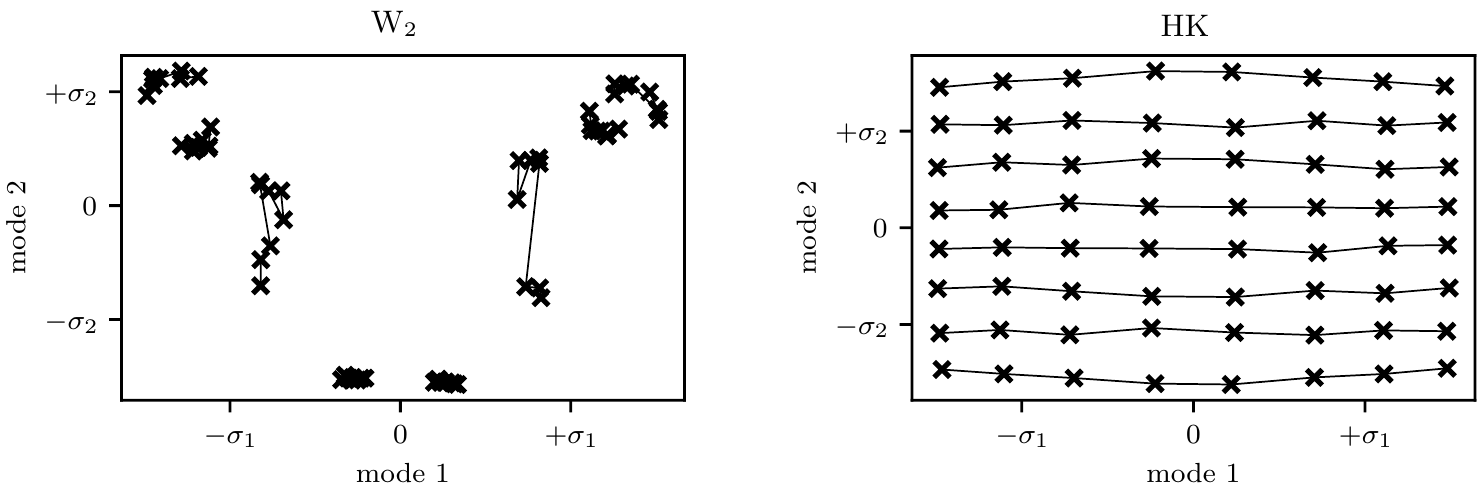}
	\caption{Coordinates of ellipse samples in PCA basis (two most dominant modes, axis scaling given in terms of standard deviation along each mode) for $\tn{Lin}_{\W_2}$ and $\tn{Lin}_{\HK}$ embeddings. Samples with identical size parameter $p_2$ are connected by lines. For $\HK$ a two-dimensional grid structure emerges, one mode corresponding to elongation and the other to size. For $\W_2$ the size variation dominates the embedding. Samples are roughly located on a one-dimensional curve, according to their size variation parameter $p_2$. Along this line, the samples are grouped into small clusters, each corresponding to one `pass' through the elongation parameter $p_1$ for fixed $p_2$.}
	\label{fig:NumericsEllipsesPCA}
\end{figure}

\paragraph{Principal component analysis}
To explore the dataset structure, we apply principal component analysis (PCA) to both ensembles of linear embeddings.
	Intuitively, if we pick the barycenter as support point for the linearization, the linear embeddings should already be centered in the tangent space.
	For the $\W_2$ metric this is known to be true \cite[Equation (3.10)]{WassersteinBarycenter}. For the $\HK$ metric we are not aware of such a result, but numerically it seems to be satisfied.
	If we pick another reference measure, we need to center the samples before applying PCA.
	
	The coordinates of the (centered) linear embeddings with respect to the two dominant PCA modes are shown in Figure \ref{fig:NumericsEllipsesPCA}.
	For the $\HK$ metric, we recover a two-dimensional grid structure that corresponds precisely two the two underlying parameters of the dataset.
	For the $\W_2$ metric, the coordinates with respect to the first two principal components are dominated by the size variation. The samples lie approximately on a one-dimensional curve, according to their size variation parameter $p_2$ with elongation variations only causing small perturbations near this curve.
	For $\HK$ the first two modes capture $95\%$ of the dataset variance, for $\W_2$ only $78\%$.
	Extracting information about the elongation will thus be decidedly more difficult from the $\W_2$ embedding.

\begin{figure}
	\centering
	\includegraphics[width=\textwidth]{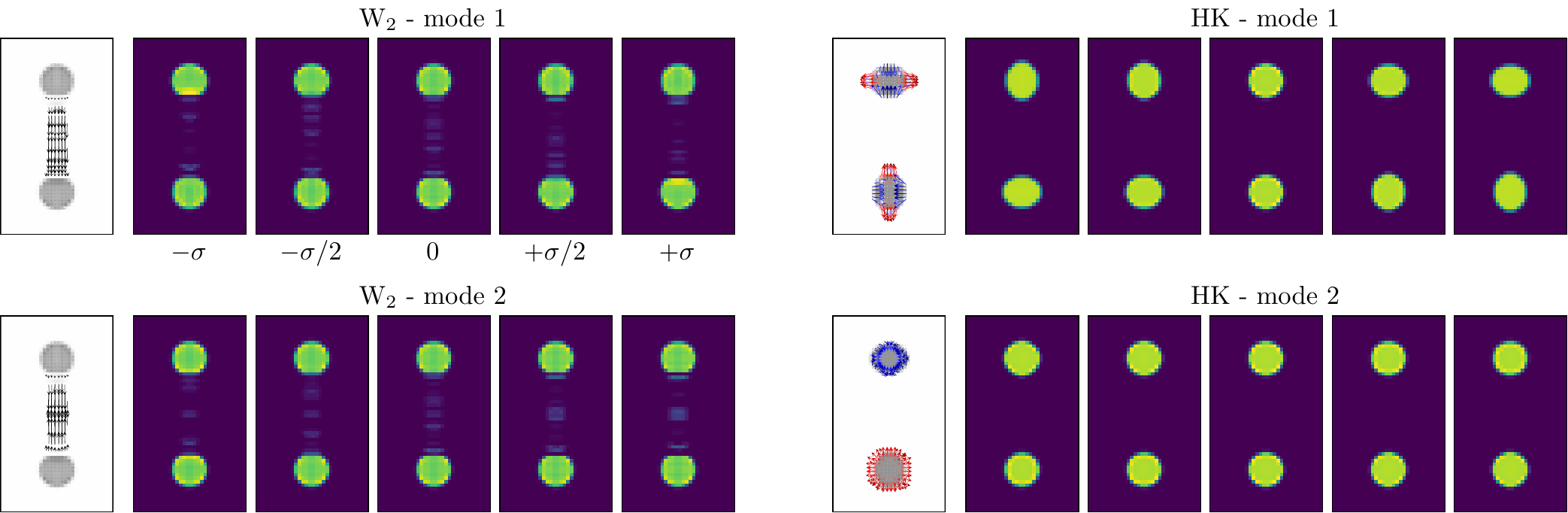}
	\caption{Visualizations of the two dominant PCA modes for linearized $\W_2$ and $\HK$ embeddings of the ellipse dataset. For each mode, the quiver plot on the left shows the initial velocity field $v_{0}$, for $\HK$ the color of the arrows encodes $\alpha_{0}$ (blue means decrease, red increase of mass).
	The five images on the right visualize the exponential map evaluated between $-\sigma$ and $\sigma$ where $\sigma$ denotes the standard deviation along the considered mode. $\HK$ accurately captures the two-dimensional structure of the dataset, $\W_2$ is dominated by the size variations.}
	\label{fig:NumericsEllipsesModes}
\end{figure}
	
	In Figure \ref{fig:NumericsEllipsesModes} the first two principal components for both linearizations are visualized as (colored) quiver plots and as curves of measures generated via the exponential map.
	In agreement with the previous observations, for $\W_2$ the first two modes seem to be concerned mostly with moving mass between the two disks, whereas for for $\HK$ the first mode clearly encodes variations in the disks elongation and the second mode encodes variations in their size.

	For the linearized $\W_2$ embedding the picture is qualitatively very similar if we pick as reference measure $\mu_0$ instead the $L^2$ or Hellinger mean of the samples, or a uniform measure on $\Omega$: the size variation shadows the elongation variation.
	For linearized $\HK$ the situation remains essentially unchanged with $L^2$ or Hellinger mean. For the uniform measure the mass variation is too subtle and is shadowed by the elongation variation. For stronger size variations a similar picture as above re-emerges.

	Of course, for more realistic datasets it cannot be expected that PCA will yield as transparent and simple results as for this toy example.
	But based on these observations we may still hope that the ability to locally vary the particle masses will make the linear embeddings via $\Log_{\HK}$ more robust to mass fluctuations and consequently simplify any subsequent analysis task, such as classification.
	We confirm this with the two subsequent examples.

\subsection{Cell morphometry}
\label{sec:NumericsCells}

\paragraph{Motivation and set up}
Our second example is an application to cell morphometry that appeared in~\cite{OptimalTransportTangent2012}.
There are 500 images of liver cells, of which 250 are cancerous and 250 are healthy.
The images were initially $192\times 192$ which were centered and reorientated, cropped to $128\times 128$ and then resized to $64\times 64$.
We used the $\Lp{2}$ barycenter as the reference image for linearisations of both $\W_2$ and $\HK$.
Three example images and the reference image are shown in Figure~\ref{fig:CellsExample}.
We set the scale parameter $\kappa$ (see Remark~\ref{rem:HKLengthScale}) to $1.5\%$, $15\%$ and $150\%$ of the size of the image, i.e.~in the $64\times 64$ images $\kappa$ was set to 1, 10 and 100 respectively.

\begin{figure}[hbt]
\centering
{\def\figw{2cm}
\begin{tikzpicture}[x=\figw,every node/.style={inner sep=0pt,anchor=north west,draw=black,line width=0.5pt}]
	\node at (0,0) []{\includegraphics[width=\figw]{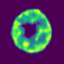}};
	\node at (1.2,0) []{\includegraphics[width=\figw]{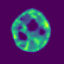}};
	\node at (2.4,0) []{\includegraphics[width=\figw]{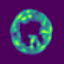}};
	\node at (3.6,0) []{\includegraphics[width=\figw]{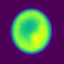}};
\end{tikzpicture}
}
\caption{Three exemplar images from the cell morphometry data set in Example~2 and the $\Lp{2}$ barycenter (right).}
\label{fig:CellsExample}
\end{figure}

\paragraph{Principal component analysis}
Following the methodology outlined in the previous sections we compute the linear Wasserstein and linear Hellinger--Kantorovich embeddings.
PCA shows one principal component dominates in all linear embeddings.
There is a gradual decrease from the second principal component onwards in the $\W_2$ embedding, whilst we see another drop between the third and fourth principal components in the HK embedding for $\kappa=10$.
The top three eigenmodes in the linear $\W_2$ and $\HK_{100}$ embeddings account for 79\% of the variance, which is more than the 75\% of the variance which the top three eigenmodes in the linear $\HK_{10}$ embedding account for, and significantly more than the 49\% of the variance the $\HK_1$ embedding accounted for.
All linear embeddings show that healthy samples have very little variation in the 2D projection (see Figure~\ref{fig:CellsPCA}).

We visualize the eigenvectors in Figure~\ref{fig:CellsEigShooting}.
The first three eigenvectors in the $\W_2$ and $\HK_{100}$ embeddings correspond (in order) to size, skewness and distribution of mass from the center to the boundary, which can be seen either from the vector fields or the shooting along eigenvectors in Figure~\ref{fig:CellsEigShooting}.
We see the same effects in the $\HK_1$ and $\HK_{10}$ eigenvectors but in a different order.
In the $\HK_1$ and $\HK_{100}$ embeddings the first two eigenvectors capture some skewness and mass transfer from the center to the boundary.
The third HK eigenvector captures the size of the cell.

\begin{figure}[hbt]
\centering

\includegraphics[]{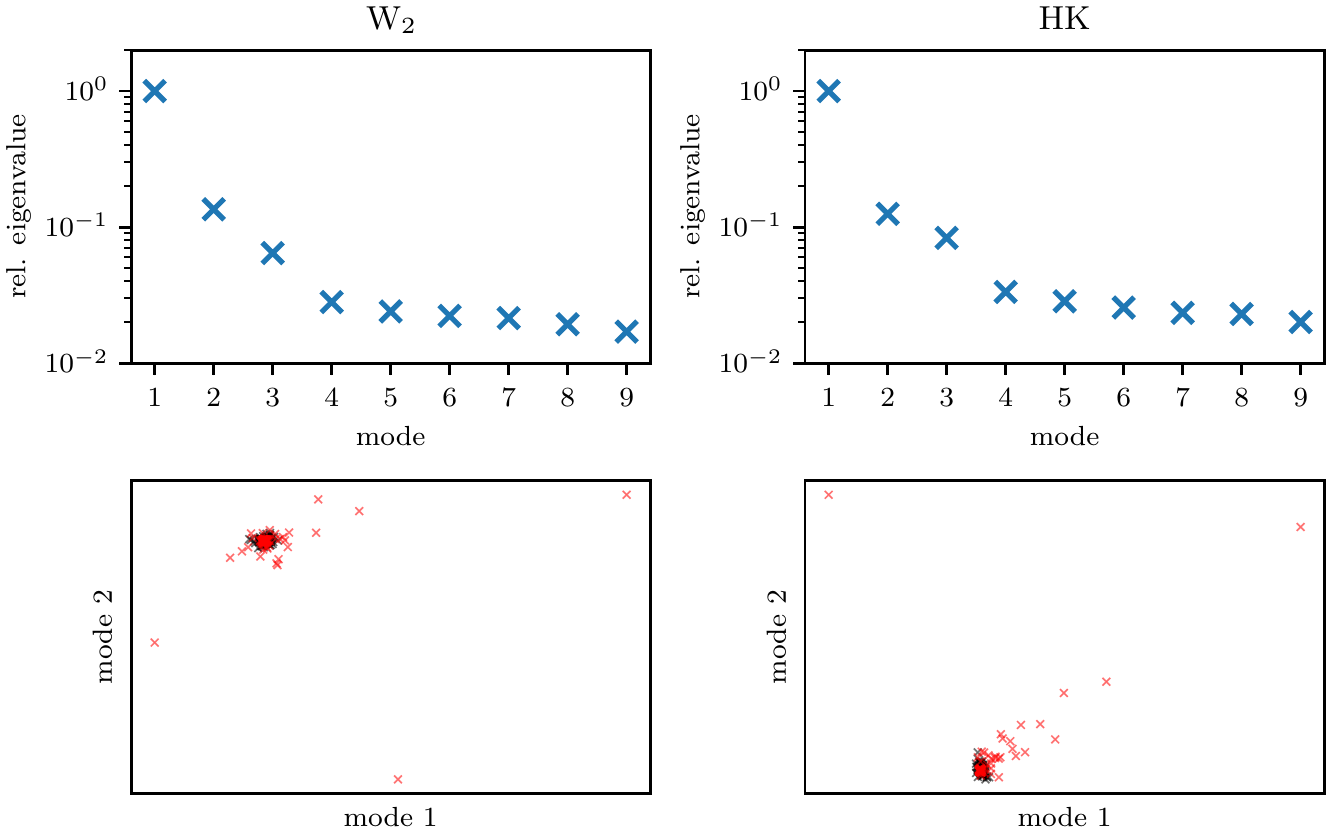}
\caption{Visualization of the PCA eigenvalues (divided by the largest eigenvalue) for linearized $\W_2$ (top left) and $\HK_{10}$ (top right) embeddings in Example~2.
Coordinates for the two dominant eigenmodes for each image (black healthy, red cancerous) in the bottom row.
Figures for $\kappa=1$ and $\kappa=100$ are omitted but are qualitatively similar to $\W_2$ and $\HK_{10}$ plots.}
\label{fig:CellsPCA}
\end{figure}

\begin{figure}[hbt]
\centering
\def\figw{1.1cm}
\begin{tikzpicture}[
				x=\figw,y=\figw,
				img/.style={inner sep=0pt,draw=black,line width=1pt,anchor=north west},
				lbl/.style={anchor=base},
	]
	
\draw (7.1,0.4) -- (7.1,-11.1);	
\draw (0,-5.3) -- (14.2,-5.3);

	\node[img] at (0,0) []{\includegraphics[width=\figw]{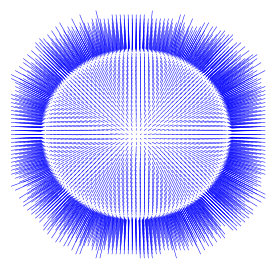}};
	\begin{scope}[shift={(1.2,0)}]
	\node[lbl] at (1.1*2+0.5,0.1) []{$\W_2$ - mode 1};
	\foreach \x/\l in {0/{$-\sigma$},1/{$-\sigma/2$},2/{$0$},3/{$+\sigma/2$},4/{$+\sigma$}} {
		\node[img] at (1.1*\x,0) {\includegraphics[width=\figw]{W2Shooting0_00\x.png}};
		\node[lbl] at (1.1*\x+0.5,-1.3*\figw) {\l};
	}
	\end{scope}	

	\begin{scope}[shift={(0,-1.8*\figw)}]
	\node[img] at (0,0) []{\includegraphics[width=\figw]{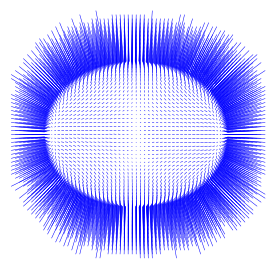}};
	\begin{scope}[shift={(1.2,0)}]
	\node[lbl] at (1.1*2+0.5,0.1) []{$\W_2$ - mode 2};
	\foreach \x in {0,1,2,3,4} {
		\node[img] at (1.1*\x,0) []{\includegraphics[width=\figw]{W2Shooting1_00\x.png}};
	}
	\end{scope}	
	\end{scope}	
	
	\begin{scope}[shift={(0,-3.6*\figw)}]
	\node[img] at (0,0) []{\includegraphics[width=\figw]{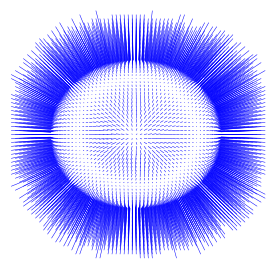}};
	\begin{scope}[shift={(1.2,0)}]
	\node[lbl] at (1.1*2+0.5,0.1) []{$\W_2$ - mode 3};
	\foreach \x in {0,1,2,3,4} {
		\node[img] at (1.1*\x,0) []{\includegraphics[width=\figw]{W2Shooting2_00\x.png}};
	}
	\end{scope}	
	\end{scope}

	\begin{scope}[shift={(7.6,0)}]
	\node[img] at (0,0) []{\includegraphics[width=\figw]{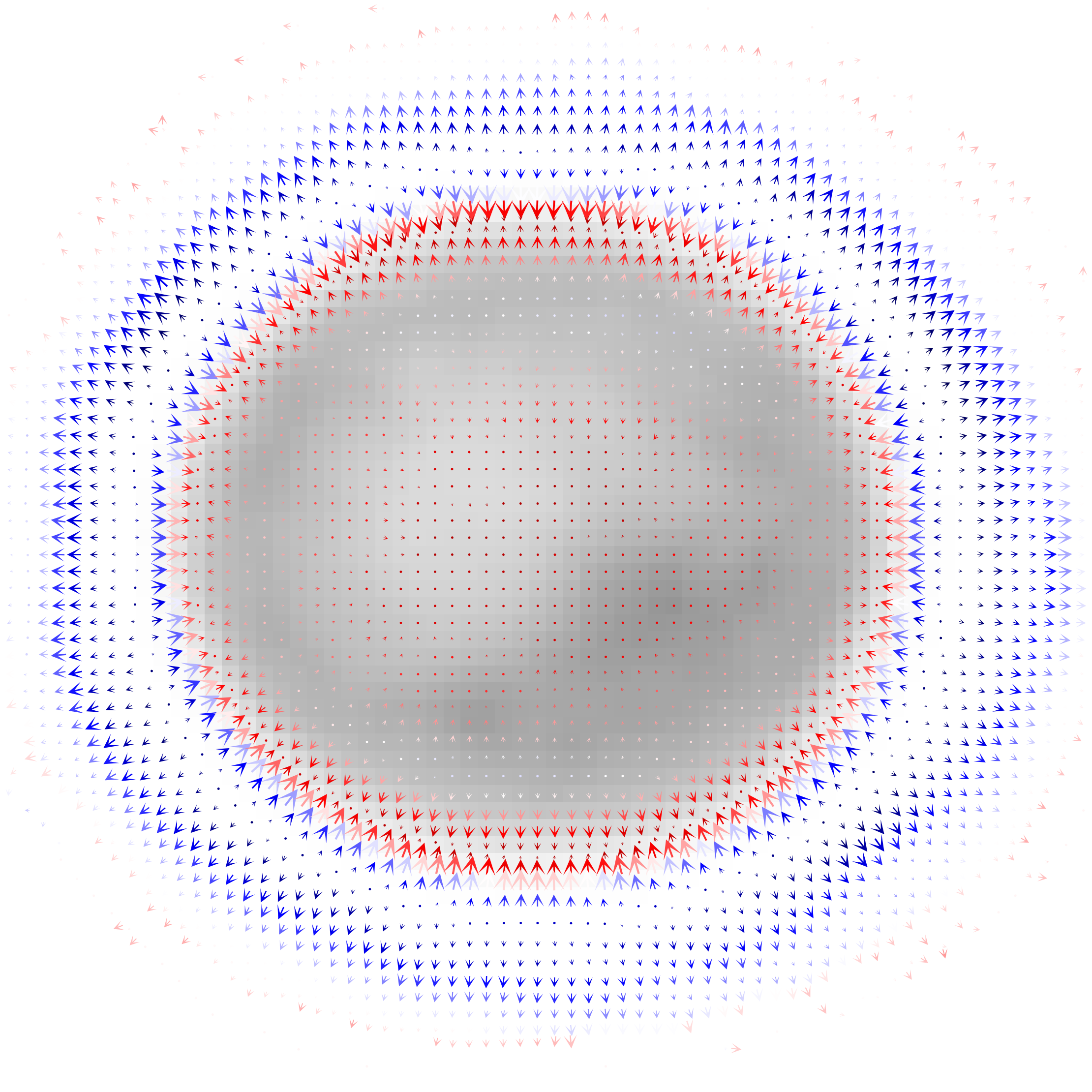}};
	\begin{scope}[shift={(1.2,0)}]
	\node[lbl] at (1.1*2+0.5,0.1) []{$\HK_{1}$ - mode 1};
	\foreach \x in {0,1,2,3,4} {
		\node[img] at (1.1*\x,0) []{\includegraphics[width=\figw]{HKShooting_kappa1_0_00\x.png}};
	}
	\end{scope}	
	\end{scope}	
	
	\begin{scope}[shift={(7.6,-1.8*\figw)}]
	\node[img] at (0,0) []{\includegraphics[width=\figw]{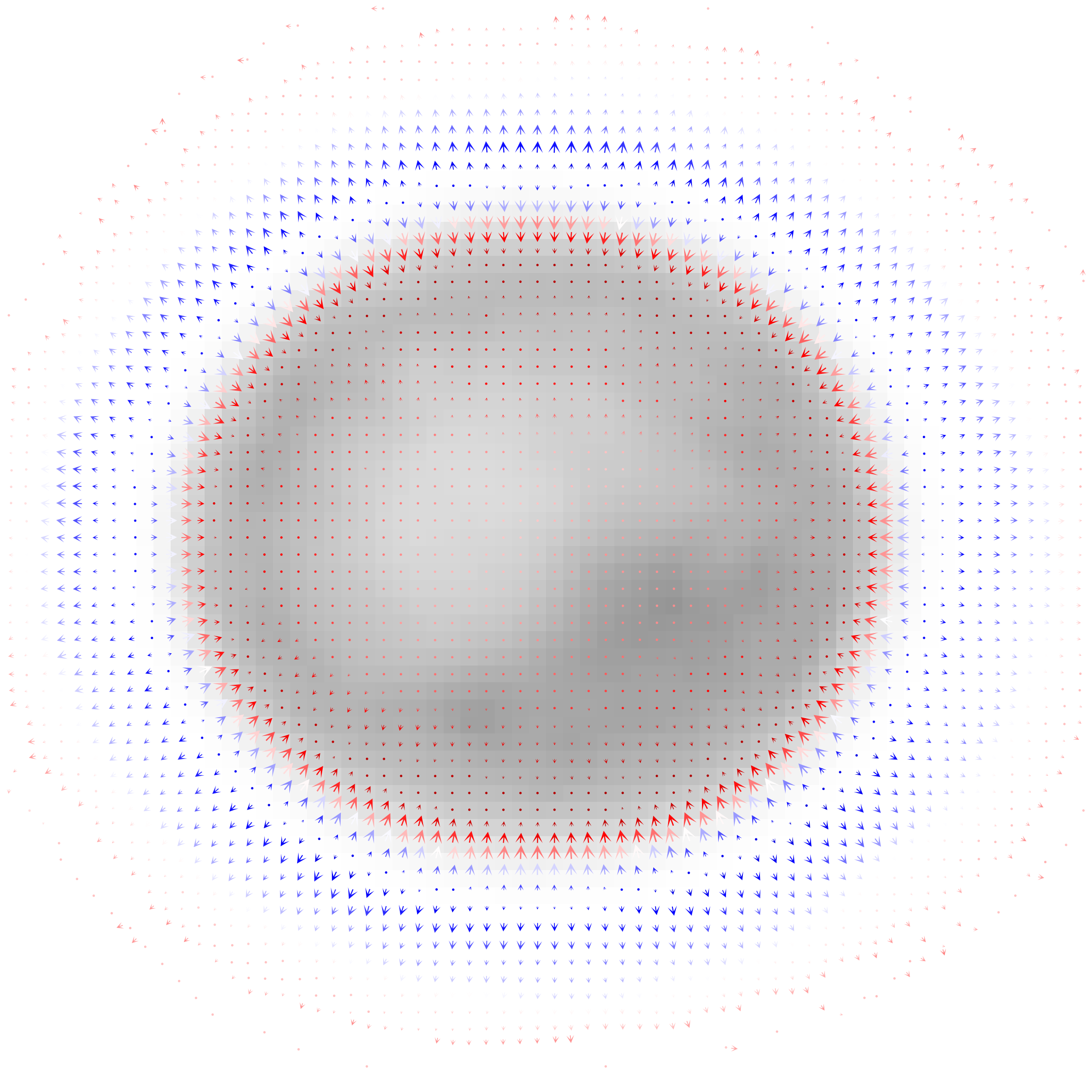}};
	\begin{scope}[shift={(1.2,0)}]
	\node[lbl] at (1.1*2+0.5,0.1) []{$\HK_{1}$ - mode 2};
	\foreach \x in {0,1,2,3,4} {
		\node[img] at (1.1*\x,0) []{\includegraphics[width=\figw]{HKShooting_kappa1_1_00\x.png}};
	}
	\end{scope}	
	\end{scope}	

	\begin{scope}[shift={(7.6,-3.6*\figw)}]
	\node[img] at (0,0) []{\includegraphics[width=\figw]{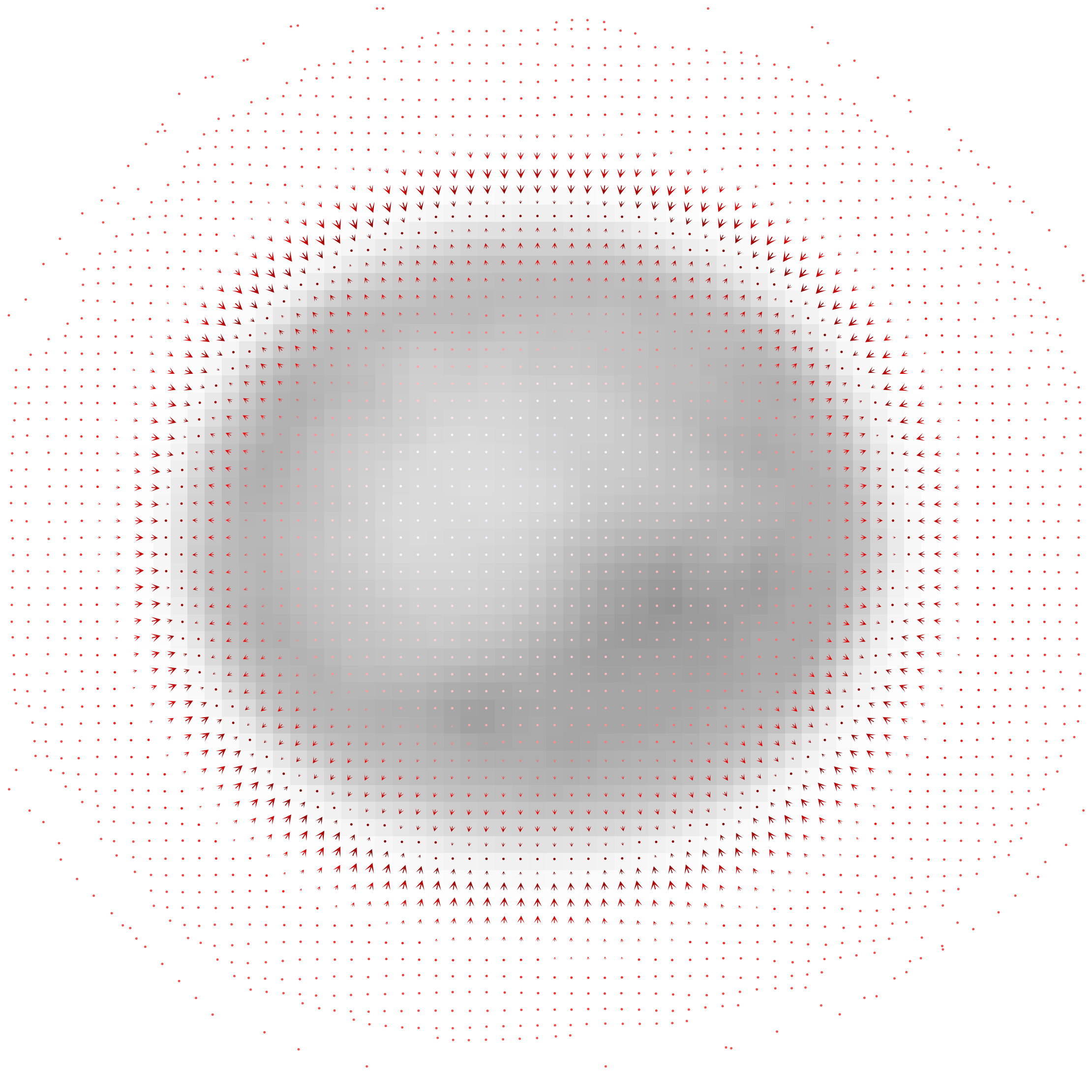}};
	\begin{scope}[shift={(1.2,0)}]
	\node[lbl] at (1.1*2+0.5,0.1) []{$\HK_{1}$ - mode 3};
	\foreach \x in {0,1,2,3,4} {
		\node[img] at (1.1*\x,0) []{\includegraphics[width=\figw]{HKShooting_kappa1_2_00\x.png}};
	}
	\end{scope}	
	\end{scope}

	\begin{scope}[shift={(0,-6.4*\figw)}]
	\node[img] at (0,0) []{\includegraphics[width=\figw]{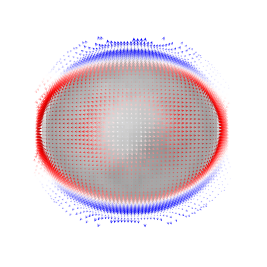}};
	\begin{scope}[shift={(1.2,0)}]
	\node[lbl] at (1.1*2+0.5,0.1) []{$\HK_{10}$ - mode 1};
	\foreach \x in {0,1,2,3,4} {
		\node[img] at (1.1*\x,0) []{\includegraphics[width=\figw]{HKShooting0_00\x.png}};
	}
	\end{scope}	
	\end{scope}	
	
	\begin{scope}[shift={(0,-8.2*\figw)}]
	\node[img] at (0,0) []{\includegraphics[width=\figw]{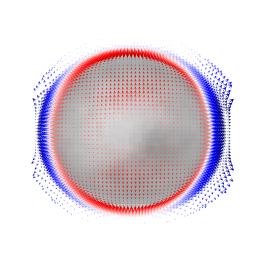}};
	\begin{scope}[shift={(1.2,0)}]
	\node[lbl] at (1.1*2+0.5,0.1) []{$\HK_{10}$ - mode 2};
	\foreach \x in {0,1,2,3,4} {
		\node[img] at (1.1*\x,0) []{\includegraphics[width=\figw]{HKShooting1_00\x.png}};
	}
	\end{scope}	
	\end{scope}	

	\begin{scope}[shift={(0,-10.0*\figw)}]
	\node[img] at (0,0) []{\includegraphics[width=\figw]{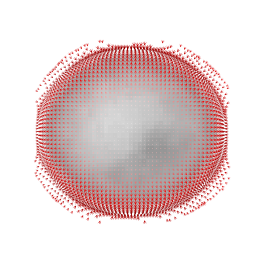}};
	\begin{scope}[shift={(1.2,0)}]
	\node[lbl] at (1.1*2+0.5,0.1) []{$\HK_{10}$ - mode 3};
	\foreach \x in {0,1,2,3,4} {
		\node[img] at (1.1*\x,0) []{\includegraphics[width=\figw]{HKShooting2_00\x.png}};
	}
	\end{scope}	
	\end{scope}

	\begin{scope}[shift={(7.6,-6.4)}]
	\node[img] at (0,0) []{\includegraphics[width=\figw]{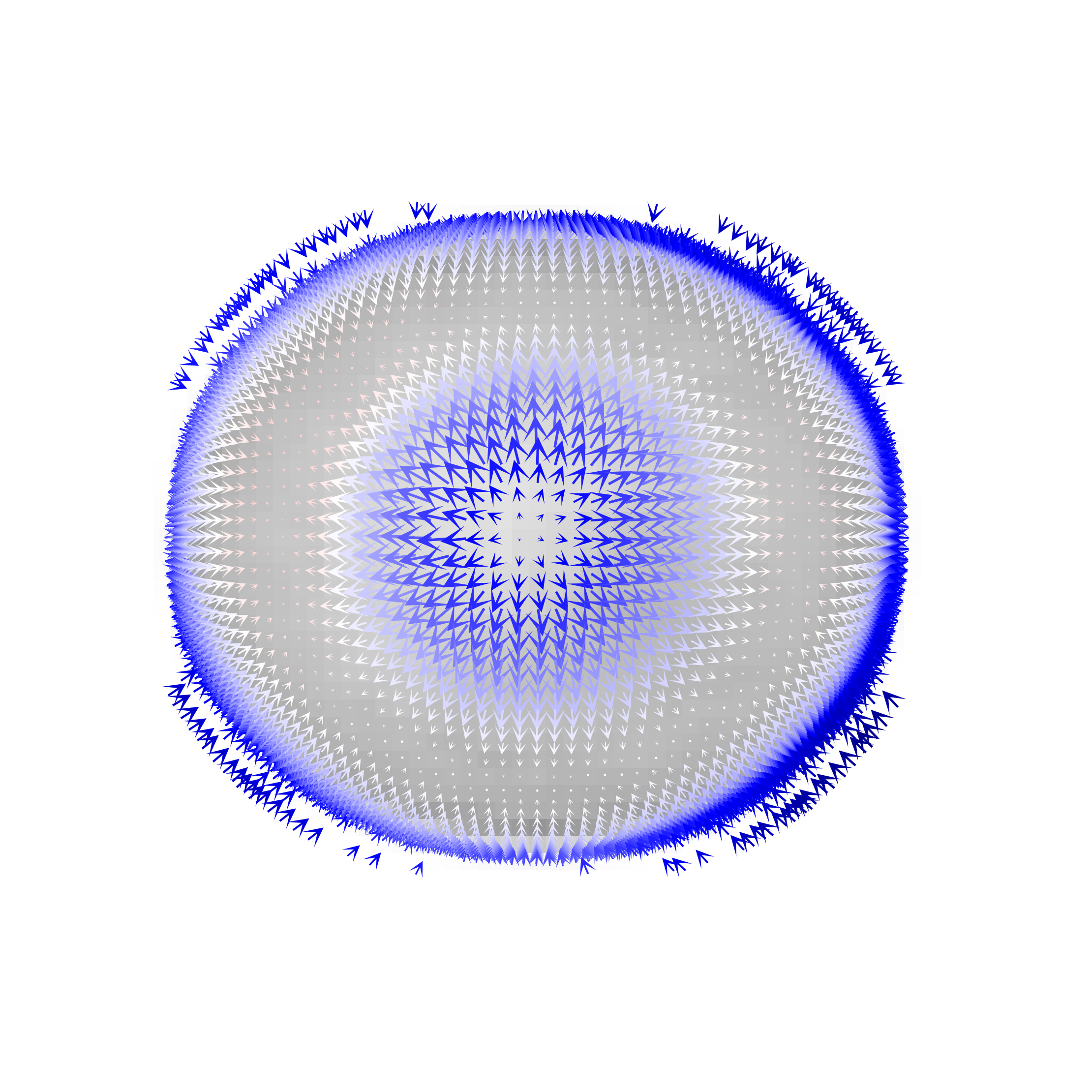}}; 
	\begin{scope}[shift={(1.2,0)}]
	\node[lbl] at (1.1*2+0.5,0.1) []{$\HK_{100}$ - mode 1};
	\foreach \x in {0,1,2,3,4} {
		\node[img] at (1.1*\x,0) []{\includegraphics[width=\figw]{HKShooting_kappa100_0_00\x.png}}; 
	}
	\end{scope}	
	\end{scope}	
	
	\begin{scope}[shift={(7.6,-8.2*\figw)}]
	\node[img] at (0,0) []{\includegraphics[width=\figw]{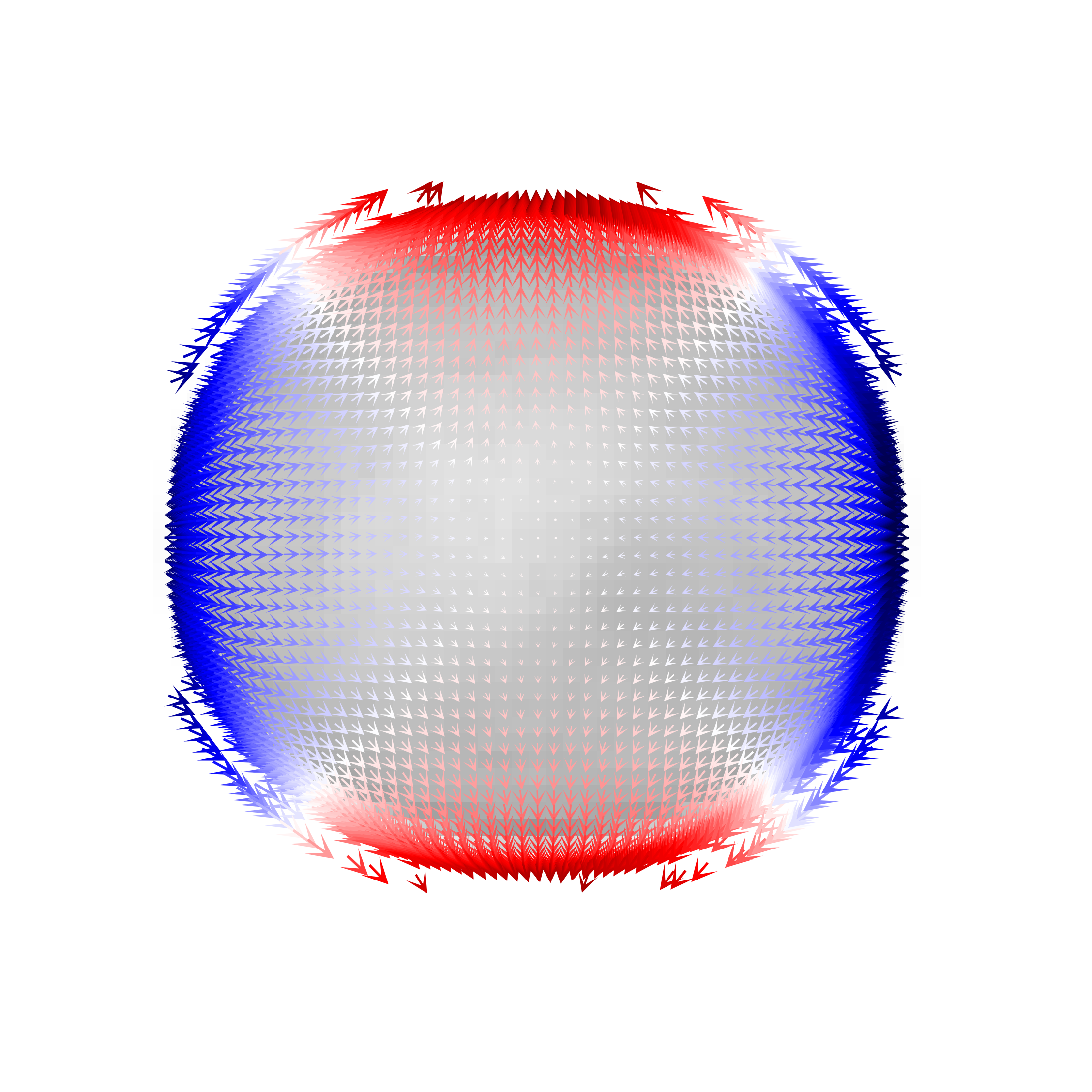}}; 
	\begin{scope}[shift={(1.2,0)}]
	\node[lbl] at (1.1*2+0.5,0.1) []{$\HK_{100}$ - mode 2};
	\foreach \x in {0,1,2,3,4} {
		\node[img] at (1.1*\x,0) []{\includegraphics[width=\figw]{HKShooting_kappa100_1_00\x.png}}; 
	}
	\end{scope}	
	\end{scope}	

	\begin{scope}[shift={(7.6,-10.0*\figw)}]
	\node[img] at (0,0) []{\includegraphics[width=\figw]{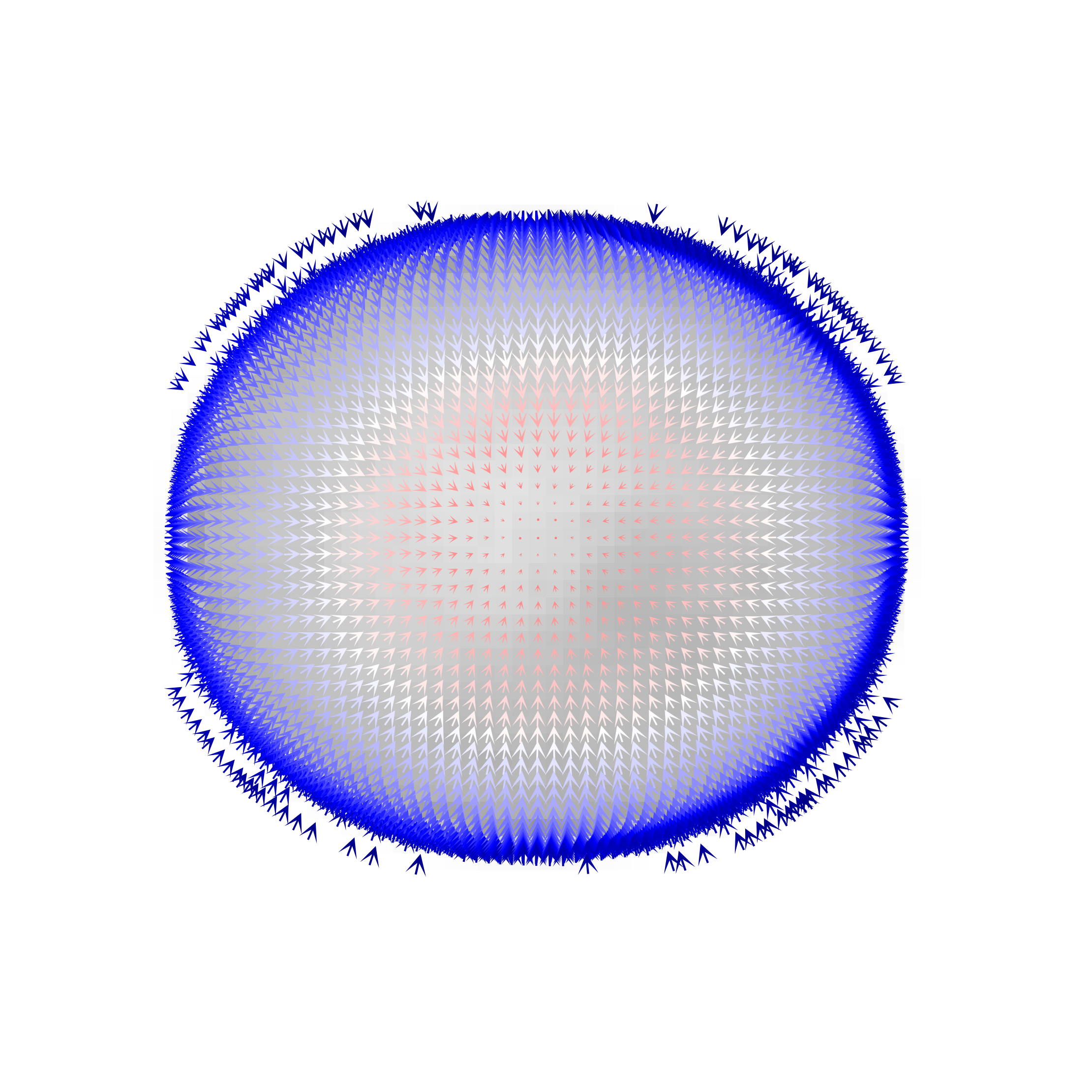}}; 
	\begin{scope}[shift={(1.2,0)}]
	\node[lbl] at (1.1*2+0.5,0.1) []{$\HK_{100}$ - mode 3};
	\foreach \x in {0,1,2,3,4} {
		\node[img] at (1.1*\x,0) []{\includegraphics[width=\figw]{HKShooting_kappa100_2_00\x.png}}; 
	}
	\end{scope}	
	\end{scope}	

\end{tikzpicture}

\caption{Variation along the first three eigenvectors in Example~2 for the $\W_2$ linear embedding (top left) and the $\HK$ linear embeddings.}
\label{fig:CellsEigShooting}
\end{figure}

\paragraph{Classification accuracy}
To test how the geometry of the data reflects the label of interest (whether the cell is cancerous or not) we implement a very simple 1NN classifier.
We test the classification accuracy based on the nearest neighbour using a leave-one-out (or equivalently 500-fold cross validation) strategy.
That is, for each image we find the label of the image of the nearest neighbour (using either the linear Wasserstein or linear Hellinger--Kantorovich embedding).
We also add results using $\HK_{1000}$ which, for $\kappa$ so large is approximately the Wasserstein distance.
We give the accuracy, false positive rate and true positive rate in Table~\ref{tab:CellsClassifierResults}; in terms of accuracy and true positive we see that the best performing HK linearisation outperforms the W2 linearisation, and in terms of false positive rate the best performing HK linearisations are essentially tied with the W2 linearisation.

\begin{table}[hbt]
	\centering
	\caption{Results for healthy vs cancerous using a 1NN classifier in the linear spaces. TPR is the true positive rate (also called sensitivity) and FPR is the false positive rate.}
	\label{tab:CellsClassifierResults}
		{\scriptsize
		\begin{tabular}{|c|c|c|c|c|c|}
			\hline
			 & $\HK_1$ & $\HK_{10}$ & $\HK_{100}$ & $\HK_{1000}$ & $\W_2$ \\
			\hline \hline
			\textbf{accuracy} & \textbf{0.654} & \textbf{0.686} & \textbf{0.670} & \textbf{0.672} & \textbf{0.668} \\
			TPR & 0.728 & 0.704 & 0.660 & 0.660 & 0.656 \\
			FPR & 0.420 & 0.332 & 0.320 & 0.316 & 0.320 \\
			\hline
		\end{tabular}
		}
\end{table}
\subsection{Collider events}
\label{sec:NumericsJets}

\paragraph{Motivation and setup} Our third example extends the analysis in~\cite{Cai_2020} on the application of OT to jet physics. Here the central goal is to extract physics from collider data which after appropriate reconstruction can be presented as images. A salient feature in those images is the presence of jets---collimated sprays of strongly-interacting particles bearing the imprint of the underlying collision. It was shown that (balanced) OT induces a physically meaningful metric on the space of collider events useful for jet tagging \cite{ColliderMetric2019}, a task that tries to ascertain the jet origin and thus the nature of the collision. Linearized $\W_2$ has been utilized to reduce computational complexity and simplify later analyses by offering a Euclidean embedding \cite{Cai_2020}. Here we show that this can be further improved by using linearized unbalanced $\HK$. 

We focus on distinguishing boosted W boson jets from QCD (quark or gluon) background jets, though the same analysis has been extended to other pairwise tagging tasks with qualitatively similar results obtained.
We use the same simulated dataset as in~\cite{Cai_2020}, with 10k jets in total, about half of which are W jets. See \cite[Section 3]{Cai_2020} for more details on data generation and preprocessing. 

A jet can be seen as a discrete measure in the rapidity-azimuth ($y-\phi$) plane, i.e.~$\Omega \subset \R^2$. Mass then corresponds to the transverse momentum $p_T$ of its constitute particle. As an illustration, the upper left plot of Figure~\ref{fig:EventDisplay} shows a typical QCD jet in green, with the size of the points proportional to the particles' $p_T$. For the reference measure we again pick a regular Cartesian grid of $15 \times 15$ points covering the rectangle $[-1.7,1.7] \times [-\pi/2,\pi/2]$ with uniform $p_T$ distribution; see the blue dots in the same figure. As in previous examples, we normalize all jets to unit $p_T$ before analysis (see Remark \ref{rem:HKMassScaling}).

\paragraph{Choice of $\HK$ length scale $\kappa$}
When using $\HK$, the choice of the length scale parameter $\kappa$ is crucial for the success of the analysis (Remark \ref{rem:HKLengthScale}). We scan values over several orders of magnitude, i.e.~$\kappa \in [0.01,100]$. Figure~\ref{fig:EventDisplay} displays various optimal transports from the uniform reference measure to a QCD and a W jet. Considering the length scale of the sample and reference measures, we expect that $\kappa=100$ is very close to balanced $\W_2$ and $\kappa=0.01$ essentially behaves like the Hellinger distance. In particular, for the latter the assumption that $\mu_{i}^\perp=0$ is far from true, since the maximal transport distance is substantially smaller than the grid spacing of the reference jet ($\approx 0.2$). Therefore, virtually no transport happens and almost all mass is being created and destroyed (which is why $\kappa=0.01$ was excluded from Figure~\ref{fig:EventDisplay}). Consequently, the classification performance clearly deteriorates for $\kappa=0.01$. We find that $\kappa$ between 0.1 and 1 performs best, which is the regime where both transport and mass creation/destruction are relevant. See Table~\ref{tab:NumericsClassifierResults} and the discussion in \textit{Classification}. 

\begin{figure}[hbt]
	\centering
	\begin{tabular}{lr}
		\includegraphics[width=0.95\textwidth]{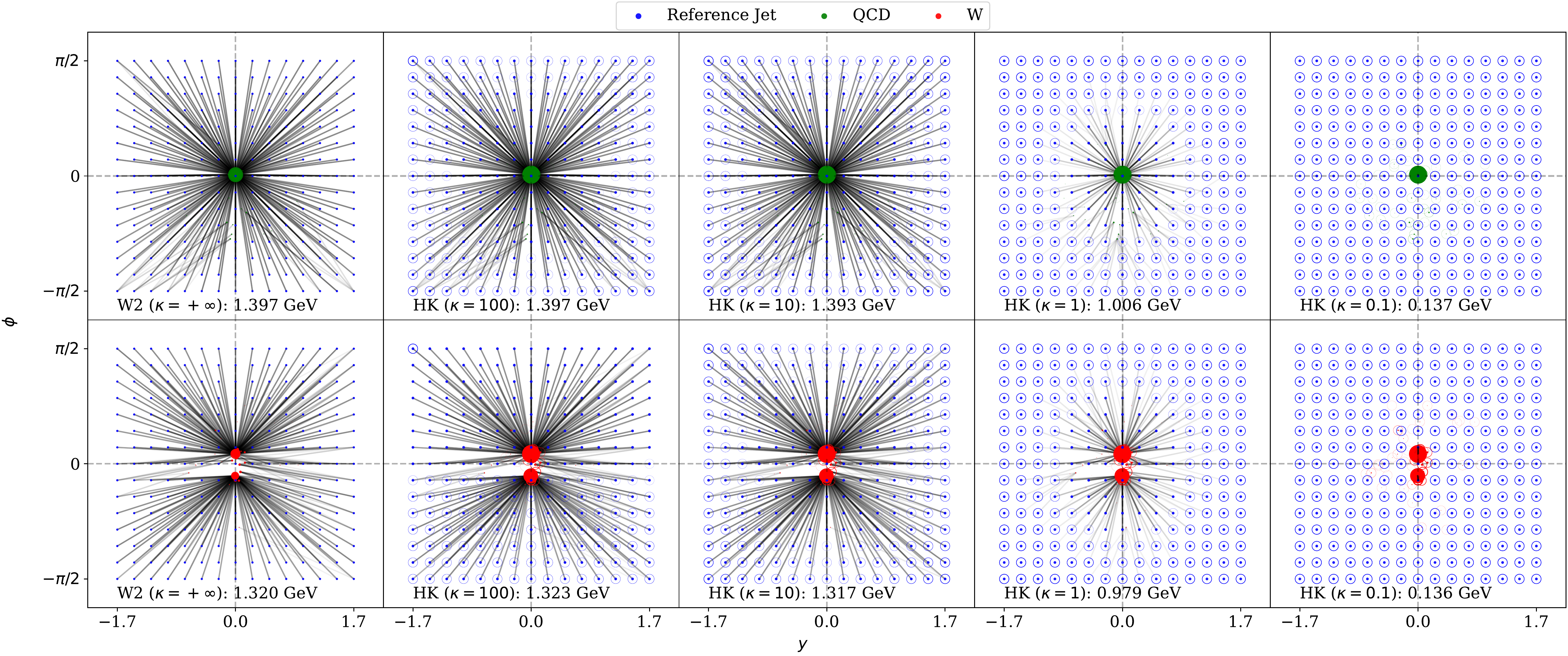} 
	\end{tabular}
	\caption{Optimal transports between the uniform reference measure (blue) and a typical QCD jet (green, first row), or a W jet (red, second row), using $\W_2$ and $\HK$ with $\kappa = 100, 10, 1, 0.1$ (from left to right with $\kappa=+\infty$ denoting balanced $\W_2$). The darkness of the lines indicates how much $p_T$ is moved from one particle to another. For $\HK$, the thickness of the circles around the points represents how much $p_T$ is destroyed for that particular particle. Shown at the bottom of the plots are the total OT distances between the jets, which are similar for $\kappa=+\infty, 100, 10$, the transport regime.}
	\label{fig:EventDisplay}
\end{figure}

\paragraph{Principal component analysis and Linear discriminant analysis}
As before, we start with a principal component analysis. Compared to the previous datasets, this one exhibits a high intrinsic dimensionality. Approximately 30 modes are needed to capture at least $90\%$ of the dataset variance ($\W_2$: 27; $\HK_{\kappa=10}$: 28; $\HK_{\kappa=1}$: 38; $\HK_{\kappa=0.1}$: 26). Moreover, for all $\kappa$ (including $\kappa=\infty$) the mean of the sample point cloud is far from any sample, which indicates that the samples lie on a submanifold of the tangent space with non-trivial shape and topology. Therefore we do not find it surprising that applying the exponential map to individual dominant modes relative to the mean, or projecting samples to very few modes, did not yield artificial jet images useful for physical interpretation.
For $\kappa=1$ the distribution of the dataset in the tangent space with respect to the first two dominant modes is visualized in Figure~\ref{fig:PCA}. We observe that the two classes are discernible as distinct clusters with relatively little overlap and that the two coordinate values are highly dependent.
In addition, for several points in the PCA coefficient space we show the jets that are closest to those coefficients and approximations by the exponential map, indicating that variations of the PCA coefficients correspond to physically meaningful changes of the jets that can locally be approximated linearly via the exponential map.
For other values of $\kappa \in [0.1,\infty]$ PCA yields qualitatively similar results.

\begin{figure}
	\centering
	\includegraphics[height=4cm]{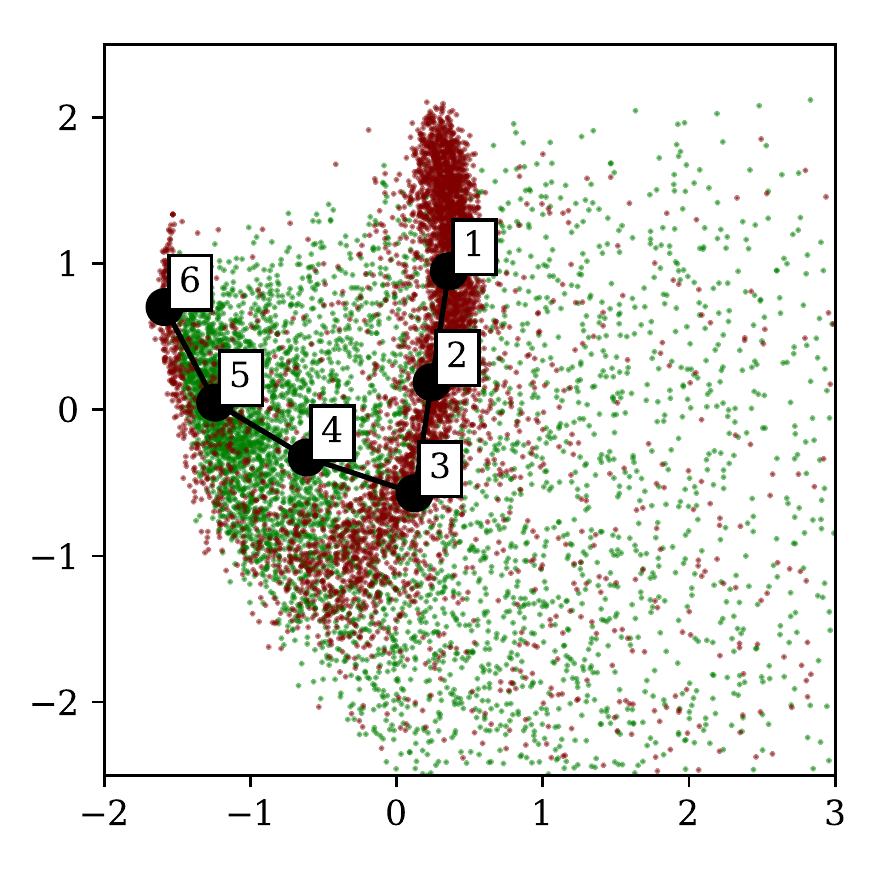}
	\hspace{1cm}
	\includegraphics[height=4cm]{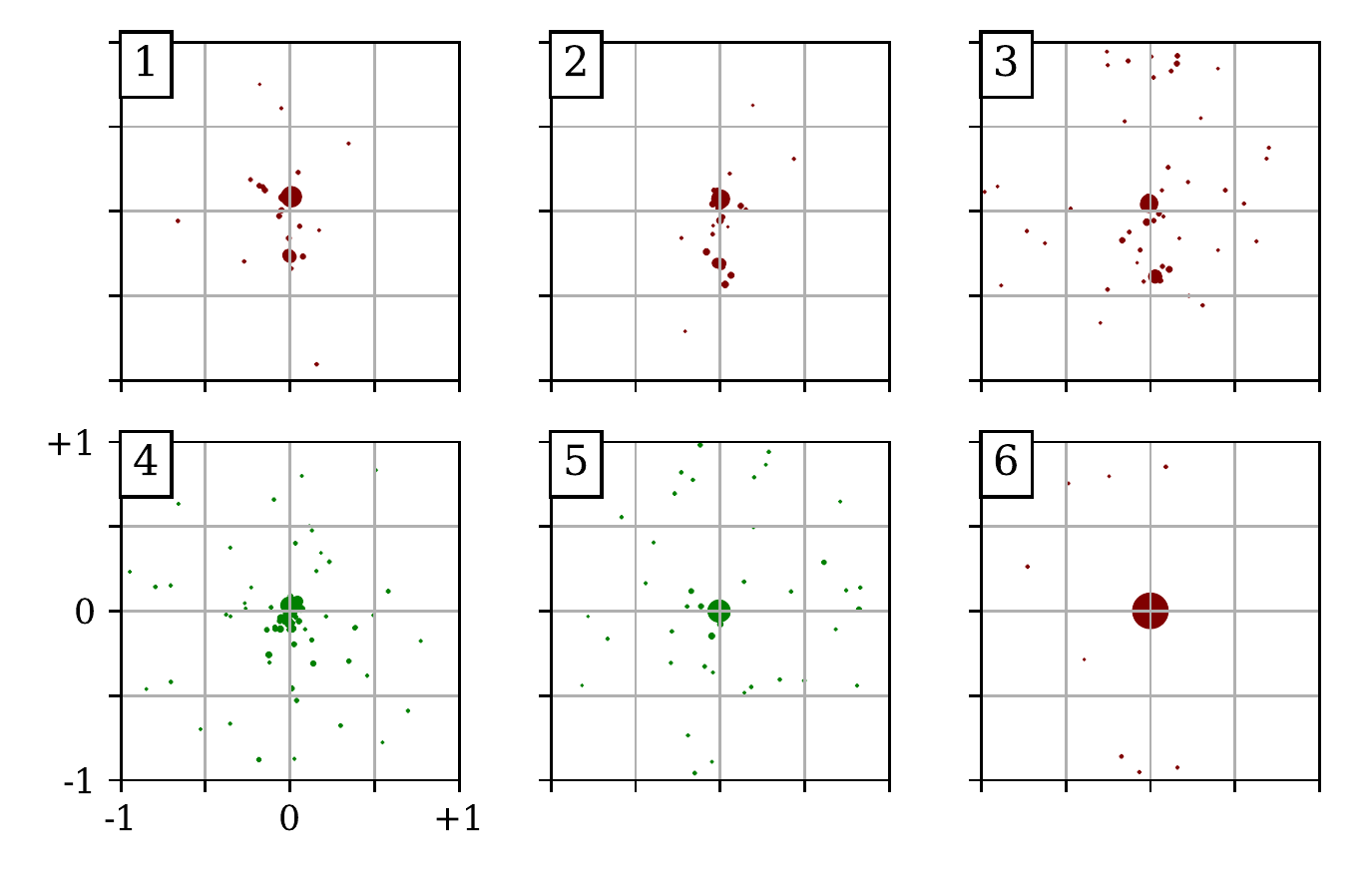}
	\hspace{1cm}
	\includegraphics[height=4cm]{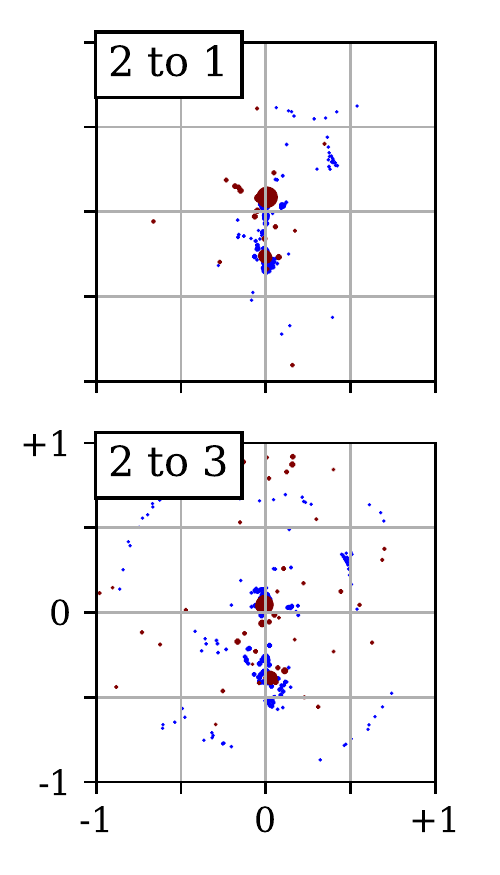}
	\caption{Left, distribution of 10k W (red) and QCD (green) jets in the tangent space with respect to the first two dominant PCA modes for and $\HK$ with $\kappa=1$ with some positions in coefficient space marked. %
	Middle, jets in $(y-\phi)$-plane with first two PCA coefficients closest to the marks, color indicating jet type. From mark 1-3 the lower prong is moving and becoming weaker, while background noise increases. Mark 4 shows a single prong with strong noise. At mark 6 there is a distinct cluster of W jets with a single, strongly focused prong. %
	Right, approximating jets 1 and 3 (red) by the exponential map (blue) from jet 2, using the difference of the first two PCA coefficients as tangential direction. While this cannot account for the considerable variation that lies orthogonal to this 2d plane, it correctly describes the movement of the lower prong and the increase in peripheral noise.}
	\label{fig:PCA}
\end{figure}

Next, we apply linear discriminant analysis (LDA). LDA assumes that samples from both classes are drawn from two Gaussian distributions with different means but identical covariance matrices. One then infers the hyperplane that optimally separates the two classes in a Bayesian sense. Let the hyperplane be parametrized by a unit normal vector $t$ and a point $z$ in the plane. (Here and in the following we use general vector space notation as our discussion will apply to both the linearized $\W_2$ and $\HK$ metrics.) On the one hand, LDA serves as a simple linear classifier where samples are labeled according to which side of the hyperplane they lie in, i.e.~the predicted class of sample $x_i$ depends on the sign of $\langle x_i-z,t \rangle$. On the other hand, we can analyze whether the direction $t$ has a physical meaning.

The first row in Figure~\ref{fig:LDA} shows the distribution of the projection coordinate $\langle x_i-z,t \rangle$ for $\W_2$ and $\HK$ with various $\kappa$. We find that for $\kappa=0.1$, $\HK$ best separates the two classes, which is confirmed by the superior performance of the LDA classifier; see Table~\ref{tab:NumericsClassifierResults}. Similarly to PCA, applying the exponential map to the discriminating direction $t$ relative to the sample mean did not yield physically valid results, presumably for the same reasons.
Instead, we visualize the direction in another way: for some $\lambda \in \R$ we find the sample $x_i$ such that $\langle x_i-z,t \rangle$ is closest to $\lambda$, i.e.~among all samples $x_i$ is closest to the hyperplane given by $\{x | \langle x-z, t \rangle = \lambda\}$. We vary $\lambda$ on the order $-3\sigma$ to $3\sigma$ where $\sigma$ denotes the standard variation of the samples along the direction $t$. This is shown in the lower two rows of Figure~\ref{fig:LDA} for $\W_2$ and $\HK_{\kappa=0.1}$ where we also record how many QCD and W jets there are in each $\lambda$ bin.
For both $\W_2$ and $\HK_{\kappa=0.1}$, the chosen jets transition from having a single mode to having two modes as $\lambda$ increases, suggesting that the direction $t$ successfully encodes the major topological difference between two-prong W jets and more diffuse single-prong QCD jets. A clearer separation is obtained for $\HK_{\kappa=0.1}$, whose class purity is slightly higher than $\W_2$ in each $\lambda$ bin.

\begin{figure}
	\centering
	\includegraphics[width=1.0\textwidth]{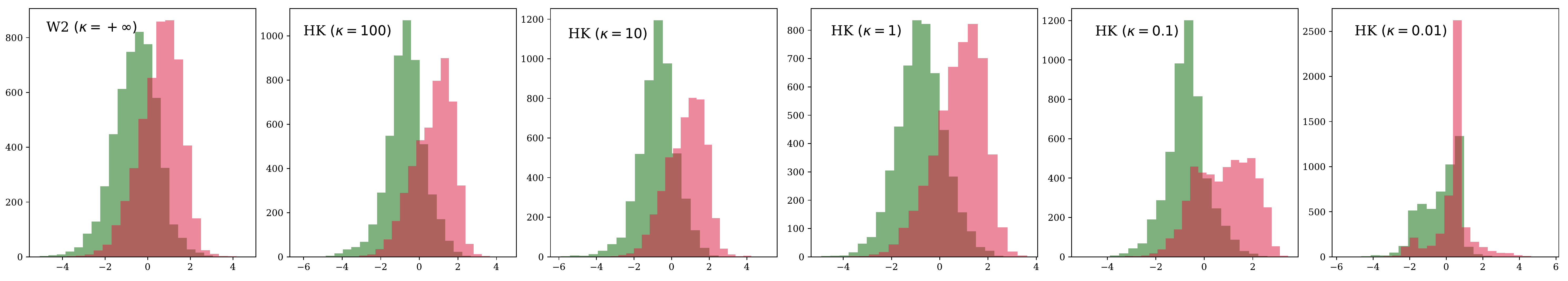}
	\includegraphics[width=1.0\textwidth]{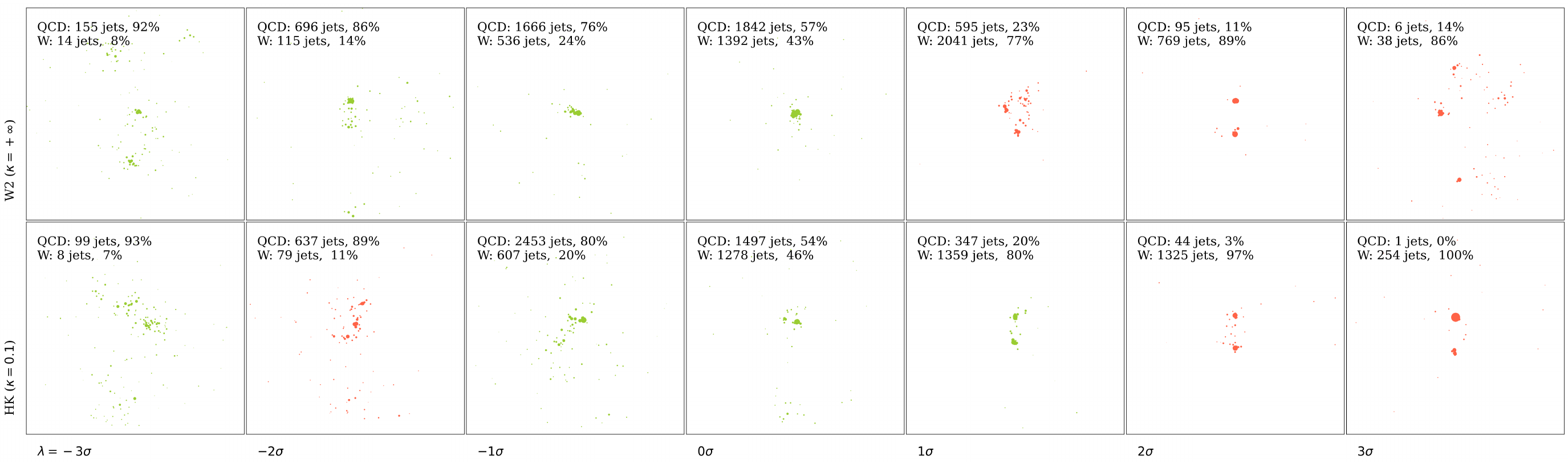}
	\caption{Upper: Histograms of the distribution of the LDA projection coordinate $\langle x_i-z,t \rangle$ with various OT distances. Lower two rows: Displays in the $y-\phi$ plane ($-1 \leq y \leq 1, -1 \leq \phi \leq 1$) of jets $x_i$ such that $\langle x_i-z,t \rangle$ is closest to $\lambda$ where $\lambda=-3$ to $3 \sigma$. In all plots, red denotes W jets and green is for QCD jets.}
	\label{fig:LDA}
\end{figure}

\paragraph{Classification}
In addition to LDA, we consider $k$-nearest neighbours (kNN) and kernel support vector machines (SVM) as two simple supervised classifiers. For kNN, we test $k$ (number of neighbours to consider) in $[10, 200]$ with an increment of 10. For SVM we use the radial basis function kernel $\exp[-\gamma d(x, x')^2]$ and test $11 \times 11 = 121$ pairs of $C$ and $\gamma$ in $[10^{-5}, 10^5]^2$ where $C$ regulates the strength of the penalty term. A training set of 5000 jets, a validation set of 2500 jets and a test set of 2500 jets are used for the two models in order to tune the hyperparameter(s). For LDA, the training and validation sets are instead merged. Details of the classifiers are expounded in \cite[Section 4]{Cai_2020}. To evaluate model performance, the AUC score (in $[0, 1]$) is given, where 1 indicates a perfect classifier and 0.5 corresponds to random guessing. Table~\ref{tab:NumericsClassifierResults} summarizes the results, including true positive rate (TPR) and false positive rate (FPR), for various $\kappa$ values. The approximate run time on Google Colab is also given to demonstrate the practicality of the linear framework.

We see that AUC peaks when $\kappa=0.1$ for LDA and $0.5$ for both kNN and SVM, with a relative increase (with respect to the linear Wasserstein baseline) in performance of 10.2\% for LDA, 3.4\% for kNN, and 1.7\% for SVM. The gains seem small but are indeed rather significant in the context of jet tagging, given that the AUC increase of using a large neural network over the optimal transport approach is only about 2\% \cite{ColliderMetric2019}.
In our experiments the gain of $\HK$ over $\W_2$ is stronger on the relatively simplistic classifiers kNN and LDA and not as pronounced on the more sophisticated SVM classifier. This suggest that the $\W_2$ representation does contain almost as much information about the jet class as the $\HK$ representations and sufficiently complex classifiers can extract it. In light of interpretability of classification results it may still be preferable to choose a representation where also simpler methods work well.
Compared to $\W_2$, the $\HK$ embedding requires the tuning of the parameter $\kappa$. The additional (computational) complexity of this step is however manageable. Based on Remark \ref{rem:HKLengthScale} and the discussions in this section a rough estimate for $\kappa$ can be obtained from physical intuition. This estimate can subsequently be refined by cross validation. Since Table \ref{tab:NumericsClassifierResults} indicates that the classification behaviour is relatively robust with respect to $\kappa$ over almost one order of magnitude, a coarse cross validation parameter search is sufficient.

In conclusion, this suggests that allowing mass to be generated and annihilated rather than only transported can have a positive effect on picking out W jets from QCD background. Further analysis still needs to be done to better understand the physical significance of $\HK$ with different $\kappa$ values and the relation of the optimal $\kappa$ with other physical length scales intrinsic to the problem. This is a topic under our current investigation. 

\begin{table}[hbt]
	\centering
	\caption{Results for the W vs.~QCD jet tagging task using LDA, kNN and SVM on the (unbalanced) linearized OT embeddings for various length scale parameters $\kappa$ ($\kappa=+\infty$ denotes balanced $\W_2$).}
	\label{tab:NumericsClassifierResults}
		{\scriptsize
		\begin{tabular}{|c|c|c|c|c|c|c|c|c|c|c|c|c|}
			\hline
			\multicolumn{2}{|c|}{length scale $\kappa$} & $+\infty$ & 100 & 10 & 5 & 1 & 0.7 &
				0.5 & 0.3 & 0.1 & 0.05 & 0.01\\
			\hline \hline
			\multirow{4}{*}{\textbf{LDA}} & \textbf{AUC} & \textbf{0.694} & \textbf{0.733} & \textbf{0.746} & \textbf{0.747} & \textbf{0.752} & \textbf{0.751} & \textbf{0.748} & \textbf{0.760} & \textbf{0.765} & \textbf{0.763} & \textbf{0.642}\\
			& TPR & 0.684 & 0.684 & 0.703 & 0.721 & 0.724 & 0.740 & 0.736 & 0.692 & 0.704 & 0.731 & 0.770\\
			& FPR & 0.296 & 0.218 & 0.211 & 0.226 & 0.220 & 0.239 & 0.239 & 0.171 & 0.174 & 0.205 & 0.486\\
			\cline{2-13}
			& run time & \multicolumn{11}{|c|}{several seconds}\\
			\hline
			\multirow{5}{*}{\textbf{kNN}} & \textbf{AUC} & \textbf{0.821} & \textbf{0.818} & \textbf{0.819} & \textbf{0.818} & \textbf{0.829} & \textbf{0.841} & \textbf{0.849} & \textbf{0.847} & \textbf{0.821} & \textbf{0.772} & \textbf{0.671} \\
			& TPR & 0.771 & 0.763 & 0.768 & 0.763 & 0.760 & 0.791 & 0.798 & 0.809 & 0.821 & 0.783 & 0.733\\
			& FPR & 0.128 & 0.127 & 0.130 & 0.126 & 0.102 & 0.110 & 0.100 & 0.114 & 0.181 & 0.238 & 0.390\\
			\cline{2-13}
			& hyperpar.~$k$ & 30 & 20 & 30 & 20 & 10 & 20 & 10 & 20 & 10 & 10 & 30 \\
			\cline{2-13}
			& run time & \multicolumn{11}{|c|}{1.5 hours}\\
			\hline
			\multirow{6}{*}{\textbf{SVM}} & \textbf{AUC} & \textbf{0.842} & \textbf{0.842} & \textbf{0.842} & \textbf{0.841} & \textbf{0.849} & \textbf{0.851} & \textbf{0.856} & \textbf{0.853} & \textbf{0.845} & \textbf{0.806} & \textbf{0.694} \\
			& TPR & 0.817 & 0.819 & 0.817 & 0.819 & 0.823 & 0.829 & 0.832 & 0.829 & 0.788 & 0.741 & 0.787 \\
			& FPR & 0.133 & 0.134 & 0.134 & 0.137 & 0.126 & 0.127 & 0.120 & 0.124 & 0.099 & 0.128 & 0.401 \\
			\cline{2-13}
			& hyperpar.~$C$ & 1 & 1 & 1 & 1 & 1 & 1 & 1 & 1 & 1 & 10 & 10 \\
			& hyperpar.~$\gamma$ & 100 & 100 & 100 & 100 & 100 & 100 & 100 & 100 & 1000 & 1000 & 100000\\
			\cline{2-13}
			& run time & \multicolumn{11}{|c|}{5 hours}\\
			\hline
		\end{tabular}
		}
\end{table}

We repeat the analysis (without $\kappa=0.7, 0.3$) on two additional datasets, each again containing 10k W and QCD jets simulated in the same way. We thus gain a quantitative, though rough, understanding on the fluctuation of AUC: the deviation is no larger than $10\%$ and the general trend is the same for all datasets, see Figure~\ref{fig:QCDref_AUC}. Moreover, we investigated the impact of the reference measure on classification performance by choosing a different reference which is the linear mean of all QCD jets in the dataset after rasterization to a grid, shown in the upper row of Figure~\ref{fig:QCDref_AUC}. We then analyze the three datasets using their respective QCD references, the AUC curves are shown in Figure~\ref{fig:QCDref_AUC}. We observe that in general an adapted choice of the reference measure slightly improves model performance in the best $\kappa$ range, yet the performance deteriorates faster when $\kappa \to 0.01$ compared to the uniform measure. On the whole, the classification performance using either reference measure is comparable to each other. 

\begin{figure}[hbt]
	\centering
	\includegraphics[width=0.6\textwidth]{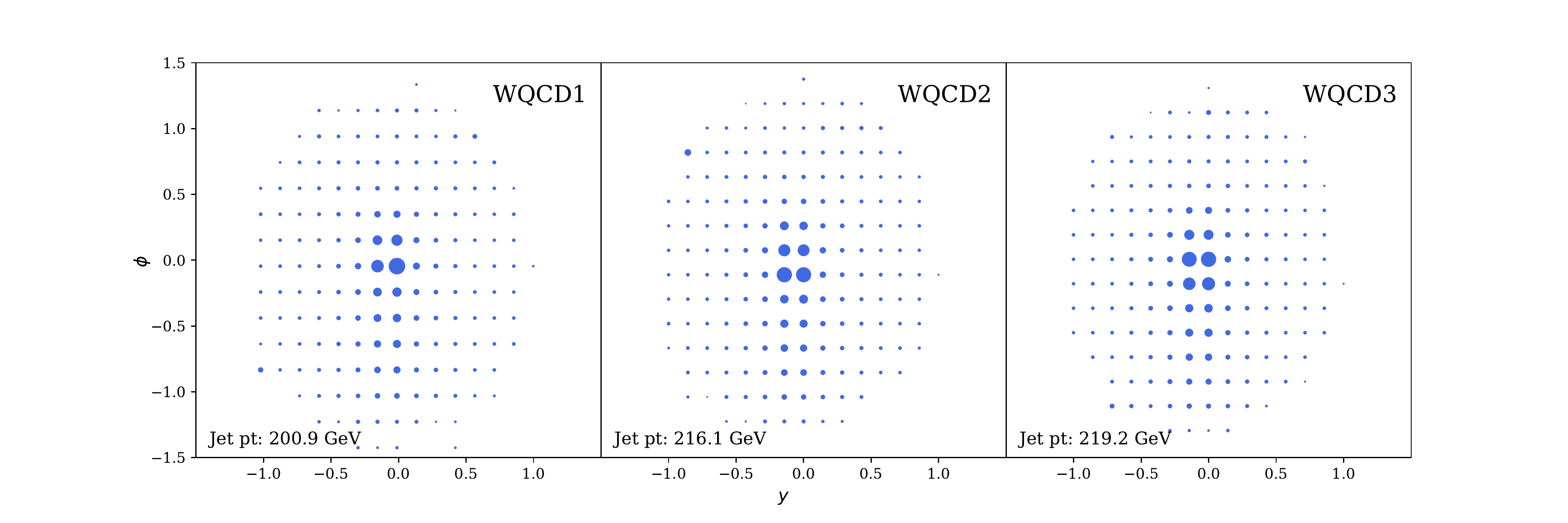}
	\includegraphics[width=0.95\textwidth]{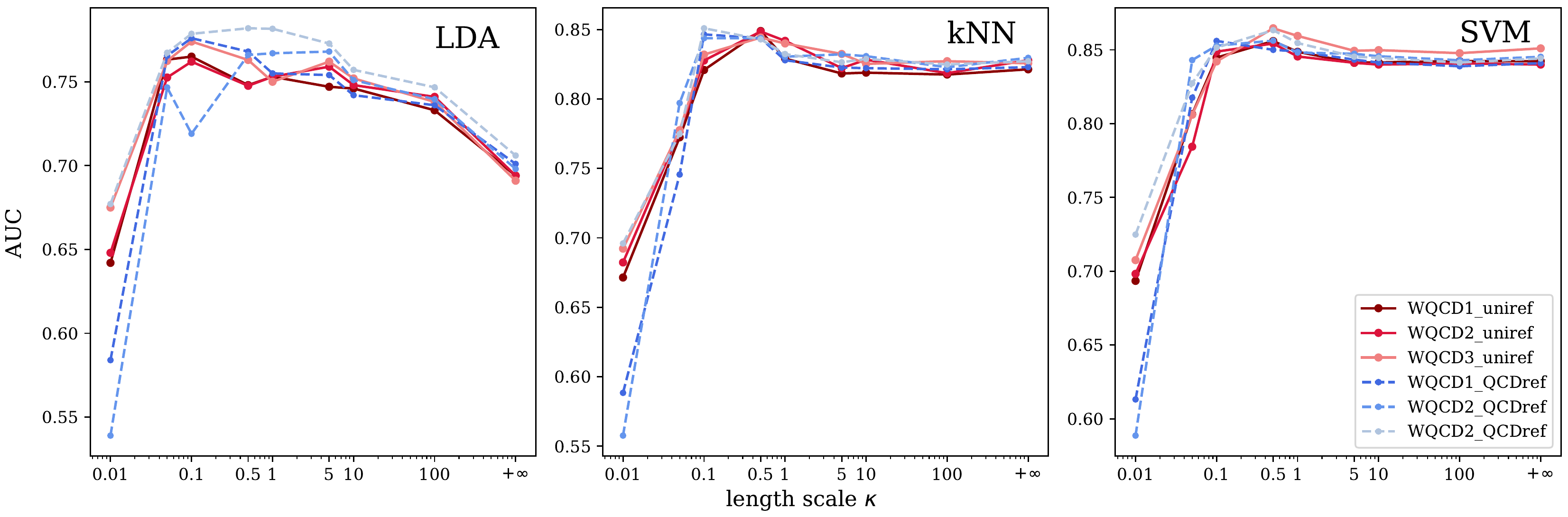}
	\caption{Upper: QCD reference jets in the $y-\phi$ plane for the three W vs.~QCD datasets, obtained by averaging all QCD jets in the respective dataset. Note that WQCD1 is the same dataset used in Table~\ref{tab:NumericsClassifierResults}. Lower: AUC scores for LDA, kNN and SVM on the three datasets with $\W2$ and $\HK$ metrics of various $\kappa$'s. Red solid lines show the results using the uniform reference measure, whereas blue dashed lines are obtained using the QCD reference measures.}
	\label{fig:QCDref_AUC}
\end{figure}


\section{Conclusion and outlook}
In this article we have studied the linearization of the Hellinger--Kantorovich metric by locally approximating its weak Riemannian structure with its tangent space. In particular, we have introduced explicit forms of logarithmic and exponential map and shown that they behave as expected along geodesics that connect samples with the reference point. We outlined how this can be leveraged numerically in data analysis problems and demonstrated that it compares favourably to the linearized Wasserstein-2 distance while we argue that from a practical algorithmic perspective the additional complexity is low.
Loosely speaking, one must apply an unbalanced version of the Sinkhorn algorithm, adjust the formula for the initial velocity field, accommodate an additional scalar mass change field, and finally fix a single real-valued length-scale parameter by validation on the data.

For the benefit of concrete numerical examples we have postponed additional analytic questions to future work. These include the study of the precise range of the logarithmic map, its discontinuity behaviour, a more careful look at the singular measure-valued component, and the domain of the exponential map.
In particular, the linearization in the Wasserstein space is closed under convex combinations which means that, if $\{v_i\}_{i=1}^n$ are a set of Wasserstein linear embeddings, any $\tilde{v}$ in the convex hull of $\{v_i\}_{i=1}^n$ is in the domain of the exponential map and hence one can generate a new measure via $\tilde{\mu}=\Exp_{\W}(\tilde{v})$ (see~\cite{park18} for applications to data augmentation).
Identifying operations under which the Hellinger--Kantorovich linearisation is closed is left open for future work.

Another open problem is to quantitatively bound the accuracy of the linear approximation of the Hellinger--Kantorovich distance.
Answering this question requires one to estimate the curvature of the Hellinger--Kantorovich space.
In fact, this question is still largely open for the Wasserstein space, however recent work~\cite{moosmuller21} has provided a bound in the Wasserstein linearisation with respect to certain perturbations.
Instead of attempting to quantify $\W_2(\mu_1,\mu_2)\approx \WLin(\mu_0;\mu_1,\mu_2)$ one can instead seek upper and lower bounds of $\WLin(\mu_0;\mu_1,\mu_2)$ in terms of $\W_2(\mu_1,\mu_2)$ and bi-H\"older bounds of this form have appeared in~\cite{delalande21,gigli11,merigot20}.

Stability of the Wasserstein optimal transport maps with respect to discretisation (and therefore the stability of the linear Wasserstein distance with respect to discretisation) was established in~\cite{berman20}.
Extending this to the Hellinger--Kantorovich distance is another possible avenue for future work.

On the practical side we seek to further study the application to collider physics and other examples.
\vspace{\baselineskip}

\paragraph{Acknowledgements}
The work of TC was supported in part by the Department of Energy under Grant No. DE-SC0011702. TC and JC would like to thank Nathaniel Craig and Katy Craig for insightful conversations about the application on jet physics and comments on the manuscript.
BS was supported by the Emmy-Noether Programme of the Deutsche Forschungsgemeinschaft (DFG). MT is grateful for the support of the Cantab Capital Institute for the Mathematics of Information and Cambridge Image Analysis at the University of Cambridge, and has received funding from the European Research Council under the European Union's Horizon 2020 research and innovation programme grant agreement No 777826 (NoMADS) and grant agreement No 647812.  


\bibliography{references}{}
\bibliographystyle{siamplain}

\end{document}